\documentclass[11pt]{article}

\usepackage{amsmath}
\usepackage{righttag}
\usepackage{natbib}
\usepackage{setspace}
\usepackage{comment,soul}
\usepackage{amsmath, amsthm, amsfonts, amssymb, algorithm}
\usepackage{graphicx,psfrag,epsf,color}
\usepackage{enumerate}
\usepackage{url,subcaption,float}
\usepackage[normalem]{ulem} 
\usepackage{pbox}
\usepackage{etoolbox}
\usepackage{enumitem}

\newtheorem{prop}{Proposition}

\newlength{\defbaselineskip}
\newcommand{\specialcell}[2][c]{%
  \begin{tabular}[#1]{@{}c@{}}#2\end{tabular}}
\newcommand{\sbf}{\boldsymbol}
\newcommand{\mbf}{\mathbf}

\newcommand{\vv}{\mbox{Var}}
\newcommand{\E}{\mbox{E}}
\newtheorem{thm}{Theorem}
\newtheorem{lem}{Lemma}
\theoremstyle{remark}

\AtBeginEnvironment{rem}{\small}

\begin{document}

\title{Asymptotic Analysis of Sampling Estimators for Randomized Numerical Linear Algebra Algorithms%
\footnote{A short preliminary conference version of this paper has appeared as~\citet{asymptotic_RandNLA_estimators_CONF}.}
}

\author{
Ping Ma%
\thanks{
Department of Statistics, 
The University of Georgia.
Email: \texttt{pingma@uga.edu}.
}
\and
Xinlian Zhang%
\thanks{
Department of Family Medicine and Public Health,
University of California at San Diego.
Email: \texttt{xizhang@health.ucsd.edu}.
}
\and
Xin Xing%
\thanks{
Department of Statistics, 
Harvard University.
Email: \texttt{xin\_xing@fas.harvard.edu}.
}
\and
Jingyi Ma%
\thanks{
Department of Statistics and Mathematics,
Central University of Finance and Economics.
Email: \texttt{jingyima@cufe.edu}.
}
\and
Michael W.~Mahoney%
\thanks{
International Computer Science Institute
and
Department of Statistics,
University of California at Berkeley.
Email: \texttt{mmahoney@stat.berkeley.edu}.
}
}

\date{}
\maketitle

\begin{abstract}
\noindent
The statistical analysis of Randomized Numerical Linear Algebra (RandNLA) algorithms within the past few years has mostly focused on their performance as point estimators. 
However, this is insufficient for conducting statistical inference, e.g., constructing confidence intervals and hypothesis testing, since the distribution of the estimator is lacking. 
In this article, we develop an asymptotic analysis to derive the distribution of RandNLA sampling estimators for the least-squares problem. 
In particular, we derive the asymptotic distribution of a general sampling estimator with arbitrary sampling probabilities. 
The analysis is conducted in two complementary settings, i.e., when the objective of interest is to approximate the full sample estimator or is to infer the underlying ground truth model parameters. 
For each setting, we show that the sampling estimator is asymptotically normally distributed under mild regularity conditions.
Moreover, the sampling estimator is asymptotically unbiased in both settings. 
Based on our asymptotic analysis, we use two criteria, the Asymptotic Mean Squared Error (AMSE) and the Expected Asymptotic Mean Squared Error (EAMSE), to identify optimal sampling probabilities. 
Several of these optimal sampling probability distributions are new to the literature, e.g., the root leverage sampling estimator and the predictor length sampling estimator. 
Our theoretical results clarify the role of leverage in the sampling process, and our empirical results demonstrate improvements over existing~methods.
\end{abstract}

\noindent
{\it Keywords:} 
least squares, 
randomized numerical linear algebra, 
leverage scores, 
asymptotic distribution, 
mean squared error, 
asymptotic mean squared error

\section{Introduction}
Recent work in Randomized Numerical Linear Algebra (RandNLA) focuses on using random sketches of the input data in order to construct approximate solutions more quickly than with traditional deterministic algorithms. 
In this article, we consider \emph{statistical} aspects of recently-developed fast RandNLA algorithms for the least-squares (LS) linear regression problem.
Given 
$\mbf{Y}=(Y_1, \ldots, Y_n)^{T}\in \mathbb{R}^n$ and 
$\mbf{X}=(\mbf{x}_1, \ldots, \mbf{x}_n)^{T}\in \mathbb{R}^{n\times p}$, we consider the~model
\begin{equation}
\label{linreg-matrix}
\mbf{Y}=\mbf{X}\sbf{\beta}_0+\sbf{\varepsilon},
\end{equation}
where $\sbf{\beta}_0\in\mathbb{R}^p$ is the coefficient vector, 
and  $\sbf{\varepsilon}=(\varepsilon_1, \ldots, \varepsilon_n)^{T}\in \mathbb{R}^n$, where $\varepsilon_i$s are 
i.i.d random errors with mean 0 and variance $\sigma^2<\infty$. 
We assume the sample size $n$ is large and that $\mbf{X}$ has full column rank. 
The ordinary least squares (OLS) estimator
of $\sbf{\beta}_0$ is 
\begin{equation}\label{olseqn}
\hat{\sbf{\beta}}_{OLS}=\arg\min\limits_{\sbf{\beta}}\|\mbf{Y}-\mbf{X}\sbf{\beta}\|^2=(\mbf{X}^T\mbf{X})^{-1}\mbf{X}^T\mbf{Y}   ,
\end{equation}
where $\|\cdot\|$ is the Euclidean norm. 
While the OLS estimate is optimal in several senses, the algorithmic complexity for computing 
it  
with direct methods is $O(np^2)$, which can be daunting when $n$ and/or $p$ are~large. 

Motivated by these algorithmic considerations, randomized sketching methods have been developed within RandNLA to achieve improved computational efficiency \citep{Mah-mat-revBOOK,Drineas:Mahoney:2016:RRN:ACM,HMT09_SIREV,woodruff2014sketching,MD16_chapter,RandNLA_PCMIchapter_chapter}. 
With these methods, one takes a (usually nonuniform) random sample of the full data (perhaps after preprocessing or preconditioning with a random projection matrix~\citep{Drineas:Mahoney:2016:RRN:ACM}), and then the sample is retained as a surrogate for the full data for subsequent computation. 
Here is an example of this approach for the LS problem. 

\begin{itemize}[leftmargin=*]
\item
\textbf{Step 1: Sampling.}
Draw a random sample of size $r\ll n$ with replacement from the full data using probabilities $\{\pi_i\}_{i=1}^{n}$. 
Denote the resulting sample and probabilities as $(\mbf{X}^{*}, \mbf{Y}^{*})$ and $\{\pi_i^{*}\}_{i=1}^{r}$.\\ 
\item
\textbf{Step 2: Estimation.} Calculate the  weighted LS solution, using the random sample, by solving \\
\begin{eqnarray}
\nonumber
\tilde{\sbf{\beta}}
   &=&\text{arg min}_{\sbf{\beta} }\|\mbf{\Phi}^*\mbf{Y}^{*}-\mbf{\Phi}^*\mbf{X}^{*}\sbf{\beta}\|^2 \\
\label{wlsq-sample}
   &=& (\mbf{X}^{*T}\mbf{\Phi}^{*2}\mbf{X}^{*})^{-1}\mbf{X}^{*T}\mbf{\Phi}^{*2}\mbf{Y}^{*}
\end{eqnarray}
where $\mbf{\Phi}^*=\text{diag}(1/\sqrt{r\pi_i^{*}})$. 
\end{itemize}

\noindent
Popular RandNLA sampling approaches include the uniform sampling estimator (UNIF), the basic leverage-based sampling estimator (BLEV), where $\pi_i^{BLEV}= h_{ii}/\sum_{i=1}^{n}h_{ii} $, where $h_{ii} = \mbf{x}_i^T(\mbf{X}^T\mbf{X})^{-1}\mbf{x}_i$ 
are the leverage scores of $\mbf{X}$, and the shrinkage leverage estimator (SLEV), which involves sampling probabilities $\pi_i^{SLEV}=\lambda h_{ii}/\sum_{i=1}^{n}h_{ii}+(1-\lambda)/n$, where $\lambda \in (0,1)$~\citep{Drineas:06,DMM08,DMMW12_JMLR,Ma:13}. 

In this article,  we study the \emph{statistical} properties of these and other estimators.
Substantial evidence has shown the practical effectiveness of core RandNLA methods~\citep{Ma:13,Ma:15,Drineas:Mahoney:2016:RRN:ACM} (as well as other randomized approximating methods, including the Hessian sketch~\citep{wang2017sketching,JMLR:v17:14-460} and iterative/divide-and-conquer methods~\citep{AMT10,meng2014lsrn}) in providing point estimators.  
However, this is not sufficient for statistical analysis since the uncertainty of the estimator is lacking. 
In statistics, uncertainty assessment can be conducted through confidence interval construction and significance testing. 
It is well-known that the construction of confidence intervals and significance testing are interrelated with each other \citep{lehmann2006testing}. 
Performing these two analyses is more difficult than point estimation, since it requires the distributional results of the estimator, rather than just moment conditions or concentration bounds. 
In the RandNLA literature, distribution results of estimators are still lacking.

There are two main challenges in studying the statistical and distributional properties of RandNLA algorithms. 
The first challenge is that there are two sources of randomness contributing to the statistical performance of RandNLA sampling estimators: one source is the random errors in the model, i.e., the $\varepsilon_i$s, which are typically attributed to measurement error or random noise inherited by $\mbf{Y}$; and the other source is the randomness in the random sampling procedure within the approximation algorithm.
The second challenge is that these two sources of randomness couple together within the estimator in a nontrivial way. 
More formally, the sampling estimator can be expressed as 
$\tilde{\sbf\beta}=(\mbf{X}^T\mbf{W}\mbf{X})^{-1}\mbf{X}^T\mbf{W}\mbf{Y}$, where $\mbf W$ is a random diagonal matrix, with the $i^{th}$ diagonal element being related to the probability of choosing the $i^{th}$ sample. 
The random variable used to denote the random sampling procedure, i.e., $\mbf W$, is involved in the sampling estimator in a nonlinear fashion, and it pre-multiplies $\mbf{Y}$, which contains randomness from the~$\varepsilon_i$s.

We address these challenges to studying the asymptotic distribution of a general RandNLA sampling estimator for LS problems. 
Our results are fundamentally different from previous results on the statistical properties of RandNLA algorithms (e.g.,~\cite{Ma:13,Ma:15,raskutti:2015,chen2016statistical,2016arXiv160102068W,DCMW19_TR}), in that we provide asymptotic distribution analysis, rather than finite-sample concentration inequalities. 
The resulting asymptotic distributions open the possibility of performing statistical inference tasks such as hypothesis testing and constructing confidence intervals, whereas finite sample concentration inequality results may not. 
It is worth mentioning that the results of asymptotic analysis are usually practically valid as long as the sample size is only moderately large.

\subsection{Main Results and Contributions}

We study the asymptotic distribution of a general RandNLA sampling estimator for the LS linear regression problem, from both a theoretical and empirical perspective. 

\paragraph{Main Theoretical Results.}
Our main theoretical contribution is to derive the asymptotic distribution of RandNLA estimators in two complementary settings.

{\bf Data are a random sample.} 
We first consider the data as a random sample from a population, in which case the goal is to estimate the parameters of the population model.
In this case, for this \emph{unconditional inference}, we establish the asymptotic normality, i.e., deriving the asymptotic distribution, of sampling estimators for the linear model under general regularity conditions.
We show that sampling estimators are asymptotically unbiased estimators with respect to the true model coefficients,
and we obtain an explicit form for the asymptotic variance, for both fixed number of predictors (Theorem~\ref{thm:beta_tilde-beta}) and diverging number of predictors (Theorem~\ref{coro:diverg:p}). 
{\bf Sampling Estimators.} 
Using these distributional results, we propose several efficient and asymptotically optimal estimators. 
Depending on the quantity of interest (e.g., $\sbf\beta_0$ versus some linear function of $\sbf\beta_0$ such as $\mbf{Y}=\mbf X\sbf\beta_0$ or $\mbf{X}^T\mbf X\sbf\beta_0$), we obtain different optimal sampling probabilities (Propositions~\ref{thm:beta:pi1}, \ref{thm:beta:pi3}, and~\ref{thm:beta:pi2}) that lead to sampling estimators that minimize the Asymptotic Mean Squared Error (AMSE) in the respective context. 
None of these distributions is proportional to the leverage scores, but one (RL of Proposition~\ref{thm:beta:pi3}) is constructed using the square roots of the leverage scores, and another (PL of Proposition~\ref{thm:beta:pi2}) is constructed using the row norms of the predictor~matrix.

{\bf Data are given and fixed.} 
We then consider the data as given/fixed, in which case the goal is to approximate the full sample OLS estimate. 
In this case, for this \emph{conditional inference}, we establish the asymptotic normality, i.e., deriving the asymptotic distribution, of sampling estimators for the linear model under general regularity conditions.
We show that sampling estimators are asymptotically unbiased with respect to the OLS estimate, and we obtain an explicit form of the asymptotic variance and the Expected Asymptotic Mean Squared Error (EAMSE) of sampling estimators (Theorem~\ref{thm:beta_tilde-beta_hat}).
{\bf Sampling Estimators}. 
Using these results, we construct sampling probability distributions that lead to sampling estimators that minimize the EAMSE. 
Depending on the quantity of interest (here, $\hat{\sbf\beta}_{OLS}$ versus some linear function of $\hat{\sbf\beta}_{OLS}$ such as $\hat{\mbf{Y}}=\mbf X\hat{\sbf\beta}_{OLS}$ or $\mbf{X}^T\mbf X\hat{\sbf\beta}_{OLS}$, we obtain different optimal sampling probabilities (Propositions~\ref{thm:beta_hat:pi}, \ref{thm:beta_hat:pi3}, and~\ref{thm:beta_hat:pi2}).

\paragraph{Main Empirical Results.}

We conduct a comprehensive empirical evaluation of the performance of these sampling estimators, on both synthetic and real datasets. 
This involves both conditional and unconditional inference cases, using predictor matrices generated from various distributions, including heavy-tailed and asymmetric distributions. 
For all settings under consideration, we calculate the squared bias and variance of the sampling estimators.
We demonstrate that the squared bias decreases as sample size increases, and we demonstrate that the squared biases are typically much smaller than the variances. 
These observations are consistent with our theory stating that the sampling estimators are asymptotically unbiased. 
The variance of sampling estimators also decreases as sample size increases, indicating the consistency of the sampling estimators. 
Depending on the specific objective considered, we also demonstrate that the novel estimators we derive have better performance, e.g., smaller variances, than existing ones, confirming the optimality results established in this paper. 
Another goal of the simulation study is to evaluate the necessity of our regularity conditions for the theorems. 
In the case of the predictor matrix generated from the $t$-distribution with 1 degree of freedom, the regularity conditions of our theory are technically not satisfied. 
The estimators, however, are shown to have performance similar to those in the aforementioned settings.  
Also, on two real-world data examples, we show that all the observations concerning asymptotic unbiasedness and asymptotic consistency in simulated datasets also appear.
In particular, our proposed sampling methods for conditional inference have smaller variances, compared to other leverage-based estimators, such as BLEV/ALEV~\citep{DMM06,DMMW12_JMLR} and SLEV~\citep{Ma:13,Ma:15}.

\subsection{Related Work}

There is a large body of related work in RandNLA 
\citep{Mah-mat-revBOOK,Drineas:Mahoney:2016:RRN:ACM,HMT09_SIREV,woodruff2014sketching,MD16_chapter,RandNLA_PCMIchapter_chapter}. 
However, very little of this work addresses statistical aspects of the methods.
Recently, significant progress has been made in the study of the statistical properties of RandNLA sampling estimators 
\citep{Ma:13,Ma:15,raskutti:2015,chen2016statistical,2016arXiv160102068W,DCMW19_TR}.
The work most related to ours is that of~\cite{Ma:13,Ma:15}, who employed a Taylor series expansion up to a linear term to study the MSE of RandNLA sampling estimators. 
\cite{Ma:13,Ma:15} failed to characterize the detailed convergence performance of the remainder term. 
They concluded that neither leverage-based sampling (BLEV) nor uniform sampling (UNIF) dominates the other in terms of variance; and they proposed and demonstrated the superiority of the SLEV sampling method.
To find the sampling distribution of estimators, leading to statistically-better RandNLA sampling estimators, it is important to examine the convergence properties of the remainder term. 
To accomplish this, we consider the asymptotic distribution of the sampling estimator.  
Such asymptotic analysis is common in statistics, and it can substantially simplify the derivation of complicated random variables, leading to simpler analytic expressions \citep{lecam1986asymptotic}. 

\cite{chen2016statistical} proposed optimal estimators minimizing the variance that account for the randomness of sampling and model error.
Our results and those of \cite{chen2016statistical} have similar goals, but they are different. 
First, \cite{chen2016statistical} used bias and variance, while we use AMSE and EAMSE. 
Second, we consider the asymptotic distribution of the sampling estimators, going beyond just the bias and variance of \cite{chen2016statistical}. 
Thus, our results 
could be used for downstream statistical inferences, e.g., constructing confidence intervals and hypothesis testing, while those of \cite{chen2016statistical} could~not.
Third, the exact expression of optimal sampling probabilities in \cite{chen2016statistical} depends on the unknown true parameter of the model, $\mbf\beta_0$ and $\sigma^2$ (Eqn.~(4) in \cite{chen2016statistical}), while our optimal sampling probabilities (see Section~\ref{sec:weighted-estimate}) are readily computed from the data. 
Fourth, \cite{chen2016statistical} only studied properties of sampling estimators for estimating true model parameters, while we consider both estimating the true parameter and approximating the full sample estimate. 

\cite{2016arXiv160102068W} proposed 
an approximated A-optimality criterion, which 
is based on the conditional variance of the sampling estimator given a subsample.
Since the randomness of sampling is not considered in the criterion, they obtained a simple analytic expressions of the optimal results. 
\cite{DCMW19_TR} also consider experimental design from the RandNLA perspective, and they propose a framework for experimental design where the responses are produced by an arbitrary unknown distribution. 
Their main result yields nearly tight bounds for the classical A-optimality criterion, as well as improved bounds for worst-case responses. 
In addition, they propose a  minimax-optimality criterion (which can be viewed as an extension of both A-optimal design and RandNLA sampling for worst-case regression). 
Related works on the asymptotic properties of subsampling estimators in logistic regression can be found in \cite{wang2018optimal} and \cite{wang2019more}.

\subsection{Outline}

The remainder of this article is organized as follows.
In Section~\ref{sec:AMSE:defn}, we introduce the necessary technical notation and definitions of MSE, AMSE, and EAMSE. 
In Section~\ref{sec:weighted-estimate}, we derive the asymptotic distribution of the sampling estimator, and we propose several criteria which give rise to optimal sampling probability distributions.  
In Section~\ref{sec:simu-and-real}, we present empirical results on simulated data and two real-world data examples. 
In Section~\ref{sec:conclusion}, we provide a brief discussion and conclusion.
All technical proofs are presented in the Appendix. 
A short preliminary conference version of this paper has appeared as~\citet{asymptotic_RandNLA_estimators_CONF}.

\section{MSE, AMSE and EAMSE: Technical Definition}
\label{sec:AMSE:defn}

In this section, we review the well-known Mean Squared Error (MSE) criterion, and we also define and discuss the standard but less well-known Asymptotic Mean Squared Error (AMSE) and Expected Asymptotic Mean Squared Error (EAMSE) criteria.

Let $\sbf T_n$ be a $p\times 1$ estimator of a $p\times 1$ parameter $\sbf\nu$, for every $n$. 
One popular quality metric for the estimator $\sbf T_n$ is the MSE, which is defined to be
\begin{eqnarray}
MSE(\sbf T_n; \sbf \nu)
   &=& \nonumber \text{E}[ (\sbf T_n-\sbf\nu)^T(\sbf T_n-\sbf\nu) ]\\
   &=& \nonumber \text{tr}(\text{Var}(\sbf T_n))+(\text{E}(\sbf T_n)-\sbf\nu)^T(\text{E}(\sbf T_n)-\sbf\nu).  
\end{eqnarray}
The MSE can be decomposed into two terms: 
one term, $\text{tr}(\text{Var}(\sbf T_n))$, quantifying the \emph{variance} of the estimator; and 
one term, $(\text{E}(\sbf T_n)-\sbf\nu)^T(\text{E}(\sbf T_n)-\sbf\nu)$, quantifying the \emph{squared bias} of the estimator. 
To evaluate the RandNLA sampling estimator $\tilde{\sbf\beta}$ in estimating the true model parameter ${\sbf\beta}_{0}$ and the full sample OLS estimate $\hat{\sbf\beta}_{OLS}$, we will be interested in the AMSE and EAMSE, respectively. 
These are the asymptotic counterparts of MSE in large sample theory.

To define the AMSE, let $\sbf T_n$ be a $p\times 1$ estimator of a $p\times 1$ parameter $\sbf\nu$, for every $n$, and let ${\sbf\Sigma_n}$ be a sequence of $p\times p$ positive definite matrices. 
Assume $\sbf\Sigma_n^{-1/2}(\sbf T_n-\sbf \nu)\stackrel{d}{\rightarrow}\sbf Z$, where $\stackrel{d}{\rightarrow}$ denotes convergence in distribution, and assume $\sbf Z$ is a $p\times 1$ random vector such that its $i^{th}$ element $Z_i$ satisfies $0<\text{E} ( Z_i^2)<\infty$, for $i=1,\ldots, p$. 
Then, the AMSE of $\sbf T_n$, denoted $AMSE(\sbf T_n; \sbf \nu)$, is defined to be
\begin{eqnarray}
AMSE(\sbf T_n; \sbf \nu)
   &=& \nonumber \text{E}(\sbf Z^T\sbf\Sigma_n\sbf Z)  \\
   &=& \nonumber \text{tr}(\sbf\Sigma_n^{1/2}\text{Var}(\sbf Z)\sbf\Sigma_n^{1/2})+(\text{E}(\sbf Z)^T\sbf\Sigma_n\text{E}(\sbf Z))  \\
   &=&  \label{eqn:mse20} \text{tr}(\text{AVar}(\sbf T_n))+(\text{AE}(\sbf T_n)-\sbf\nu)^T(\text{AE}(\sbf T_n)-\sbf\nu), 
\end{eqnarray}
where $\text{AVar}(\sbf T_n)=\sbf\Sigma_n^{1/2}\text{Var}(\sbf Z)\sbf\Sigma_n^{1/2}$ and $\text{AE}(\sbf T_n)=\sbf\nu+\sbf\Sigma_n^{1/2}\text{E}(\sbf Z)$ denote the \emph{asymptotic variance-covariance matrix} and the \emph{asymptotic expectation} of $\sbf T_n$ in estimating $\sbf \nu$, respectively.  

To define the EAMSE, let $\sbf T_r$ be an $p\times 1$ estimator of a $p\times 1$ parameter $\sbf\nu_{\mbf Y}$, for every sample size $r$, and let $\sbf\Sigma_r$ be a sequence of $p\times p$ positive definite matrices. 
Assume that $\sbf\Sigma_r^{-1}(\sbf T_r-\sbf \nu_{\sbf Y})\stackrel{d}{\rightarrow}\sbf Z_{\sbf Y} $, and that $\sbf Z_{\sbf Y} $ is a $p\times 1$ random vector such that its $i^{th}$ element ${Z_{\mbf Y}}_i$ satisfies $0<\text{E} ( { Z_{\mbf Y}}_i^2)<\infty$, for $i=1,\ldots, p$. 
The EAMSE of $\sbf T_r$, denoted $EAMSE(\sbf T_r; \sbf\nu_{\mbf Y})$, is defined to be
\begin{eqnarray}
EAMSE(\sbf T_r; \sbf \nu_{\mbf Y})
   &=& \nonumber \text{E}_{\mbf Y} (\text{E}(\sbf Z_{\mbf Y} ^T\sbf\Sigma_r\sbf Z_{\mbf Y} ))  \\
  &=& \nonumber \text{E}_{\mbf Y}(\text{tr}(\sbf\Sigma_r^{1/2}\text{Var}(\sbf Z_{\mbf Y})\sbf\Sigma_r^{1/2})+\text{E}_{\mbf Y}(\text{E}(\sbf Z_{\mbf Y} )^T\sbf\Sigma_r\text{E}(\sbf Z_{\mbf Y} )))  \\
  &=&\label{eqn:eamse}\text{E}_{\sbf Y}(\text{tr}(\text{AVar}(\sbf T_r)))+\text{E}_{\sbf Y}(\text{AE}(\sbf T_r-\sbf \nu_Y)^T\text{AE}(\sbf T_r-\sbf \nu_Y)),
 \end{eqnarray}  
where $\text{AVar}(\sbf T_r)=\sbf\Sigma_r^{1/2}\text{Var}(\sbf Z_{\sbf Y})\sbf\Sigma_r^{1/2}$ and $\text{AE}(\sbf T_r)=\sbf\nu+\sbf\Sigma_r^{1/2}\text{E}(\sbf Z_{\sbf Y})$ denote the \emph{asymptotic variance-covariance matrix} and the \emph{asymptotic expectation} of $\sbf T_r$ in estimating $\sbf \nu_{\sbf Y}$, respectively. 

If $\text{E}(\sbf Z)=\sbf 0$, or $\text{E}(\sbf Z_{\sbf Y})=\sbf 0$, then we say $\sbf T_n$, or $\sbf T_r$, is an \emph{asymptotically unbiased estimator} of $\sbf\nu$ or $\sbf\nu_{\sbf Y}$, respectively. 
If $\text{tr}(\text{AVar}(\sbf T_n))\rightarrow0$ as $r\rightarrow\infty$, or $\text{E}_{\sbf Y}(\text{tr}(\text{AVar}(\sbf T_r)))\rightarrow0$ as $r\rightarrow\infty$, then we say $\sbf T_n$ or $\sbf T_r$ is an \emph{asymptotically consistent estimator}, respectively.

We may think of the EAMSE as the expectation of the AMSE. 
An important subtlety, however, in the use of the AMSE versus the use of the EAMSE lies in the limiting distribution. 
In unconditional inference (i.e., where we consider a statistical model, and where we will use the AMSE), the limiting distribution is $\sbf Z$, i.e., it does not involve the data $\mbf Y$; whereas, in conditional inference (i.e., where we consider the dataset $\mbf Y$ and sample size $n$ as fixed and given, and where we will use the EAMSE), the limiting distribution is $\sbf Z_{\mbf Y}$, i.e., it involves the data $\mbf Y$. 
In this paper, the basic estimator is denoted as $\tilde{\sbf\beta}$ (i.e., the counterpart for $\sbf T_n$ or $\sbf T_r$ in the definitions above will be $\tilde{\sbf\beta}$, or a linear function of $\tilde{\sbf\beta}$).
We will obtain the explicit form for both $\text{AVar}(\tilde{\sbf\beta})$ and $\text{AE}(\tilde{\sbf\beta})$ by deriving the large sample distributions of sampling estimators, when performing unconditional inference and conditional inference, in Section~\ref{subsec:est:beta} and Section~\ref{subsec:approx:betaols}, respectively. 
As we show in Section~\ref{subsec:approx:betaols}, the sequences of $\sbf \Sigma_r$ involve statistics based on the full sample. 
Thus, the motivation for taking the expectation of AMSE to construct EAMSE is to avoid calculating those full sample statistics in proposing the optimal RandNLA sampling estimators in conditional~inference.

\section{Sampling Estimation Methods}
\label{sec:weighted-estimate}

In this section, we derive asymptotic properties of the RandNLA sampling estimator $\tilde{\sbf{\beta}}$ under two scenarios: unconditional inference, which involves estimating the true model parameter $\sbf{\beta}_0$; and conditional inference, which involves approximating the full sample OLS estimator $\hat{\sbf{\beta}}_{OLS}$. 
We use the AMSE and EAMSE to develop two criteria for sampling estimators, and we obtain several optimal estimators.

\subsection{Unconditional Inference: Estimating Model Parameters}
\label{subsec:est:beta}

For Model (\ref{linreg-matrix}), from the traditional statistical perspective of using the data to perform inference, one major goal is to estimate the underlying true model parameters, i.e., $\sbf\beta_0$. 
We refer to this as \emph{unconditional inference}. 
For unconditional inference, both randomness in the data and randomness in the algorithm contribute to randomness in the RandNLA sampling estimators.

The following theorem states that, in unconditional inference, the asymptotic distribution of the sampling estimator $\tilde{\sbf\beta}$  is a normal distribution (with mean $\sbf\beta_0$ and variance $\sigma^2\mbf{\Sigma}_0$). 
The proof of Theorem~\ref{thm:beta_tilde-beta} is provided in Appendix~\ref{sec:proof:thm:beta_tilde-beta}. 

\begin{thm}[\textbf{Unconditional inference, fixed $p$}]
\label{thm:beta_tilde-beta}
Assume the number of predictors $p$ is fixed and the following regularity conditions hold.
\begin{itemize}
   \item
   \textit{(A1)[Data condition].} 
   There exist positive constants $b$ and $B$ such that 
   $ b\le \lambda_{min}\le\lambda_{max} \le B, $  
   where $\lambda_{min}$ and  $\lambda_{max}$  are the minimum and maximum eigenvalues of matrix $\mbf X^T\mbf X/n$, respectively. 
   \item
   \textit{(A2)[Sampling condition].}  The sample size $r=O(n^{1-\alpha})$, where $0\le\alpha<1$ and where the minimum sampling probability $\pi_{min} = O(n^{-\gamma_0})$, where $\gamma_0\ge1$.
   The parameters $\gamma_0$ and $\alpha$ satisfy $\gamma_0+\alpha<2$. 
\end{itemize} 
Under these assumptions, as the sample size $n\to\infty$, we have
\begin{eqnarray}
   (\sigma^2\mbf{\Sigma}_0)^{-\frac 12}(\sbf{\tilde\beta}-\sbf\beta_0)
   &\stackrel{d}{\rightarrow}&
   \label{eqn:thm:beta_tilde-beta}
   \textbf{N}(\sbf 0,\mbf{I}_p)
\end{eqnarray}
where 
$$
\mbf{\Sigma}_0 = 
   (\mbf X^T\mbf X)^{-1}
   \left( \mbf{X}^T(\mbf I_n+\mbf \Omega)\mbf{X} \right)
   (\mbf X^T\mbf X)^{-1}, \quad \mbf{\Omega} = \text{diag}\{1/r\pi_i\}_{i=1}^n,
$$ 
and $\mbf{I}_p$ is the $p\times p$ identity. 
Thus, for unconditional inference, the asymptotic mean of $\sbf{\tilde\beta}$~is
\begin{equation}
\label{eqn:thm:uncond:ae}
\text{AE}(\sbf{\tilde\beta})=\sbf\beta_0, 
\end{equation}
i.e., $\sbf{\tilde\beta}$ is an asymptotically unbiased estimator of $\sbf\beta_0$, and the asymptotic variance of $\sbf{\tilde\beta}$ is 
\begin{equation}\label{eqn:thm:uncond:avar}
AVar(\sbf{\tilde\beta})= \sigma^2\mbf{\Sigma}_0. 
\end{equation}
\end{thm}

\textbf{Remark.}
Theorem~\ref{thm:beta_tilde-beta} considers the case of a fixed parameter dimension $p$. 
The case of diverging parameter dimension $p\rightarrow\infty$ is considered in Theorem~\ref{coro:diverg:p} below.

\textbf{Remark.}
Theorem~\ref{thm:beta_tilde-beta} shows that, as the number of data points $n$ gets larger, the distribution of $\tilde{\sbf\beta}$ is well-approximated by a normal distribution, with mean $\sbf\beta_{0}$ and variance $\sigma^2\mbf{\Sigma}_0$. 

\textbf{Remark.}
Condition 
(A1)
in Theorem~\ref{thm:beta_tilde-beta} indicates that $\mbf X^T\mbf X/n$ is positive definite (as opposed to being just positive semi-definite). 
This condition requires the predictor matrix $\mbf X$ to be of full column rank and that the elements in $\mbf X$ are not over-dispersed. 
This condition ensures the consistency of the full sample OLS estimator~\citep{lai1978PNAS}, and it has been used in many related problems, e.g., variable selection~\citep{zou:adaptive:2006}. 

\textbf{Remark.}
Condition 
(A2) in Theorem~\ref{thm:beta_tilde-beta}, which can be rewritten as $n^{-\gamma_0} > n^{-(2-\alpha)}$, provides a lower bound on the \emph{smallest} sampling probability.
The smallest sampling probability cannot be too small, in the sense that it should be $O(n^{\alpha})$ away from $O(n^{-2})$. 
Bounding sampling probabilities from below mitigates the inflation of the variance $\mbf{\Sigma}_0$, which is proportional to the reciprocal sampling probability.
The importance of this condition for establishing \emph{statistical} properties of RandNLA algorithms was highlighted by~\citet{Ma:13,Ma:15}. 
Condition 
(A2)
can also be rewritten as $n^{1-\alpha}  n^{-\gamma_0} > n^{-1}$, which states that when the smallest sampling probability is very small, one compensates by making the sample size large.

\textbf{Remark.} 
In Theorem~\ref{thm:beta_tilde-beta}, the asymptotic variance $AVar(\sbf{\tilde\beta})$ can be written as 
\begin{equation}
\label{eqn:avar_thm1}
AVar(\sbf{\tilde\beta})=\sigma^2(\mbf{X}^T\mbf{X})^{-1}+ \sigma^2(\mbf{X}^T\mbf{X})^{-1}\mbf{X}^T\mbf{\Omega}\mbf{X}(\mbf{X}^T\mbf{X})^{-1},
\end{equation}
where the first term is the variance of the full sample OLS, and the second term is the variation related to the sampling process. 
The second term of Eqn.~(\ref{eqn:avar_thm1}) has a ``sandwich-type'' expression. 
The center term, $\mbf{\Omega}$, depends on the reciprocal sampling probabilities, suggesting that extremely small probabilities will result in large asymptotic variance and large AMSE of the corresponding estimator. 
This was observed previously in the non-asymptotic case by~\citet{Ma:15}.

{\bf Remark.} 
In light of efficient estimation methods such as iterative Hessian sketch and dual random projection, we emphasize that besides estimation, our distribution results can be used for performing additional inference analysis, e.g., constructing a confidence interval and conducting hypothesis testing. 
These inference analyses cannot be achieved by other iterative methods as far as we know.

Given Theorem~\ref{thm:beta_tilde-beta}, it is natural to ask whether there is an optimal estimator, i.e., one with the smallest AMSE for estimating $\sbf\beta_0$. 
Using the asymptotic results in Theorem~\ref{thm:beta_tilde-beta}, we propose the following three estimators. 

\paragraph{Estimating $\sbf\beta_0$.}

By Theorem~\ref{thm:beta_tilde-beta}, we could express the $AMSE(\tilde{\sbf\beta}, {\sbf\beta}_{0})$ as a function of $\{\pi_i\}_{i=1}^n$, as shown, e.g., in Eqn.~(\ref{eqn:mse2}) below. 
Since this expression is a function of the sampling probabilities, it is straightforward to employ the method of Lagrange multipliers to find the minimizer of the right-hand side of Eqn.~(\ref{eqn:mse2}), subject to the constraint $\sum_{i=1}^n\pi_i=1$.
The minimizer is then the optimal sampling probabilities for estimating ${\sbf\beta}_{0}$. 
The proof of Proposition~\ref{thm:beta:pi1} is provided in Appendix~\ref{sxn:proof_of_prop1}.

\begin{prop}
\label{thm:beta:pi1}
For the $AMSE(\tilde{\sbf\beta}, {\sbf\beta}_{0}) $, we have that 
\begin{eqnarray}\label{eqn:mse2}
AMSE(\tilde{\sbf\beta}, {\sbf\beta}_{0}) 
&= & \sigma^2\text{tr}\{(\mbf{X}^T\mbf{X})^{-1}\}+ \frac{1}{r}\sum_{i=1}^n \frac{\sigma^2}{\pi_i}||(\mbf{X}^T\mbf{X})^{-1}\mbf{x}_i||^2.
\end{eqnarray}
Given \eqref{eqn:mse2}, the sampling estimator with the sampling probabilities 
\begin{eqnarray}\label{eqn:pi_opt}
   \pi_i= \frac{\|(\mbf{X}^T\mbf{X})^{-1}\mbf{x}_i\|}{\sum_{i=1}^{n}\|(\mbf{X}^T\mbf{X})^{-1}\mbf{x}_i\|},
   \quad 
   i=1,\ldots,n, 
\end{eqnarray}
(which we call the inverse-covariance (IC) sampling estimator)
has the smallest 
$AMSE(\tilde{\sbf\beta}; {\sbf\beta}_{0})$.
\end{prop}

\textbf{Remark.} 
The implication of this optimal estimator is two-fold. 
On the one hand, as defined, the proposed IC estimator has the smallest AMSE.
On the other hand, if given the same tolerance of uncertainty, i.e., to achieve a certain small standard error, the IC estimator requires the smallest sample size. 

\textbf{Remark.}
Obviously, the IC sampling probabilities can be computed in $O(np^2)$ time, using standard methods.
More importantly, using the main Algorithm 1 in~\citet{DMMW12_JMLR}, they can be computed in $O(np\log(n)/\epsilon)$ time, where $\epsilon$ is the desired approximation error parameter.

\paragraph{Estimating linear functions of $\sbf\beta_0$.}

In addition to making inference on ${\sbf\beta}_{0}$, one may also be interested in linear functions of ${\sbf\beta}_{0}$, i.e., $\mbf L\sbf \beta_0$, where $\mbf L$ is any constant matrix of suitable dimension. 
Here, we present results for $\mbf X\sbf\beta_0$ and $\mbf{X}^T\mbf X\sbf\beta_0$ (although clearly similar results hold for other functions of the form $\mbf L\sbf \beta_0$). 

We start with estimating $\mbf{Y} = \mbf X\sbf\beta_0$ since, in regression analysis, inference on the true regression line $\mbf X\sbf\beta_0$ is crucially important. 
The proof of Proposition~\ref{thm:beta:pi3} (and other similar propositions below) is similar to that of Proposition~\ref{thm:beta:pi1}, and thus it is omitted.

\begin{prop}
\label{thm:beta:pi3}
For the $AMSE(\mbf X\tilde{\sbf\beta}, \mbf X{\sbf\beta}_{0})$, we have that
\begin{eqnarray}
\label{eqn:mse3}
AMSE(\mbf X\tilde{\sbf\beta}, \mbf X{\sbf\beta}_{0}) 
&= & p\sigma^2 + \frac{1}{r}\sum_{i=1}^n \frac{\sigma^2}{\pi_i}||\mbf X(\mbf{X}^T\mbf{X})^{-1}\mbf{x}_i||^2.
\end{eqnarray}
Given \eqref{eqn:mse3}, the sampling estimator with the sampling probabilities 
\begin{eqnarray}
   \pi_i&=&\label{eqn:pi_sq_hii} \frac{\|\mbf X(\mbf{X}^T\mbf{X})^{-1}\mbf{x}_i\|}{\sum_{i=1}^{n}\|\mbf X(\mbf{X}^T\mbf{X})^{-1}\mbf{x}_i\|}
   =\frac{\sqrt{h_{ii}}}{\sum_{i=1}^n \sqrt{h_{ii}}}, 
   \quad 
   i=1,\ldots, n, 
\end{eqnarray}
(which we call the root leverage (RL) sampling estimator)
has the smallest 
$AMSE(\mbf{X}\tilde{\sbf\beta}; \mbf{X}{\sbf\beta}_{0})$. 
\end{prop}

\textbf{Remark.}
Note that
\begin{equation}
\nonumber
\|\mbf X(\mbf{X}^T\mbf{X})^{-1}\mbf{x}_i\|^2 = (\mbf X(\mbf{X}^T\mbf{X})^{-1}\mbf{x}_i)^T\mbf X(\mbf{X}^T\mbf{X})^{-1}\mbf{x}_i=\mbf{x}_i^T(\mbf{X}^T\mbf{X})^{-1}\mbf{x}_i = h_{ii} .
\end{equation} 
These quantities, the so-called leverage scores (called BLEV, in~\citet{Ma:13,Ma:15}), have been central to RandNLA theory~\citep{Mah-mat-revBOOK,DMMW12_JMLR,Drineas:Mahoney:2016:RRN:ACM,MD16_chapter}.
Using the main Algorithm 1 in~\citet{DMMW12_JMLR}, they can be computed in $O(np\log(n)/\epsilon)$ time, where $\epsilon$ is the desired approximation error parameter.

\textbf{Remark.}
The probabilities in RL are a nonlinear transformation of the probabilities in BLEV. Comparing to the BLEV estimator, the RL estimator shrinks the large probabilities and pulls up the small probabilities. 
Thus, we expect RL to provide an estimator with smaller variances, in a way similar to SLEV.

\textbf{Remark.}
\cite{chen2016statistical} proposed optimal sampling estimators for estimating $\sbf\beta_0$ and predicting $\mbf Y$. Their sampling probabilities depend on the unknown parameters, and they proposed the probabilities in (\ref{eqn:pi_sq_hii}) as a rough approximation of their proposed probabilities without demonstration.  

We next consider estimating $\mbf{X}^T\mbf X\sbf\beta_0$, which is also of interest in regression analysis. 

\begin{prop}
\label{thm:beta:pi2}
For the $AMSE(\mbf{X}^T\mbf{X}\tilde{\sbf\beta}, \mbf{X}^T\mbf{X}{\sbf\beta}_{0})$, we have that 
\begin{eqnarray}\label{eqn:mse4}
AMSE(\mbf{X}^T\mbf{X}\tilde{\sbf\beta}, \mbf{X}^T\mbf{X}{\sbf\beta}_{0})
& = & \sigma^2\text{tr}( \mbf{X}^T\mbf{X})+ \frac{\sigma^2}{r}\sum_{i=1}^n \frac{1}{\pi_i}||\mbf{x}_i||^2.
\end{eqnarray}
Given \eqref{eqn:mse4}, the sampling estimator with the sampling probabilities 
\begin{eqnarray}\label{eqn:pi_pl}
   \pi_i=\frac{ \|\mbf x_i\|}{\sum_{i=1}^{n} \|\mbf x_i\|}, 
   \quad 
   i=1,\ldots, n, 
\end{eqnarray}
(which we call the predictor-length (PL) sampling estimator)
has the smallest 
value for the  
$AMSE(\mbf{X}^T\mbf{X}\tilde{\sbf\beta}; \mbf{X}^T\mbf{X}{\sbf\beta}_{0})$.
\end{prop}

\textbf{Remark.}
The PL probabilities have a connection with the Fisher information of the full sample OLS estimate. 
The Fisher information measures the ``amount of information'' about the parameter that is present in the data~(see Section 11.10 of \citet{Cover:2006:EIT:1146355}). The inverse of the Fisher information matrix gives a lower bound (the Cramer-Rao lower bound) on the variance of any estimator constructed from the data to estimate a parameter~(see Section 3.1.3 of \citet{shao2003mathematical}). 
Since the Fisher information of the full data can be written as the summation of the Fisher information of each data point, i.e., 
$\frac{1}{\sigma^2}\mbf{X}^T\mbf{X}= \frac{1}{\sigma^2}\sum_{i=1}^{n}\mbf{x}_i \mbf{x}_i^{T},$
we have that
$\text{tr}\{\frac{1}{\sigma^2}\mbf{X}^T\mbf{X}\}= \frac{1}{\sigma^2}\sum_{i=1}^{n}||\mbf{x}_i ||^2.$
The PL probability is high if the data point has a high contribution to the Fisher~information.

\paragraph{Diverging number of predictors, $p\rightarrow\infty$.}

Theorem~\ref{thm:beta_tilde-beta} considers the number of predictors/features, $p$, as fixed. 
It is also of interest to study the asymptotic properties of RandNLA estimators in the scenario that $p$ diverges with $n\rightarrow\infty$ (at a suitable rate relative to $n$).
The following theorem states our results concerning this case.
Observe that, in the case of a divergent $p$, the vector $(\sbf{\tilde\beta}-\sbf\beta_0)$ is of divergent dimension.
Thus, we characterize its asymptotic distribution via the scalar $\sbf a^T(\sbf{\tilde\beta}-\sbf\beta_0)$, where $\sbf a$ is an arbitrary bounded-norm~vector. 
The proof of Theorem~\ref{coro:diverg:p} is provided in Appendix~\ref{sec:proof-of-corollaryThm}. 

\begin{thm}[\textbf{Unconditional inference, diverging~$p$}]
\label{coro:diverg:p}
In addition to Condition (A1) in Theorem~\ref{thm:beta_tilde-beta}, assume the following regularity conditions hold.
\begin{itemize}
    \item 
    \textit{(B1)[Data condition].} The number of predictors $p$ diverges at a rate $p=n^{1-\kappa}$, $0<\kappa<1$; 
    and $\frac{\max_{i}\|\mbf x_i\|^2}{n}=O(\frac pn)$, where  $\mbf x_i$ is the $i^{th}$ row of $\mbf X$.
    \item
    \textit{(B2)[Sampling condition]:} 
    The parameters $\alpha$, $\gamma_0$, and $\kappa$ satisfy
    $\alpha+\gamma_0-\kappa<1$.
\end{itemize}
Under these assumptions, as the sample size  $n\to\infty$, we have
\begin{eqnarray}
   (\sigma^2\sbf a^T\mbf{\Sigma}_0^\prime\sbf a)^{-\frac{1}{2}}\sbf a^T(\sbf{\tilde\beta}-\sbf\beta_0)&\stackrel{d}{\rightarrow}&\textbf{N}(0,1)  ,
\end{eqnarray}
where 
$\sbf a\in\mathbb{R}^p$ is any finite-norm vector, i.e, $\|\sbf a\|^2<\infty$. 
\end{thm}

\textbf{Remark.}
Condition (B2) is more stringent than Condition (A2), and this is required for accommodating a divergent $p$. 

\textbf{Remark.}
It is easy to verify that the sampling estimators in Propositions~\ref{thm:beta:pi1}, \ref{thm:beta:pi3}, and~\ref{thm:beta:pi2} are still the optimal sampling estimators for their respective purposes.
Thus, we omit restating the~results.

\subsection{Conditional Inference: Approximating the Full Sample OLS Estimate}
\label{subsec:approx:betaols}

For Model (\ref{linreg-matrix}), a second major goal is to approximate the full sample calculations, say the OLS estimate $\hat{\sbf\beta}_{OLS}$ in Eqn.~(\ref{olseqn}), regardless of the underlying true model parameter $\sbf{\beta}_0$. 
We refer to this as \emph{conditional inference}. 
For conditional inference, we consider the full sample as given, and thus the only source of randomness contributing to the RandNLA sampling estimators is the randomness in the sampling algorithm. 
The following theorem states that, in conditional inference, the asymptotic distribution of the sampling estimator $\tilde{\sbf\beta}$ is a normal distribution (with mean $\sbf\beta_{OLS}$ and variance $\sigma^2\mbf{\Sigma}_c$). 
The proof of Theorem~\ref{thm:beta_tilde-beta_hat} is provided in Appendix~\ref{sec:proof:thm:beta_tilde-betahat}. 

\begin{thm}[\textbf{Conditional inference}]
\label{thm:beta_tilde-beta_hat}
Assume the following regularity conditions hold.
\begin{itemize}
    \item 
    \textit{(C1)[Data condition].}  The full sample data $\{\mbf{X}, \mbf{Y}\}$, i.e., the full sample size $n$ and the number of predictors $p$ are considered fixed; $\mbf X$ is of full column rank, and $\|\mbf x_i\|<\infty$, for $i=1, \ldots, n$, where $\mbf x_i$ is the $i^{th}$ row of $\mbf X$.
    \item
    \textit{(C2)[Sampling condition].} 
    The sampling probabilities $\{\pi_i\}_{i=1}^n$ are nonzero.
\end{itemize}
Under these assumptions, as the sample size $r\to\infty$, we have 
\begin{eqnarray}
   (\sigma^2\mbf{\Sigma}_c)^{-\frac12}(\tilde{\sbf\beta}-\hat{\sbf\beta}_{OLS})
   &\stackrel{d}{\rightarrow}&
   \label{thm:eqn:btt-bth}
   \textbf{N}\left(\sbf 0, \mbf{I}_p\right), 
\end{eqnarray}
where 
$$
\mbf{\Sigma}_{c}
   = 
   \frac{1}{r}(\mbf{X}^T\mbf{X})^{-1}
   \left( \sum_{i=1}^n \frac{e_i^2}{\pi_i}\mbf{x}_i\mbf{x}_i^T \right)
   (\mbf{X}^T\mbf{X})^{-1}, \quad  e_i=Y_i -\mbf x_i^T \hat{\sbf\beta}_{OLS} ,
$$ 
and $\mbf{I}_p$ is the $p\times p$ identity. 
Thus, for conditional inference, the asymptotic mean of $\sbf{\tilde\beta}$ is 
\begin{equation}\label{eqn:thm:cond:ae}
\text{AE}(\sbf{\tilde\beta})=\hat{\sbf\beta}_{OLS}, 
\end{equation}
i.e., $\sbf{\tilde\beta}$ is an asymptotically unbiased estimator of $\sbf\beta_{OLS}$, and the asymptotic variance of $\sbf{\tilde\beta}$ is 
\begin{equation}\label{eqn:thm:cond:avar}
AVar(\sbf{\tilde\beta})= \sigma^2\mbf{\Sigma}_c. 
\end{equation}
\end{thm}

\textbf{Remark.}
Theorem~\ref{thm:beta_tilde-beta_hat} shows that as the sample size $r$ gets larger, the distribution of $\tilde{\sbf\beta}$ is well-approximated by a normal distribution, with mean $\hat{\sbf\beta}_{OLS}$ and variance $\sigma^2\mbf{\Sigma}_c$. 

\textbf{Remark.}
Similar to unconditional inference, the asymptotic variance $AVar(\sbf{\tilde\beta})$ here also has ``sandwich-type'' expression, where 
the center term (here, $\left( \sum_{i=1}^n \frac{e_i^2}{\pi_i}\mbf{x}_i\mbf{x}_i^T \right)$) depends on the reciprocal sampling probabilities. 
Thus, we also expect that extremely small probabilities will result in large variances of the corresponding estimators. 

\textbf{Remark.}
In Theorem~\ref{thm:beta_tilde-beta_hat}, $AVar(\sbf{\tilde\beta})$ depends on the full sample least square residuals, i.e., the $e_i$s.
These are not readily available from the sample. 
To solve this problem and to obtain meaningful results, we take the expectation of the $e_i^2$s.
The metric we use is thus the EAMSE,
\begin{equation}
\label{eqn:EVc0}
EAMSE(\tilde{\sbf\beta}; \hat{\sbf\beta}_{OLS})= E_\mbf{Y}( AMSE(\tilde{\sbf\beta}; \hat{\sbf\beta}_{OLS})) .
\end{equation}
The EAMSE is a function of  sampling probabilities $\{\pi_i\}_{i=1}^n$.

It is natural to ask whether there is an optimal estimator, i.e., a sample estimator with the smallest EAMSE for estimating $\sbf\beta_{OLS}$. 
Using the asymptotic results in Theorem~\ref{thm:beta_tilde-beta_hat}, we propose the following three estimators for various purposes.

\paragraph{Estimating $\hat{\sbf\beta}_{OLS}$}

We can use the results of Theorem~\ref{thm:beta_tilde-beta_hat} to obtain expressions of interest for the EAMSE of various quantities.
As with the AMSE, these will depend on the sampling probabilities.
Thus, we can derive the optimal sampling probabilities for various quantities of interest.
We start with $EAMSE(\tilde{\sbf\beta}; \hat{\sbf\beta}_{OLS})$.

The following proposition gives the minimum $EAMSE(\tilde{\sbf\beta}; \hat{\sbf\beta}_{OLS})$ sampling estimator.
For this result, we denote that $\text{E}_{\mbf Y}(e_i^2)=(1-h_{ii})\sigma^2$.

\begin{prop}
\label{thm:beta_hat:pi}
For the $EAMSE(\tilde{\sbf\beta}; \hat{\sbf\beta}_{OLS})$, we have that 
\begin{equation}
{\label{eqn:EVc}}
EAMSE(\tilde{\sbf\beta}; \hat{\sbf\beta}_{OLS})=\text{E}_{\sbf Y}(\text{tr}(AVar(\sbf{\tilde\beta})))
= \frac{1}{r}\sum_{i=1}^n \frac{(1-h_{ii})\sigma^2}{\pi_i}||(\mbf{X}^T\mbf{X})^{-1}\mbf{x}_i||^2.
\end{equation}
Given \eqref{eqn:EVc}, the sample estimator with the sampling probabilities
\begin{eqnarray}\label{eqn:pi_ICNLEV}
\pi_i= \frac{\sqrt{1-h_{ii}}\|(\mbf{X}^T\mbf{X})^{-1}\mbf{x}_i\|}{\sum_{i=1}^{n}\sqrt{1-h_{ii}}\|(\mbf{X}^T\mbf{X})^{-1}\mbf{x}_i\|}, i=1,\ldots, n, 
\end{eqnarray}
(which we call the inverse-covariance negative-leverage (ICNLEV) estimator)
has the smallest 
$EAMSE(\tilde{\sbf\beta}; \hat{\sbf\beta}_{OLS})$.
\end{prop}

\paragraph{Estimating linear functions of $\hat{\sbf\beta}_{OLS}$.}

In addition to approximating $\hat{\sbf\beta}_{OLS}$, one may also be interested in linear functions of $\hat{\sbf\beta}_{OLS}$. 
Here, we present results for $\hat{\mbf Y} = \mbf X\hat{\sbf\beta}_{OLS}$ and $\mbf{X}^{T}\mbf{X}\hat{\sbf\beta}_{OLS}$ (although clearly similar results hold for other functions of the form $\mbf L \hat{\sbf\beta}_{OLS}$).

We start with estimating $\hat{\mbf Y} = \mbf X\hat{\sbf\beta}_{OLS}$. 

\begin{prop}\label{thm:beta_hat:pi3}
For the $EAMSE(\mbf X\tilde{\sbf\beta}; \mbf X\hat{\sbf\beta}_{OLS})$, we have that
\begin{equation}
{\label{eqn:EVc:X}}
EAMSE(\mbf X\tilde{\sbf\beta}; \mbf X\hat{\sbf\beta}_{OLS})
= \frac{1}{r}\sum_{i=1}^n \frac{(1-h_{ii})\sigma^2}{\pi_i}||\mbf X(\mbf{X}^T\mbf{X})^{-1}\mbf{x}_i||^2.
\end{equation}
Given \eqref{eqn:EVc:X}, the sample estimator with the sampling probabilities
\begin{eqnarray}\label{eqn:pi_ICNLEV:hrow}
\pi_i= \frac{\sqrt{1-h_{ii}}\|\mbf X(\mbf{X}^T\mbf{X})^{-1}\mbf{x}_i\|}{\sum_{i=1}^{n}\sqrt{1-h_{ii}}\|\mbf X(\mbf{X}^T\mbf{X})^{-1}\mbf{x}_i\|}
=\frac{\sqrt{(1-h_{ii})h_{ii}}}{\sum_{i=1}^{n}\sqrt{(1-h_{ii})h_{ii}} }, i=1,\ldots, n, 
\end{eqnarray}
(which we call the root leveraging negative-leverage (RLNLEV) estimator) 
has the smallest 
value for the 
$EAMSE(\mbf X\tilde{\sbf\beta}; \mbf X\hat{\sbf\beta}_{OLS})$.
\end{prop}

We next consider estimating $\mbf{X}^{T}\mbf{X}\hat{\sbf\beta}_{OLS}$. 

\begin{prop}
\label{thm:beta_hat:pi2}
For the $EAMSE(\mbf{X}^{T}\mbf{X}\tilde{\sbf\beta}; \mbf{X}^{T}\mbf{X}\hat{\sbf\beta}_{OLS})$, we have that
\begin{equation}
\label{traceve}
EAMSE(\mbf{X}^{T}\mbf{X}\tilde{\sbf\beta}; \mbf{X}^{T}\mbf{X}\hat{\sbf\beta}_{OLS})
=\frac{1}{r}\sum_{i=1}^n \frac{(1-h_{ii})\sigma^2}{\pi_i}||\mbf{x}_i||^2.
\end{equation}
Given \eqref{traceve}, the sampling estimator with the sampling probabilities 
\begin{eqnarray}\label{eqn:pi_opt_e}
\pi_i=\frac{\sqrt{1-h_{ii}}\|\mbf x_i\|}{\sum_{i=1}^{n}\sqrt{1-h_{ii}}\|\mbf x_i\|}, i=1,\ldots, n, 
\end{eqnarray}
(which we call the predictor-length negative-leverage (PLNLEV) estimator) 
has the smallest 
value for the 
$EAMSE(\mbf{X}^{T}\mbf{X}\tilde{\sbf\beta}; \mbf{X}^{T}\mbf{X}\hat{\sbf\beta}_{OLS})$.
\end{prop}

\textbf{Remark.}
All these proposed metrics can be computed in the time it takes to approximate leverage scores, i.e., the time to implement a random projection, using the algorithm of \cite{DMMW12_JMLR}, since they are essentially strongly related to leverage scores.  

As a summary, the six proposed estimators (IC, RL, PL, ICNLEV, RLNLEV, PLNLEV), along with three existing estimators (UNIF, BLEV/ALEV, SLEV) are presented in Table~\ref{tab:intro:summary}.  

\begin{table}[t]  
\begin{center}\footnotesize
\begin{tabular}{ c c c c c }
\hline
\hline
Estimator &
\specialcell{Sampling\\Probabilities}
&\specialcell{Criterion} 
& Results  \\
\hline
\vspace{2mm}
UNIF & $\pi_i = \frac 1n$  & $--$ &$--$ \\
\vspace{2mm}
BLEV/ALEV & $\pi_i = \frac{h_{ii}}{\sum_{i=1}^n h_{ii}}$  &  $--$&\citet{Drineas:06,DMMW12_JMLR} \\
\vspace{2mm}
SLEV & $\pi_i = \lambda\frac{h_{ii}}{\sum_{i=1}^n h_{ii}}+(1-\lambda)\frac{1}{n}$ & $--$&\citet{Ma:13,Ma:15} \\
\vspace{2mm}
IC &$\pi_i= \frac{\|(\mbf{X}^T\mbf{X})^{-1}\mbf{x}_i\|}{\sum_{i=1}^{n}\|(\mbf{X}^T\mbf{X})^{-1}\mbf{x}_i\|}$ & 
\footnotesize{$AMSE(\tilde{\sbf\beta}; {\sbf\beta}_{0})$}& 
Section~\ref{subsec:est:beta}, Eqn.~(\ref{eqn:pi_opt})\\
\vspace{2mm}
RL &$\pi_i= \frac{\sqrt{h_{ii}}}{\sum_{i=1}^{n}\sqrt{h_{ii}}}$ & 
\footnotesize{$AMSE(\mbf{X}\tilde{\sbf\beta}; \mbf{X}{\sbf\beta}_{0})$}& 
Section~\ref{subsec:est:beta}, Eqn.~(\ref{eqn:pi_sq_hii})\\
\vspace{2mm}
PL &$\pi_i= \frac{\|\mbf{x}_i\|}{\sum_{i=1}^{n}\|\mbf{x}_i\|}$  & 
\footnotesize{$AMSE(\sbf X^T\sbf X\tilde{\sbf\beta}; \sbf X^T\sbf X{\sbf\beta}_{0})$}& 
Section~\ref{subsec:est:beta}, Eqn.~(\ref{eqn:pi_pl})\\
\vspace{2mm}
ICNLEV &$\pi_i= \frac{\sqrt{1-h_{ii}}\|(\mbf{X}^T\mbf{X})^{-1}\mbf{x}_i\|}{\sum_{i=1}^{n}\sqrt{1-h_{ii}}\|(\mbf{X}^T\mbf{X})^{-1}\mbf{x}_i\|}$ &
\footnotesize{$EAMSE(\tilde{\sbf\beta}; \hat{\sbf\beta}_{OLS})$} & 
Section~\ref{subsec:approx:betaols}, Eqn.~(\ref{eqn:pi_ICNLEV})\\
\vspace{2mm}
RLNLEV &$\pi_i= \frac{\sqrt{(1-h_{ii})h_{ii}}}{\sum_{i=1}^{n}\sqrt{(1-h_{ii})h_{ii}}}$ &
\footnotesize{$EAMSE(\mbf{X}\tilde{\sbf\beta}; \mbf{X}\hat{\sbf\beta}_{OLS})$} & 
Section~\ref{subsec:approx:betaols}, Eqn.~(\ref{eqn:pi_ICNLEV:hrow})\\
\vspace{2mm}
PLNLEV &$\pi_i= \frac{\sqrt{1-h_{ii}}\|\mbf{x}_i\|}{\sum_{i=1}^{n}\sqrt{1-h_{ii}}\|\mbf{x}_i\|}$ &
\footnotesize{$EAMSE(\sbf X^T\sbf X\tilde{\sbf\beta};\sbf X^T\sbf X\hat{\sbf\beta}_{OLS})$}&
Section~\ref{subsec:approx:betaols}, Eqn.~(\ref{eqn:pi_opt_e})\\
\hline		
\end{tabular}
\caption{Summary of three existing sampling estimators (UNIF, BLEV, SLEV) and the six sampling estimators (IC, RL, PL, ICNLEV, RLNLEV, PLNLEV) presented in this paper. 
}
\label{tab:intro:summary}
\end{center}
\end{table}

\subsection{Relationship of the Sampling Estimators}
\label{subsec:relation}

Here, we study the relationships between the probability distributions given by IC, RL, PL, ICNLEV, RLNLEV, PLNLEV, and those given by SLEV and BLEV.

\subsubsection{ ``Shrinkage'' Properties of Proposed Estimators }

\begin{figure}[t] 
\centering
  \includegraphics[width=0.75\textwidth]{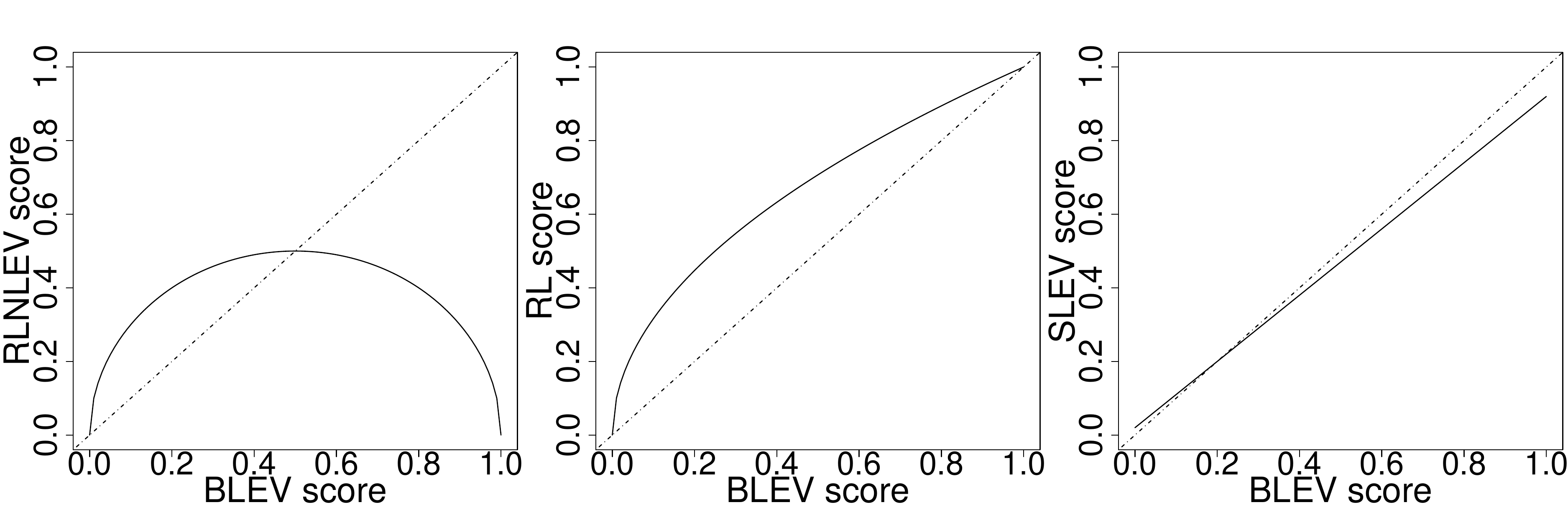}
\caption{Relationship between sampling methods.
Left panel: RLNLEV score ($\sqrt{(1-h_{ii})h_{ii}}$) versus BLEV score ($h_{ii}$). Middle panel: RL score ($\sqrt{h_{ii}} $) versus BLEV score ($h_{ii}$). Right panel: {SLEV score ($0.9h_{ii}+0.1p/n$, where $p/n=0.2$) versus BLEV score ($h_{ii}$).}}
\label{fig:OptimalShrinkage}
\end{figure}

We illustrate the ``shrinkage'' property of the proposed optimal sampling probabilities, compared to the BLEV sampling probabilities.  
For convenience,  we refer to the numerators of the sampling probabilities in a sampling estimators as the scores, e.g., the RL score is $\sqrt{h_{ii}}$ and the RLNLEV score is $\sqrt{(1-h_{ii})h_{ii}}$. 
In Figure~\ref{fig:OptimalShrinkage}, we plot 
the RL score, RLNLEV score, and SLEV score ($0.9h_{ii}+0.1p/n$ with $p/n=0.2$) as functions of the leverage score $h_{ii}$ (i.e., the BLEV score in Figure~\ref{fig:OptimalShrinkage}). 
Observe that the RLNLEV score  amplifies small $h_{ii}$s but shrinks large $h_{ii}$s.  
Both RLNLEV and RL scores provide nonlinear shrinkage of the BLEV. 
The SLEV scores also shrink large $h_{ii}$s and amplify small $h_{ii}$s, but in a linear fashion.

The advantage of such ``shrinkage'' is two-fold. 
On the one hand, the data with high leverage scores could be ``outliers.'' 
Shrinking the sampling probabilities of high leverage data points reduces the risk of selecting outliers into the sample. 
On the other hand, amplifying the sampling probabilities of low leverage data points reduces the variance of the resulting sampling estimators.

\subsubsection{The Role of $h_{ii}$s.}
On the one hand, if the $h_{ii}$s are homogeneous, 
then the sampling probabilities of the ICNLEV estimator ($\frac{\sqrt{1-h_{ii}}\|(\mbf{X}^T\mbf{X})^{-1}\mbf{x}_i\|}{\sum_{i=1}^{n}\sqrt{1-h_{ii}}\|(\mbf{X}^T\mbf{X})^{-1}\mbf{x}_i\|}$) and those of the IC estimator ($\frac{\|(\mbf{X}^T\mbf{X})^{-1}\mbf{x}_i\|}{\sum_{i=1}^{n}\|(\mbf{X}^T\mbf{X})^{-1}\mbf{x}_i\|}$) will be similar to each other. 
On the other hand, since $\sum_{i=1}^nh_{ii}=p$, given a fixed value of $p$, we expect that $h_{ii}$s are small when sample size $n$ is large. 
When $h_{ii}=o(1)$ for all $i=1, \ldots, n$, i.e.,  $h_{ii}$s are extremely small compared to 1, the sampling probabilities of the ICNLEV estimator and those of the IC estimator will also be similar. 
Analogous arguments also apply to PLNLEV and PL.

\subsubsection{Two Examples. }
\label{subsubsec:example}
We now use two examples to illustrate the relationship between the sampling probabilities in various sampling estimators.

\begin{figure}[H] 
    \centering
    \begin{subfigure}[!ht]{0.9\textwidth}
        \centering
       \includegraphics[width=0.27\textwidth]{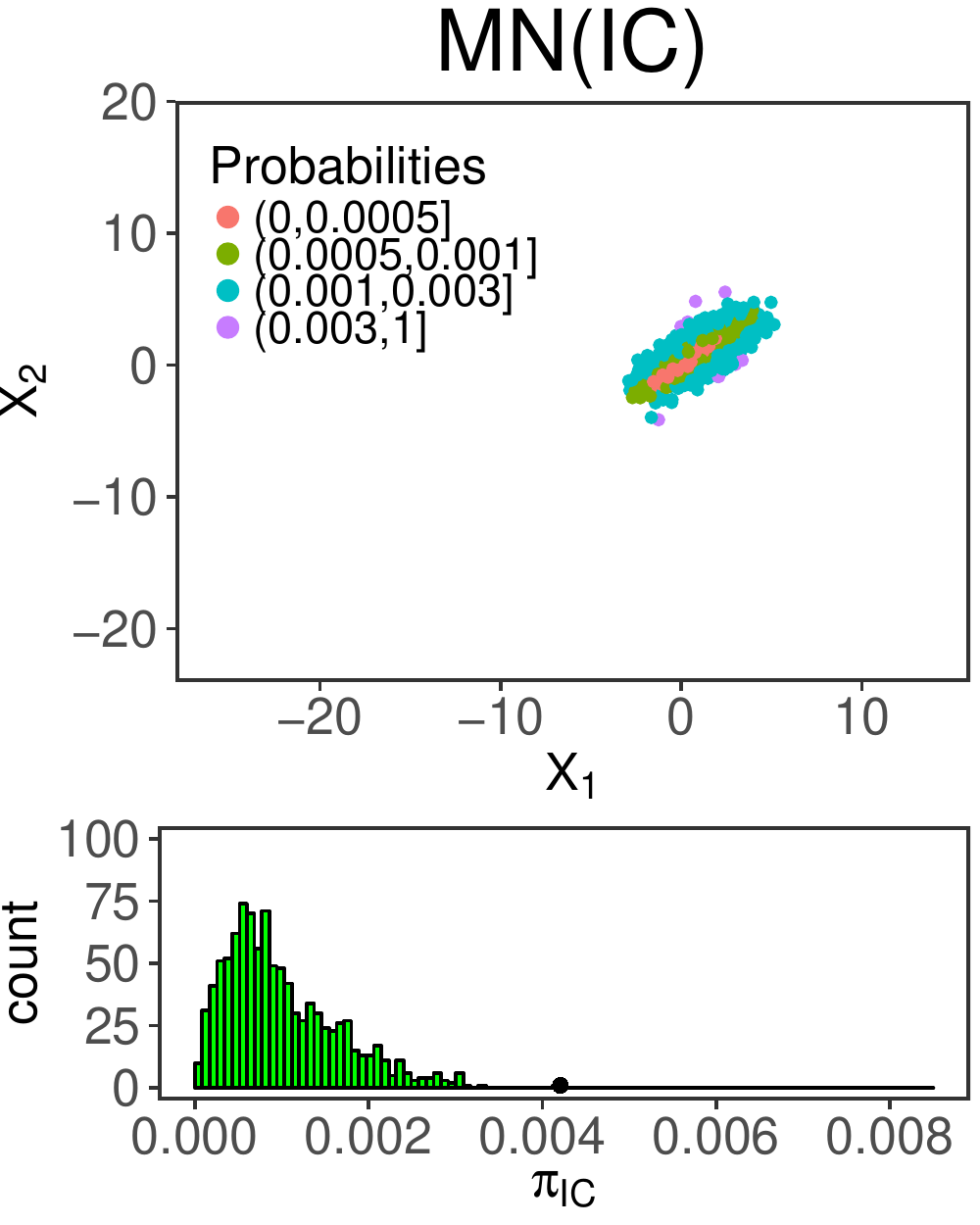} 
       \includegraphics[width=0.27\textwidth]{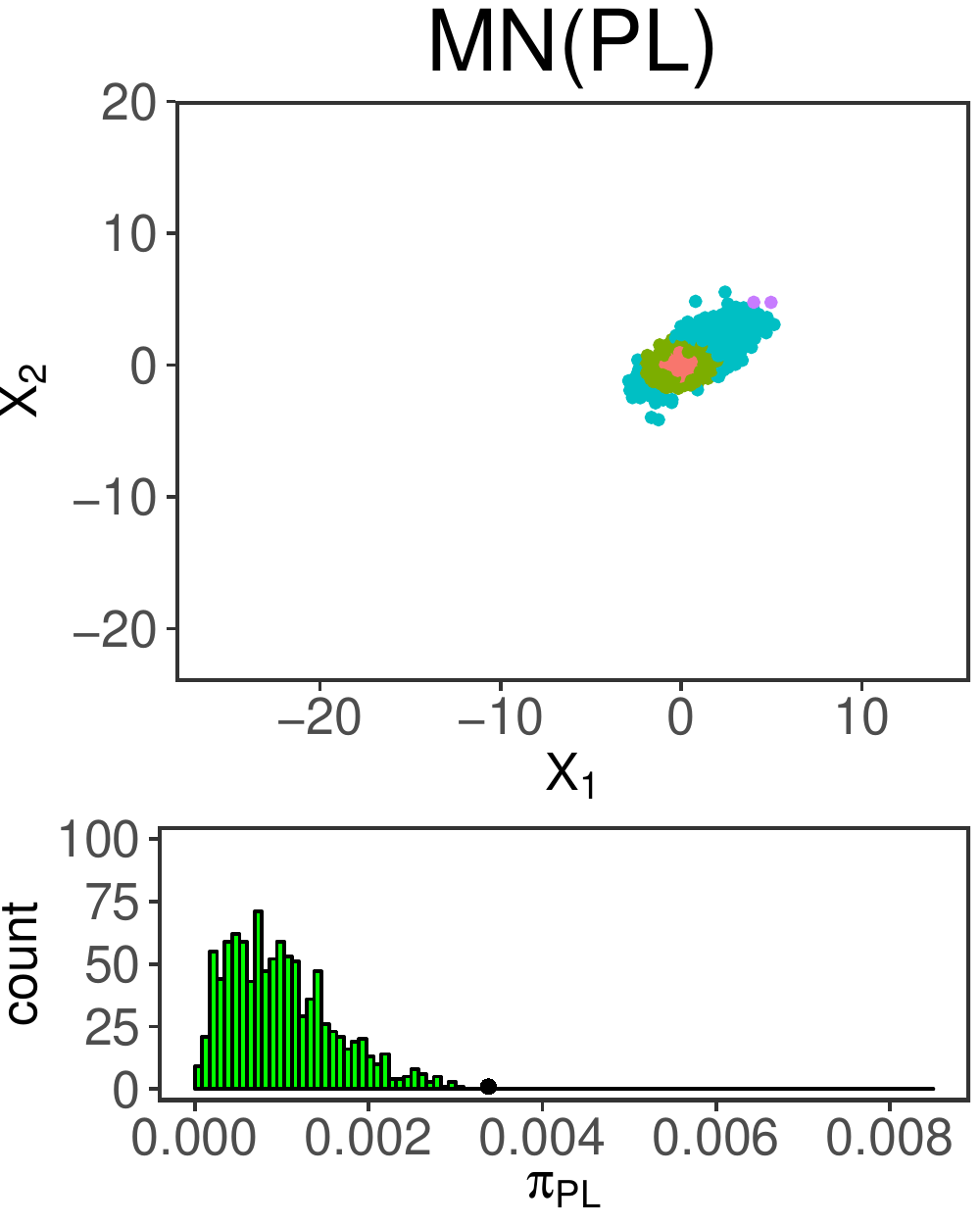}
       \includegraphics[width=0.27\textwidth]{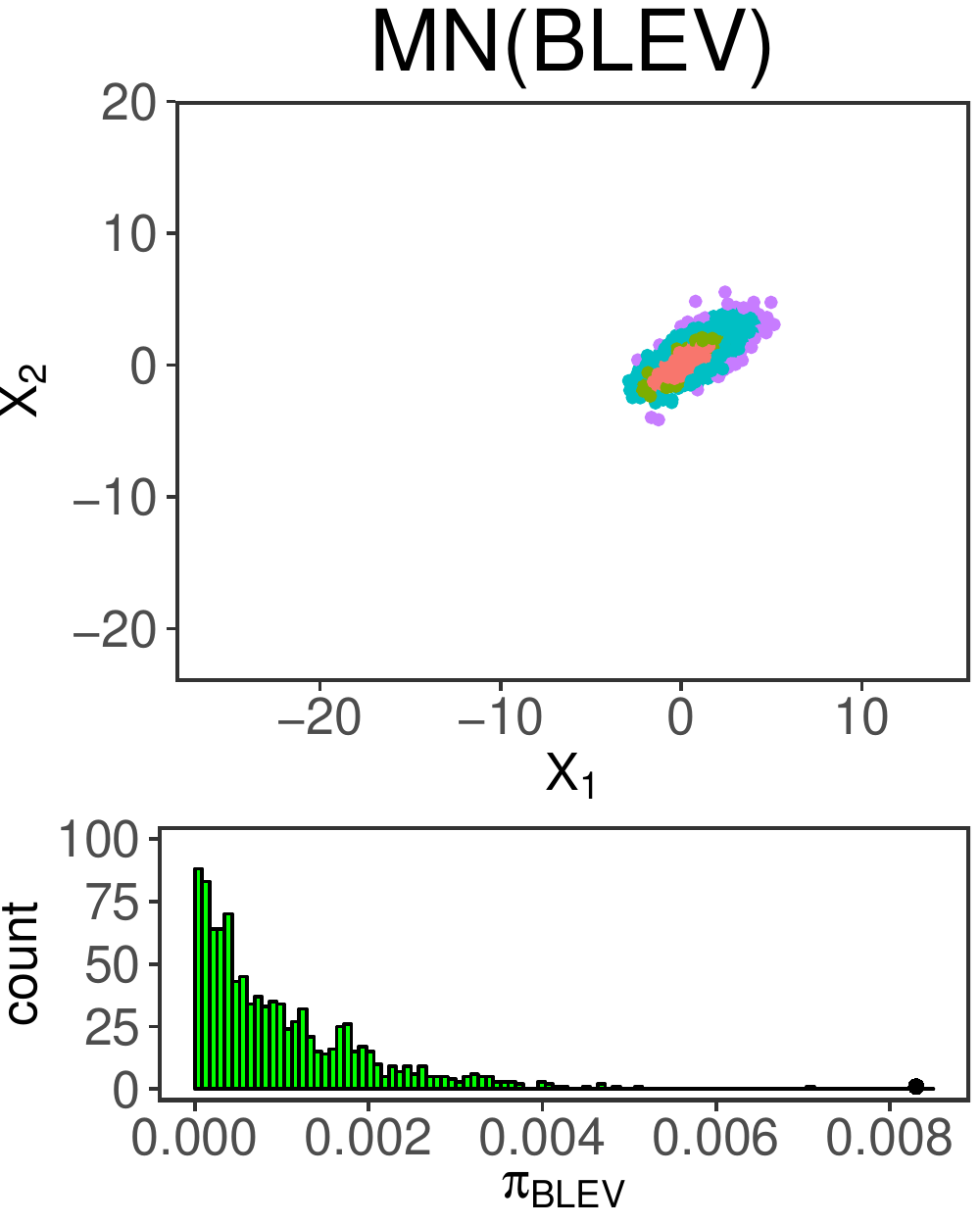}
        \caption{\footnotesize{Scatter plots (first row) of data points generated from a bivariate normal distribution with colors coding the sampling probability in the IC (left panel), PL (middle panel), and BLEV (right panel). Below each scatter plot is the histogram of the corresponding sampling probabilities, with the dot representing the maximum probability. } }
    \end{subfigure}%
    \\
    \begin{subfigure}[!ht]{0.9\textwidth}
        \centering
       \includegraphics[width=0.27\textwidth]{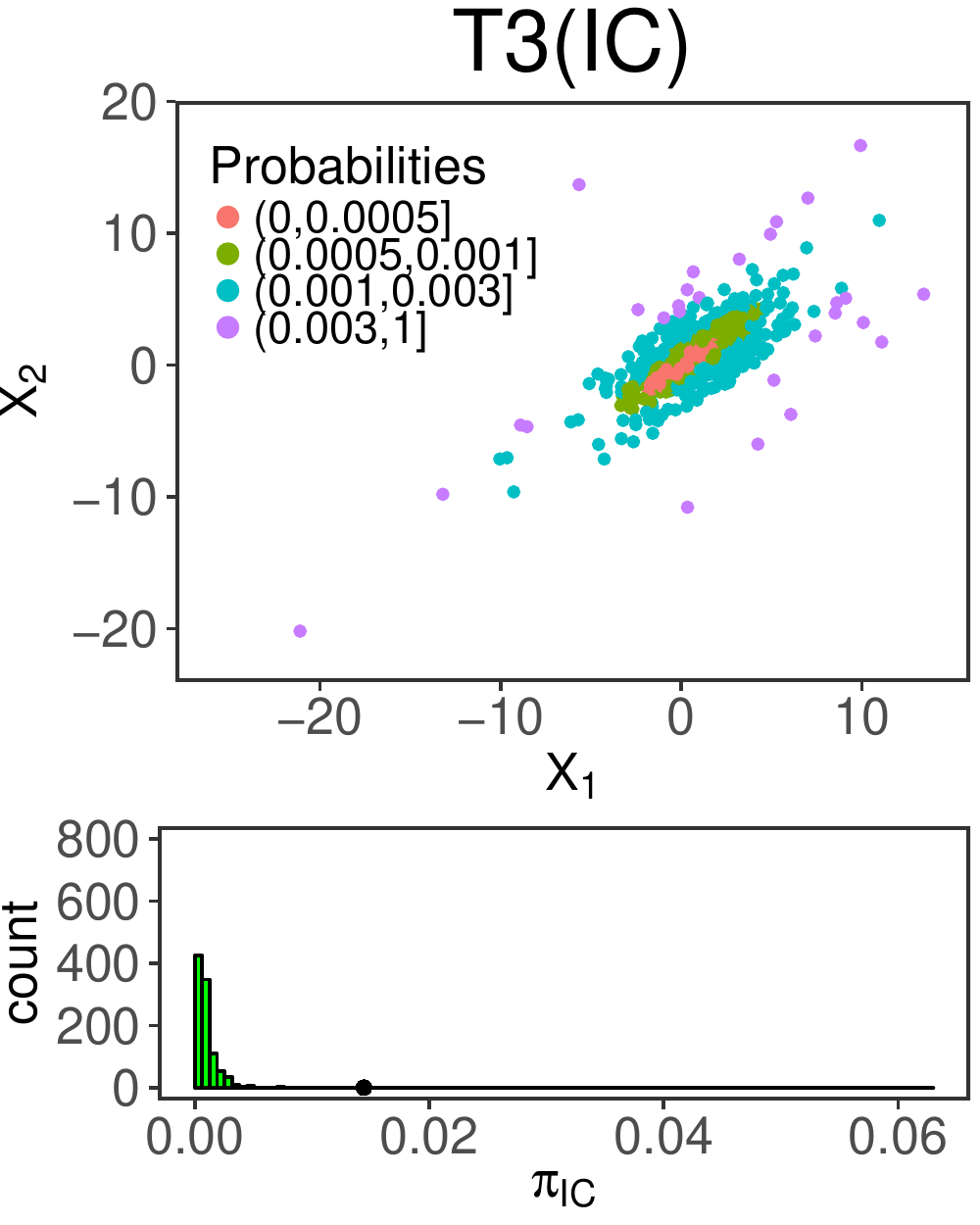}
       \includegraphics[width=0.27\textwidth]{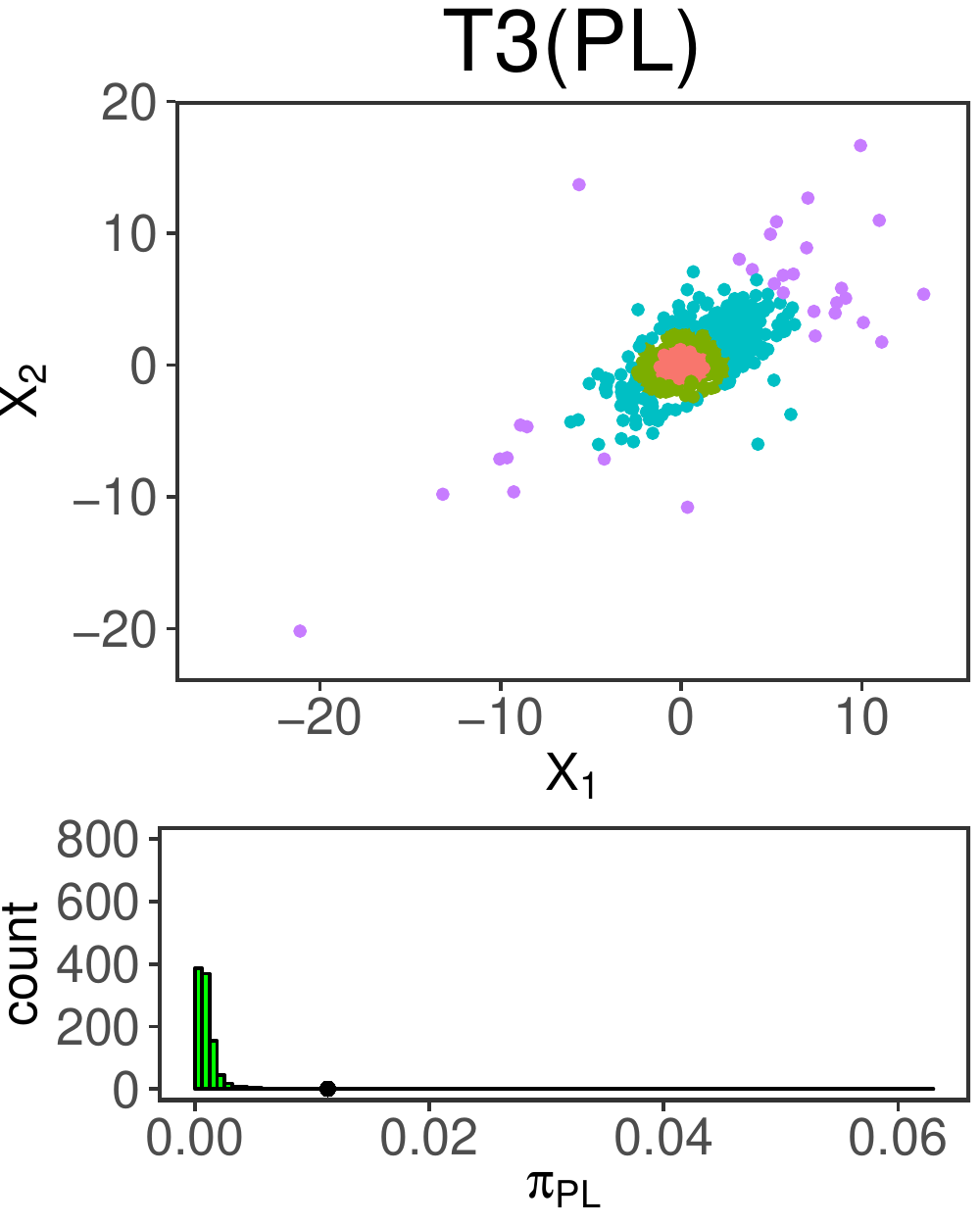}
       \includegraphics[width=0.27\textwidth]{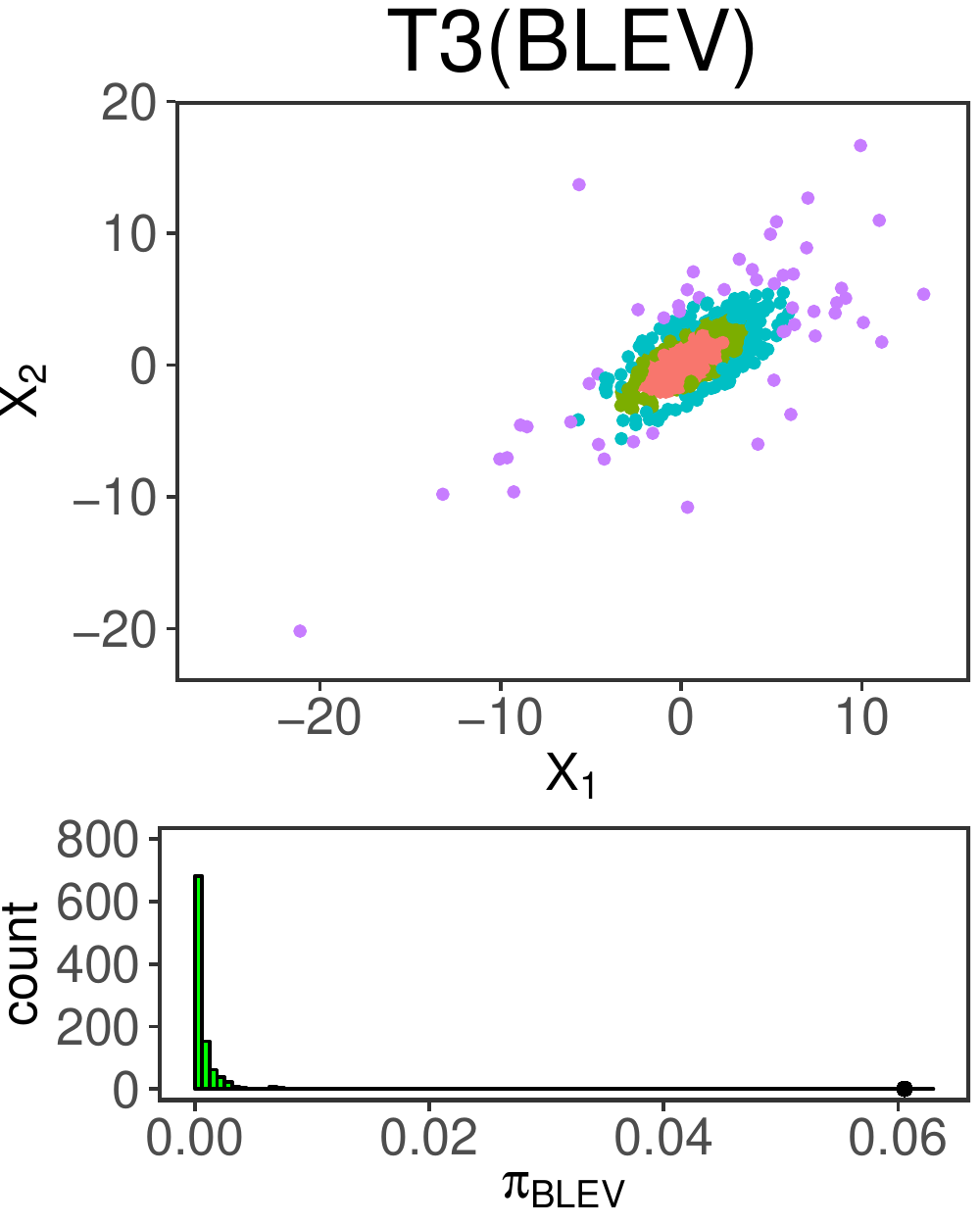}
        \caption{\footnotesize{Same as in (a), except that the data points are generated from a bivariate noncentral $t$ distribution with three degrees of freedom. } }
    \end{subfigure}
    \\
     \begin{subfigure}[!ht]{0.9\textwidth}
        \centering
        \includegraphics[width=0.27\textwidth]{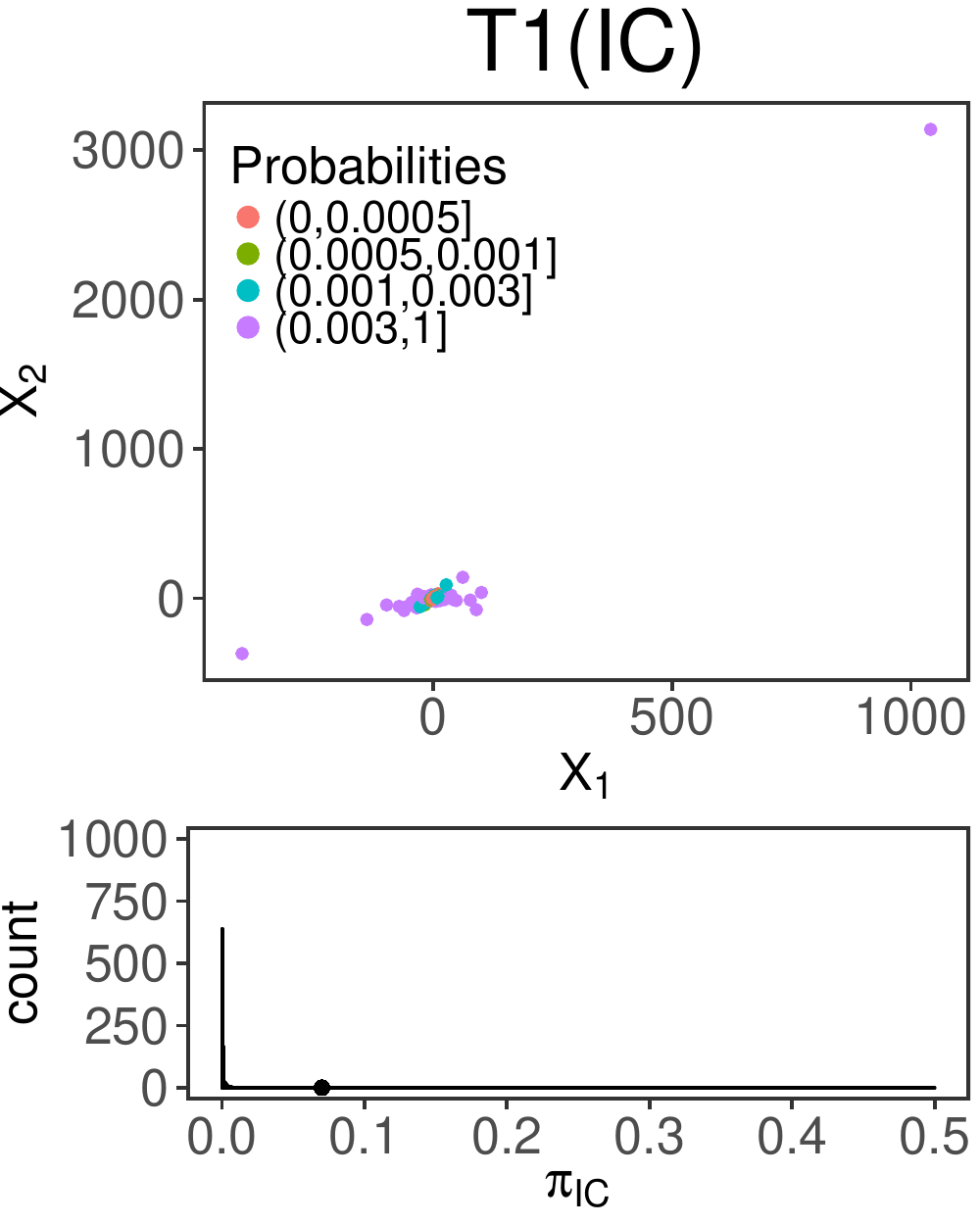}
        \includegraphics[width=0.27\textwidth]{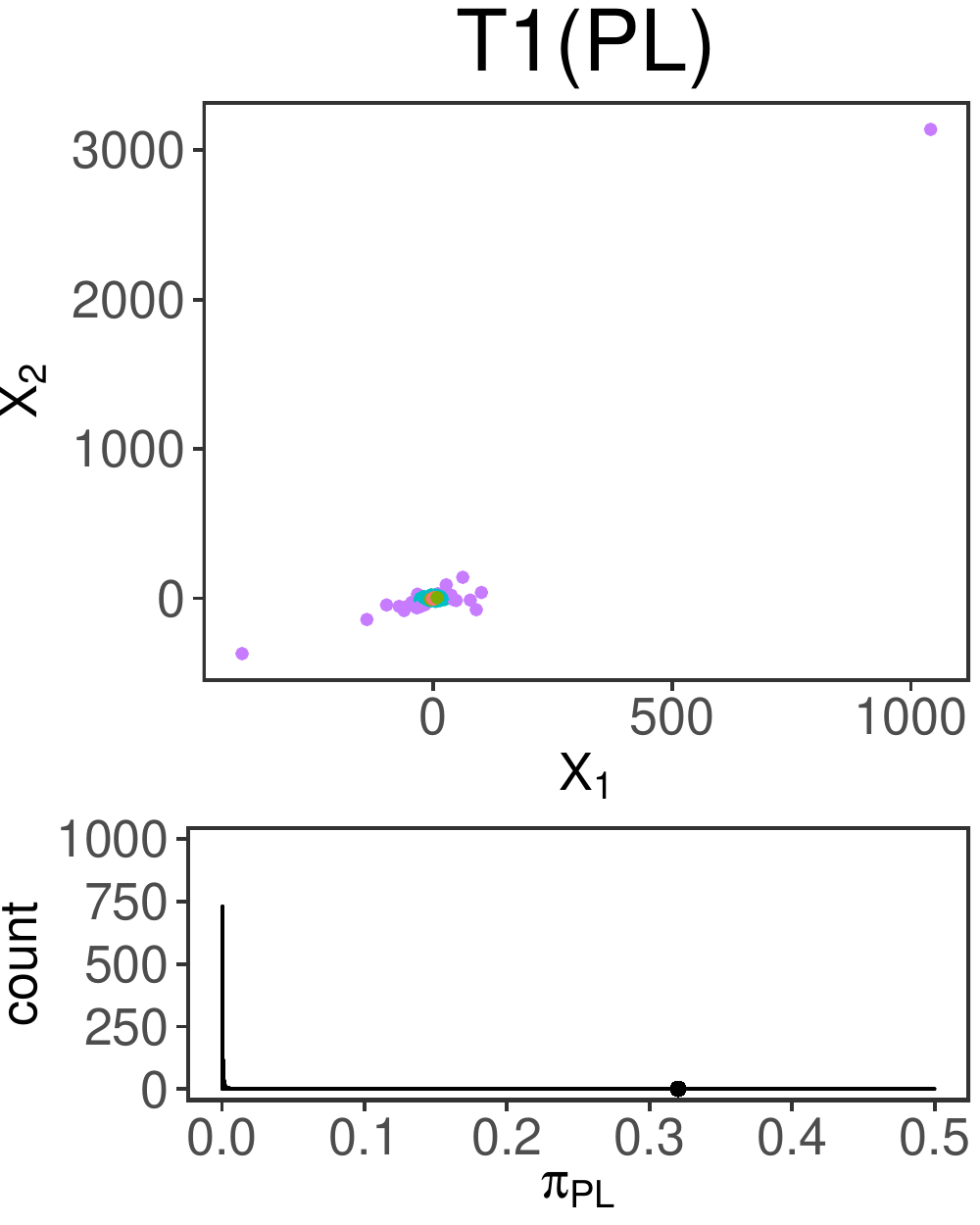}
        \includegraphics[width=0.27\textwidth]{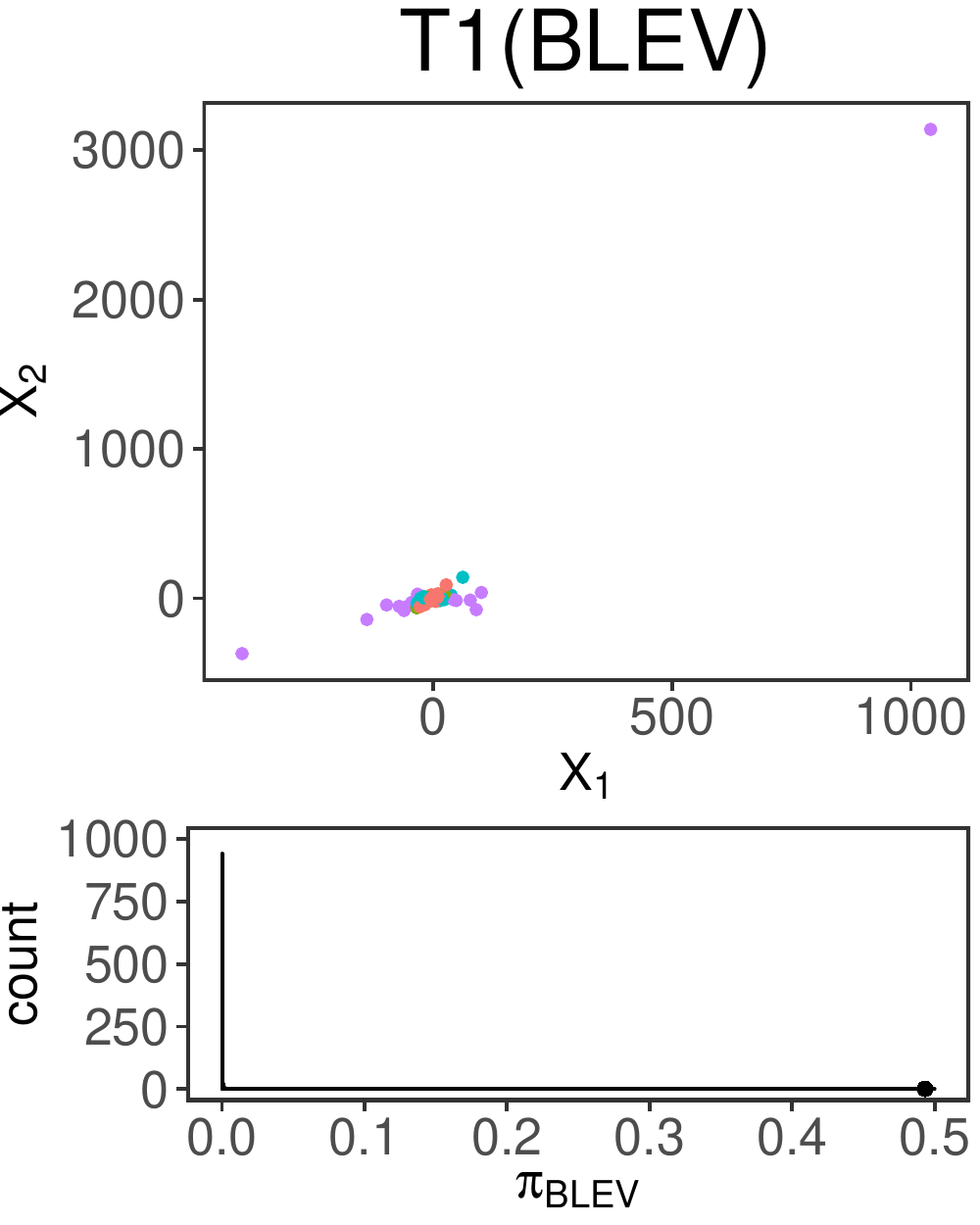}
        \caption{\footnotesize{Same as in (a), except that the data points are generated from a bivariate noncentral $t$ distribution with one degree of freedom.  } }
    \end{subfigure}
    \caption{\footnotesize{Scatter plots of 1000 data points generated from three distributions in Example 2 in Section~\ref{subsubsec:example} and the histograms of sampling probabilities. }}
    \label{fig:simu:sampling}
\end{figure}

{\bf Example 1: Orthogonal predictor matrix, i.e., $\mbf{X}^T\mbf{X}=\mbf{I}$}. 
Consider a linear regression model with an orthogonal predictor matrix, i.e., $\mbf{X}^T\mbf{X}=\mbf{I}$. 
In this case, we have $h_{ii}= \mbf{x}_i^T(\mbf{X}^T\mbf{X})^{-1}\mbf{x}_i =\|\mbf{x}_i\|^2$. 
Further, the ICNLEV score, RLNLEV score, and PLNLEV score are the same and equal $\sqrt{(1-h_{ii})h_{ii}}$. 
Analogously, the IC score coincides with the RL score and the PL score, and all equal $\|\mbf x_i\|$.
 
{\bf Example 2: A two dimensional example}. 
Consider also a toy example of a linear regression model with $p=2$ correlated predictors. 
We generated 1000 data points for two predictors from a multivariate normal distribution, a multivariate noncentral $t$ distribution with three degrees of freedom, and a multivariate noncentral $t$ distribution with one degree of freedom.  
In Figure~\ref{fig:simu:sampling}, we present scatter plots of these data points. 
In each scatter plot, the color of points indicates the magnitude of sampling probabilities in IC, PL and BLEV methods. Below each scatter plot, we also present histograms of the corresponding sampling probabilities. 
Examination of Figure~\ref{fig:simu:sampling} reveals one pattern shared by all sampling distributions, i.e., 
the sampling probabilities of data points in the center are smaller than those of data points at the boundary. 
In addition, note that, compared to 
$\pi_{i}^{PL}\propto \|\mbf{x}_i\|$,
both 
$\pi_{i}^{IC} \propto\|(\mbf{X}^T\mbf{X})^{-1}\mbf{x}_i\|$ 
and 
$\pi_{i}^{BLEV}\propto \mbf x_i^T(\mbf{X}^T\mbf{X})^{-1}\mbf{x}_i$ depend on $(\mbf{X}^T\mbf{X})^{-1}$, which normalizes the scale of the predictors. 
Thus, we notice that data points with high probabilities in PL scatter around the upper right and lower left corner.
However, the data points with high probabilities in IC and BLEV form a contour toward the exterior of the data cloud. 
This difference is caused by the effect of the normalization using $(\mbf{X}^T\mbf{X})^{-1}$. 
The histograms in each row also show the key difference between the sampling probabilities of BLEV and those of IC and PL, i.e., the sampling probability distribution of BLEV is more dispersed than others. 
In other words, there are a significant number of data points with either extremely large or extremely small probabilities in BLEV. 
This phenomenon is also observed in Figure~\ref{fig:prob} in Section~\ref{sec:simu-and-real}.

\section{Empirical Results}
\label{sec:simu-and-real}

In this section, we present a summary of the main results of our empirical analysis, which consisted of an extensive analyses on simulated and real datasets.

\subsection{Simulation Setting}
\label{simu:setting}

We generated synthetic data from Model~(\ref{linreg-matrix})  
with $p=10$, $n=5000$, and random error $\varepsilon_i \stackrel{iid}{\sim} N(0,1)$. 
We set the first and last two entries of $\sbf{\beta}_0$ to be $1$ and the rest to be $0.1$. 
We generated the predictors from the following distributions.
\begin{itemize}
\item 
Multivariate normal distribution $\textbf{N}(\sbf{1}, \mbf{D})$, where $\sbf{1}$ is a $p\times 1$ column vector of $1$s, and the $(i,j)^{th}$ element of $\mathbf{D}$ is set to $1\times 0.7^{|i-j|}$, for $i,j=1,\ldots,p$. 
We refer to this as \underline{MN} data. 
\item 
Multivariate noncentral $t$-distribution with 3 degrees of freedom, noncentrality parameter $\sbf{1}$, and scale matrix $\mbf{D}$, i.e., $t_3(\sbf{1}, \mbf{D})$. 
We refer to this as \underline{T3} data. 
\item 
Log-normal distribution $\textbf{LN}(\sbf{1},\mathbf{D})$. 
We refer to this as \underline{LN} data. 
\item 
Multivariate noncentral $t$-distribution with 1 degree of freedom, noncentrality parameter $\sbf{1}$, and scale matrix $\mbf{D}$, i.e., $t_1(\sbf{1}, \mbf{D})$. 
We refer to this as \underline{T1} data. 
\end{itemize}
For  
$t_1(\sbf{1}, \mbf{D})$, the expectation and variance do not exist. 
This violates Condition (A1) in Theorem~\ref{thm:beta_tilde-beta}. 
Thus, the asymptotic squared bias and asymptotic variance of the proposed estimators might not converge quickly to 0, as $r$~increases.

\begin{figure}[t] 
\centering
\includegraphics[scale=0.5]{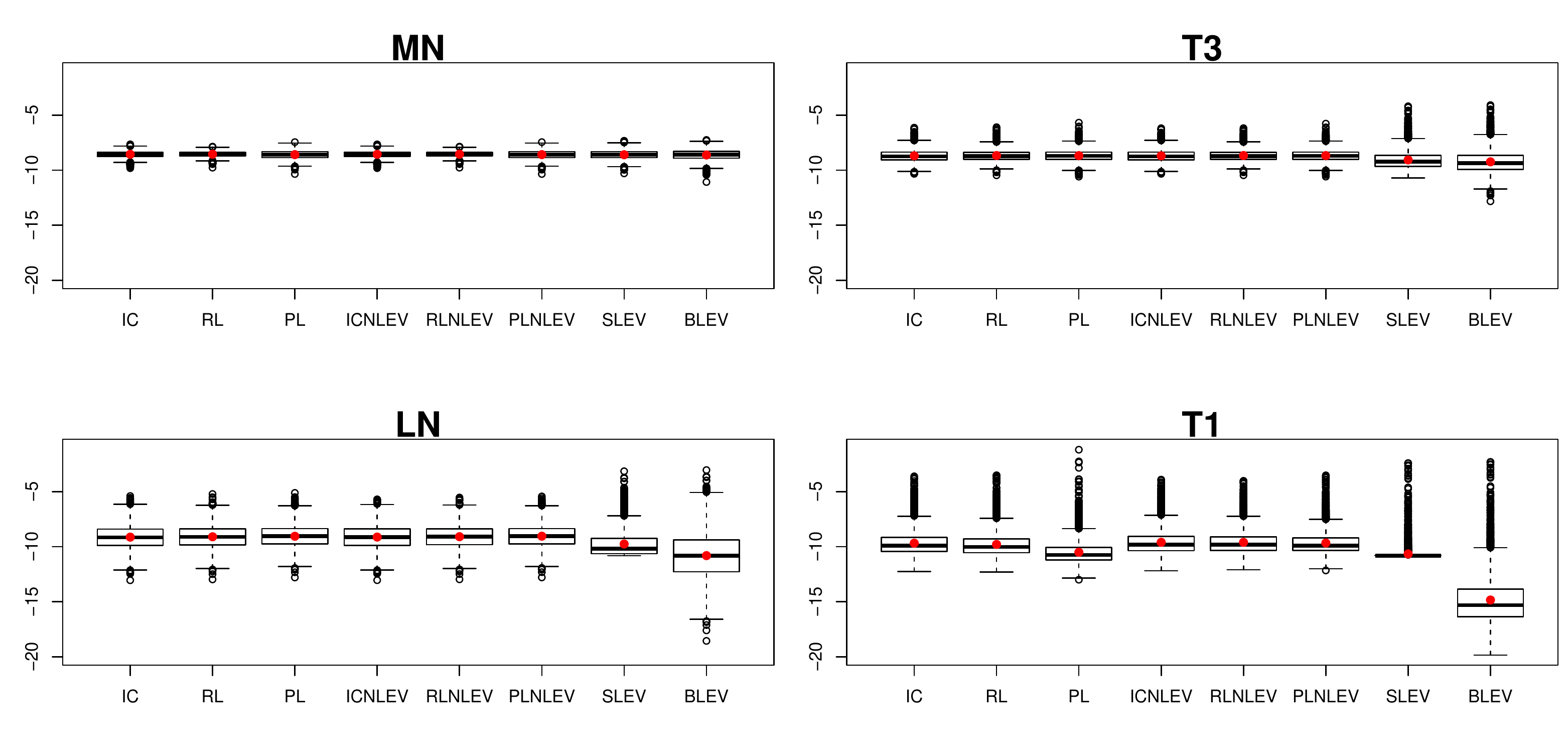}
\caption{Box plots of the sampling probabilities (in log scale) of all data points for IC, RL, PL, ICNLEV, RLNLEV, PLNLEV, SLEV, and BLEV (from left to right in each panel) for \underline{MN}, \underline{T3}, \underline{LN}, and \underline{T1} data, for $p$=10 and $n=$5000. In each box plot, the dot inside the box indicates the mean of corresponding sampling probabilities (in log scale).}
\label{fig:prob}
\end{figure}

In Figure~\ref{fig:prob}, we present box plots of the sampling probabilities (in log scale) of all the data points in  IC, RL, PL, ICNLEV, RLNLEV, PLNLEV, SLEV, and BLEV (from left to right) for \underline{MN}, \underline{T3}, \underline{LN}, and \underline{T1}. 
The sampling probability distributions of BLEV are more dispersive than those of other estimators. There exist a significant number of extremely small sampling probabilities in BLEV, especially when the data distribution has heavier tails, such as is the case for \underline{LN} and \underline{T1}. 
These extremely small sampling probabilities in BLEV are effectively mitigated in SLEV.  
However, the medians of the sampling probabilities in SLEV are still smaller than the first quartiles of the sampling probabilities in ICNLEV, IC, PLNLEV, and PL in \underline{T3}, \underline{LN}, and \underline{T1}.
The relatively small sampling probabilities in BLEV and SLEV will inflate the variance of the sampling estimators (recall the expression for the asymptotic variances in Theorems~\ref{thm:beta_tilde-beta} and~\ref{thm:beta_tilde-beta_hat}). 
Thus, it is expected that BLEV and SLEV will give rise to estimates with relatively large variances, especially when data were generated from more heavy-tailed distributions, e.g., \underline{LN} and \underline{T1}.

\subsection{Sampling Estimators for Estimating Model Parameters}
\label{sec:simu:est:bt}

Here, we evaluate the performance of the proposed sampling estimators in estimating $\sbf\beta_0$, $\sbf X\sbf\beta_0$, and $\sbf X^T\sbf X\sbf\beta_0$. 
Under the simulation settings of Section~\ref{simu:setting}, we generated 100 replicates of \underline{MN}, \underline{T3}, \underline{LN}, and \underline{T1} data. 
We applied IC, RL, PL, SLEV (with $\lambda=0.9$ here and after), and BLEV to each replicated dataset to obtain sampling estimates at sample sizes $r= 100,  200,  500,  700, 1000$. 
Then, we calculated the squared bias and variance for each~method. 

\begin{figure}[t] 
\centering
\includegraphics[scale=0.5]{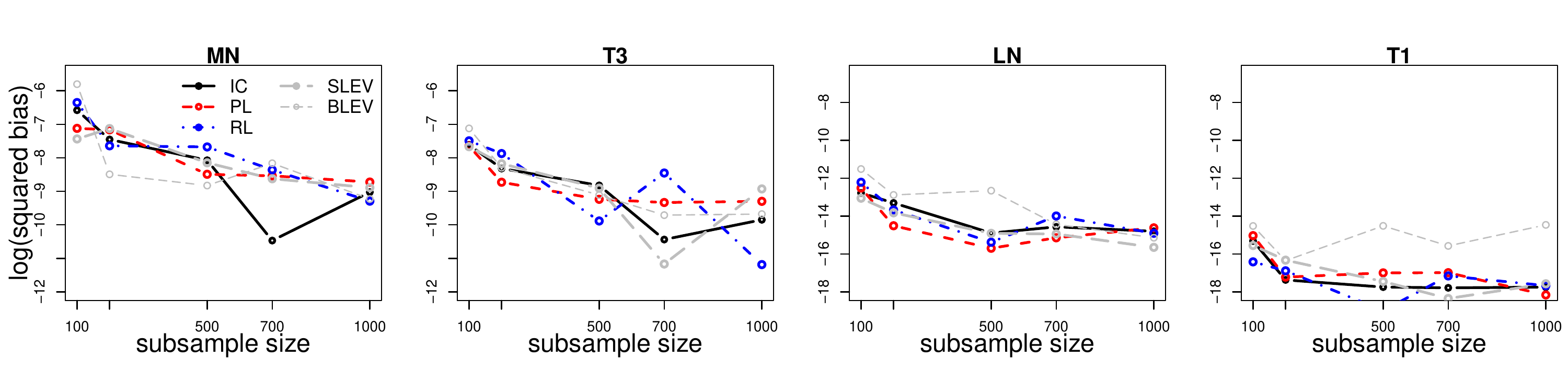}
\includegraphics[scale=0.5]{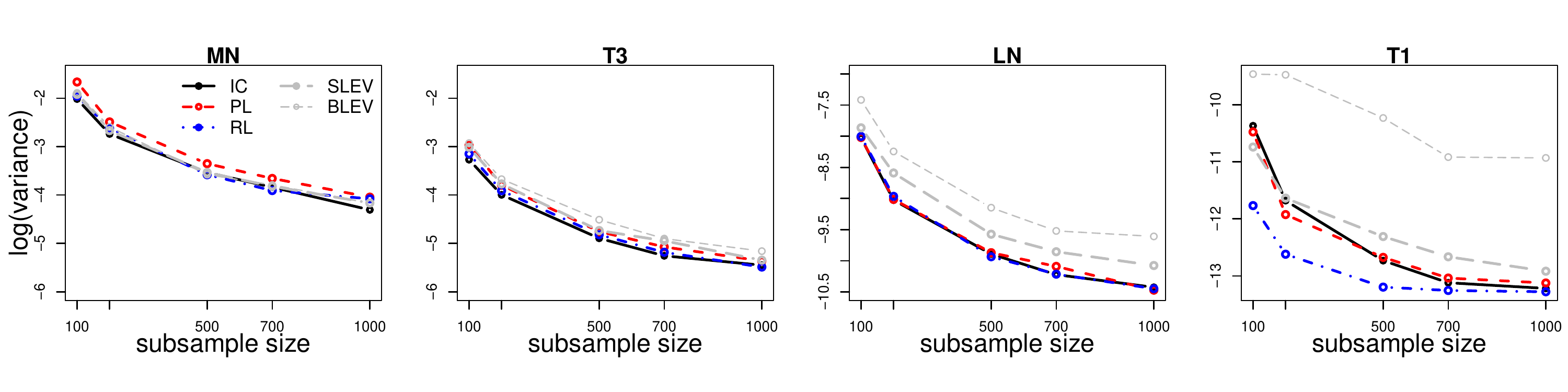}
\caption{Squared biases (first row) and variances (second row) of IC, RL, PL, SLEV, and BLEV estimates in estimating $\sbf\beta_0$ (in log scale) at different sample sizes. }
\label{fig:uncond:bias:var:bt}
\end{figure}

In Figure~\ref{fig:uncond:bias:var:bt}, we plot the squared biases (first row) and the variances (second row) (in log scale) for IC, RL, PL, SLEV, and BLEV estimates in estimating $\sbf\beta_0$ in \underline{MN}, \underline{T3}, \underline{LN}, and \underline{T1}.
First, both the squared biases and the variances show decreasing patterns as $r$ increases. 
The squared biases of different methods are similar to each other and are much smaller than the corresponding variances. 
These observations are expected, since Theorem~\ref{thm:beta_tilde-beta} states that the RandNLA estimators are asymptotically unbiased and consistent estimators of $\sbf\beta_0$. 
Second, the variances of estimates using IC, whose sampling probabilities minimize $AMSE(\tilde{\sbf\beta}; {\sbf\beta}_{0})$, are slightly smaller than the variances of estimates using other methods in \underline{MN} and \underline{T3}, at most sample sizes. 
The variances of estimates using IC, RL, and PL are all smaller than those of BLEV and SLEV estimates in \underline{T3}.
As mentioned in the discussion of Figure~\ref{fig:prob}, the larger variances of BLEV estimates are caused by the extremely small sampling probabilities in BLEV. 
Taking a weighted average of the sampling probability distribution of BLEV and that of UNIF shows a beneficial effect on the variances for SLEV estimators.
However, the variances of SLEV estimators are still larger than those of IC in \underline{T3}, \underline{LN}, and \underline{T1} at larger sample sizes. 
Third, for \underline{T1}, despite the violation of the regularity condition in Theorem~\ref{thm:beta_tilde-beta}, our proposed estimators IC, RL, and PL still outperform BLEV and SLEV in terms of variances, when sample size is greater than 200. 
Fourth, the squared biases and variances of all estimates get smaller from left panels to right panels.

For estimating $\mbf{Y}$ and $\mbf{X}^T\mbf{X}\sbf{\beta}_{0}$, the biases of all sampling estimators are very similar to each other and are much smaller than the corresponding variances. 
This observation is consistent with what we observed in 
estimating $\sbf{\beta}_{0}$ in Figure~\ref{fig:uncond:bias:var:bt}. We thus only present the variances of  IC, RL, PL, SLEV, and BLEV estimates in estimating $\mbf{Y}$ and $\mbf{X}^T\mbf{X}\sbf{\beta}_{0}$ at different sample sizes in Figure~\ref{fig:uncond:bias:var:yhat} and Figure~\ref{fig:uncond:bias:var:Lbt}. 
As shown, the variances of the estimates for estimating both $\mbf{Y}$ and $\mbf{X}^T\mbf{X}\sbf{\beta}_{0}$, using PL, IC, and RL, are smaller than the variances of estimates using BLEV and SLEV in \underline{T3} and \underline{LN}, at most sample sizes.

\begin{figure}[t] 
\centering
\includegraphics[scale=0.5]{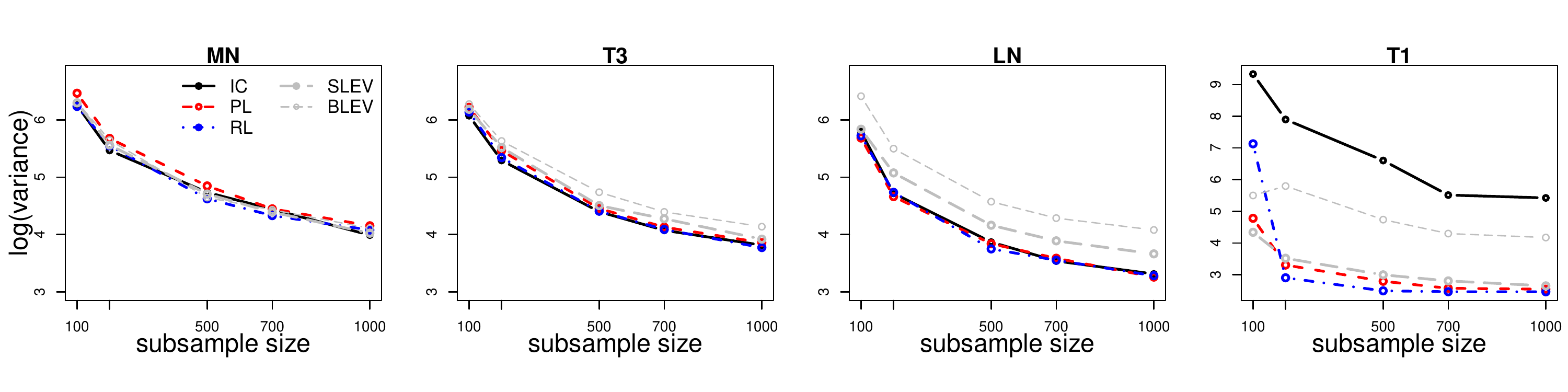}
\caption{The variances of  IC, RL, PL, SLEV, and BLEV estimates in predicting $\sbf{Y}$ (in log scale) at different sample sizes. }
\label{fig:uncond:bias:var:yhat}
\end{figure}

\begin{figure}[t] 
\centering
\includegraphics[scale=0.5]{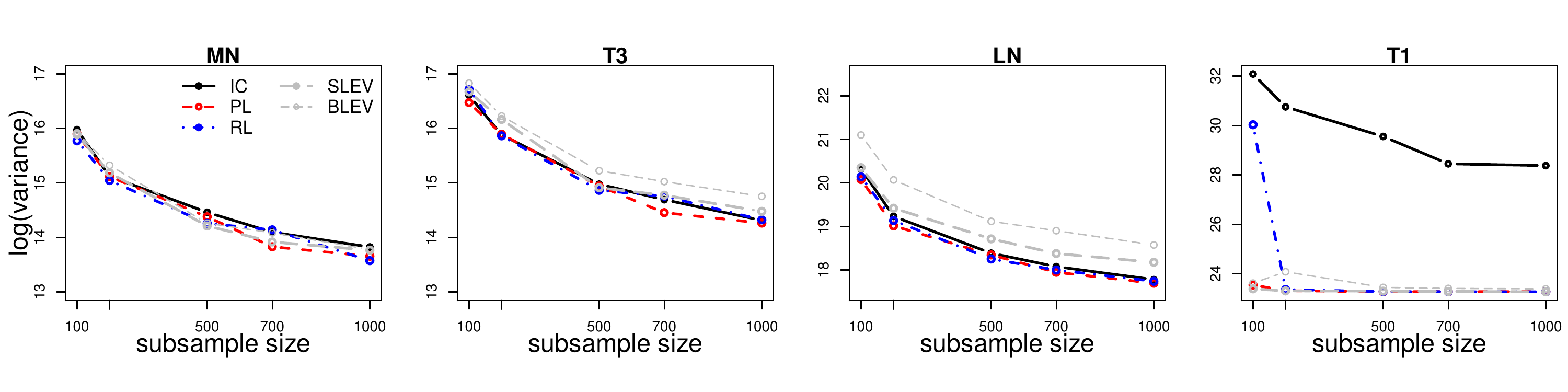}
\caption{The variances of  IC, RL, PL, SLEV, and BLEV estimates in estimating $\mbf{X}^T\mbf{X}\sbf{\beta}_{0}$ (in log scale) at different sample sizes. }
\label{fig:uncond:bias:var:Lbt}
\end{figure}

\subsection{Sampling Estimators for Approximating the Full Sample OLS Estimate}
\label{sec:simu:app:betols}

Here, we evaluate the performance of the proposed sampling estimators for approximating $\hat{\sbf\beta}_{OLS}$, $\sbf X \hat{\sbf\beta}_{OLS}$, and $\sbf X^T \sbf X \hat{\sbf\beta}_{OLS}$. 
Under the simulation settings of Section~\ref{simu:setting}, we generated four datasets without replicates from \underline{MN}, \underline{T3}, \underline{LN}, and \underline{T1}, respectively. 
For each dataset, the full sample OLS estimate was calculated. 
We set samples sizes at $r=100, 200, 500, 700, 1000$. 
We repeatedly applied ICNLEV, RLNLEV, PLNLEV, SLEV, and BLEV methods $100$ times at each sample size to get sampling estimates $\tilde{\sbf{\beta}}_b$, where $b=1,\ldots, 100$. 
Using these estimates, we calculated the squared bias and variance for each method for approximating $\hat{\sbf{\beta}}_{OLS}$. 

\begin{figure}[t] 
\centering
\includegraphics[scale=0.5]{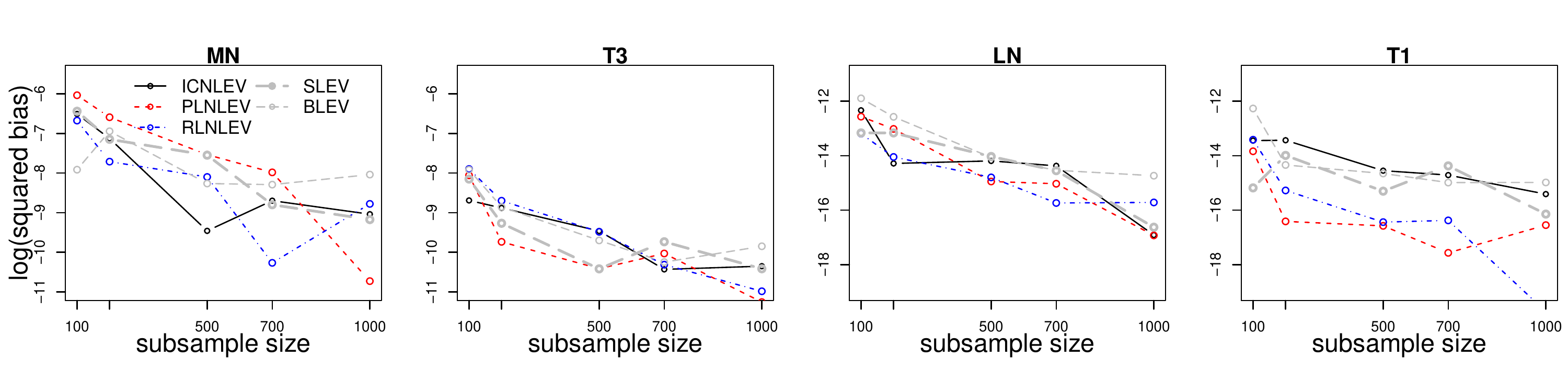}
\includegraphics[scale=0.5]{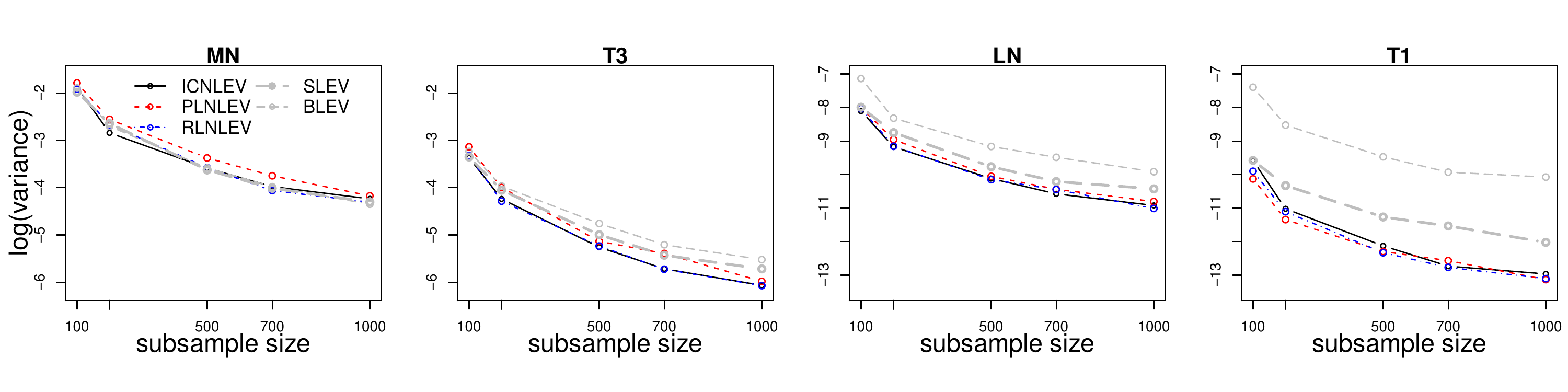}
\caption{Squared biases (first row) and variances (second row) of ICNLEV, RLNLEV, PLNLEV, SLEV, and BLEV estimates in approximating $\hat{\sbf{\beta}}_{OLS}$ (in log scale) at different sample~sizes. 
}
\label{fig:cond:bias:var:bt}
\end{figure}

In Figure~\ref{fig:cond:bias:var:bt}, we plot the squared biases and variances (in log scale) for ICNLEV, RLNLEV, PLNLEV, SLEV, and BLEV estimates for approximating $\hat{\sbf{\beta}}_{OLS}$ at different sample sizes in all datasets. 
Several observations are worth noting in Figure~\ref{fig:cond:bias:var:bt}. 
First, the squared biases are negligible compared to the corresponding variances. 
For all sampling methods, both the squared biases and the variances decrease as sample size increases. 
These observations are in agreement with Theorem~\ref{thm:beta_tilde-beta_hat}, which states that the sampling estimators are asymptotically unbiased estimators of $\hat{\sbf\beta}_{OLS}$, provided that the regularity conditions are satisfied. 
Second, the variances of estimates using ICNLEV 
and RLNLEV are slightly smaller than the variances of estimates using other methods in \underline{T3} and \underline{LN} at most sample sizes. The variances of estimates using ICNLEV, RLNLEV, and PLNLEV are consistently smaller than those of SLEV and BLEV in \underline{LN} and \underline{T1}. 
Third, all sampling estimators perform better in \underline{LN} and \underline{T1} than in \underline{T3} and \underline{MN}, i.e., the squared biases and variances of all estimates in \underline{LN} and \underline{T1} are smaller than those in \underline{T3} and \underline{MN}. 

\begin{figure}[t] 
\centering
\includegraphics[scale=0.5]{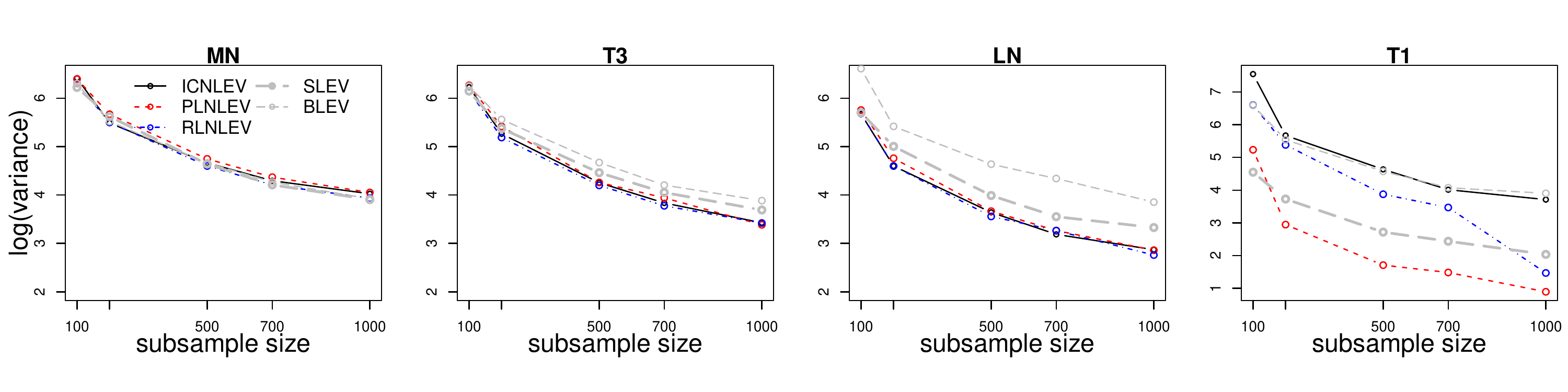}
\caption{The variances of ICNLEV, RLNLEV, PLNLEV, SLEV, and BLEV estimates in approximating $\hat{\mbf{Y}}_{OLS}(=\mbf X\hat{\sbf{\beta}}_{OLS})$ (in log scale) at different sample sizes. 
}
\label{fig:cond:bias:var:yhat}
\end{figure}

To examine the performance of the RandNLA sampling estimators for approximating $\hat{\mbf{Y}}_{OLS}(=\mbf X\hat{\sbf{\beta}}_{OLS})$, we plot the variances (in log scale) of $\mbf X\tilde{\sbf{\beta}}_{b}$, at different sample sizes, for all sampling estimators in Figure~\ref{fig:cond:bias:var:yhat}. 
The variances of estimates using RLNLEV, whose sampling probabilities minimize $EAMSE(\mbf{X}\tilde{\sbf\beta};\mbf{X}\hat{\sbf\beta}_{OLS})$, are slightly smaller than those of estimates using other methods at all sample sizes in \underline{T3} and at most sample sizes in \underline{LN}. 

To assess the performance of the RandNLA sampling estimators for approximating $\mbf X^T\mbf X\hat{\sbf{\beta}}_{OLS}$, we plot the variances (in log scale) of $\mbf{X}^T\mbf{X}\tilde{\sbf{\beta}}_{b}$, at different sample sizes, for all sampling estimators in Figure~\ref{fig:cond:bias:var:Lbt}. 
For all estimators, the variances decrease as the sample size increases. 
Also, in \underline{T3}, the variances of estimates using PLNLEV, whose sampling probabilities minimize $EAMSE(\mbf{X}^{T}\mbf{X}\tilde{\sbf\beta}; \mbf{X}^{T}\mbf{X}\hat{\sbf\beta}_{OLS})$  are smaller than the variances of estimates using other methods at most sample sizes. In this case, despite the violation of the conditions for the proper definition of EAMSE in \underline{T1}, the variances of PLNLEV estimates are still the smallest, when sample sizes are greater than 200. 

\begin{figure}[t] 
\centering
\includegraphics[scale=0.5]{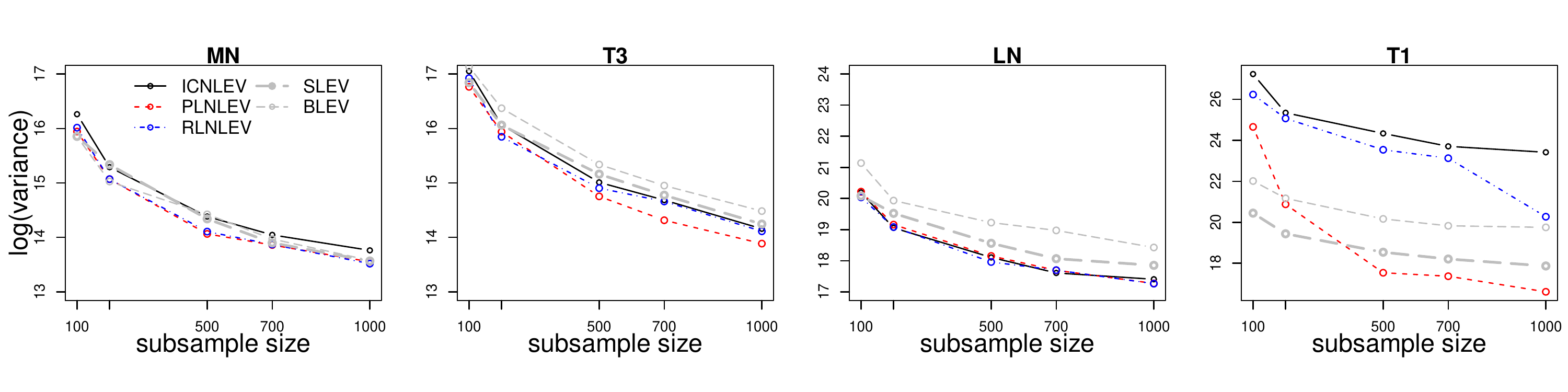}
\caption{The variances of ICNLEV, RLNLEV, PLNLEV, SLEV, and BLEV estimates in approximating $\mbf{X}^T\mbf{X}\hat{\sbf{\beta}}_{OLS}$ (in log scale)  at different sample sizes. 
}
\label{fig:cond:bias:var:Lbt}
\end{figure}

\subsection{Flight Delay Dataset}

Here, we evaluate the performance of the sampling estimators on a flight delay dataset we compiled from the website of the US Department of Transportation.%
\footnote{U. S. Bureau of Transportation Statistics. Rita airline delay data was downloaded from: \url{https://www.transtats.bts.gov/DL_SelectFields.asp?Table_ID=236}.  }
The dataset contains records of  $3,274,894$ US domestic flights during weekdays from Mondays to Thursdays in 2017. 
There are five variables for each flight record: arrival delay (difference in minutes between scheduled and actual arrival time, and early arrivals show negative numbers), arrival taxi in time (in minutes), departure taxi out time (in minutes), departure delays (difference in minutes between scheduled and actual departure time, and early departures show negative numbers), and computer reservation system based elapsed time of the flight (in minutes; a measure for the distance of the flight).
We are interested in predicting the arrival delay of each flight using the rest of the variables. 
We fitted  Model~(\ref{linreg-matrix}), with the response being flight arrival delay. 
In addition to using the four variables (other than arrival delay) in our dataset as linear predictors, we also included their quadratic and all pairwise interaction terms. We thus have 14 predictors. 
Considering the large number of flights, we use the sampling methods to approximate the full sample OLS estimate. 

\begin{figure}[t] 
\centering
\includegraphics[scale=0.26]{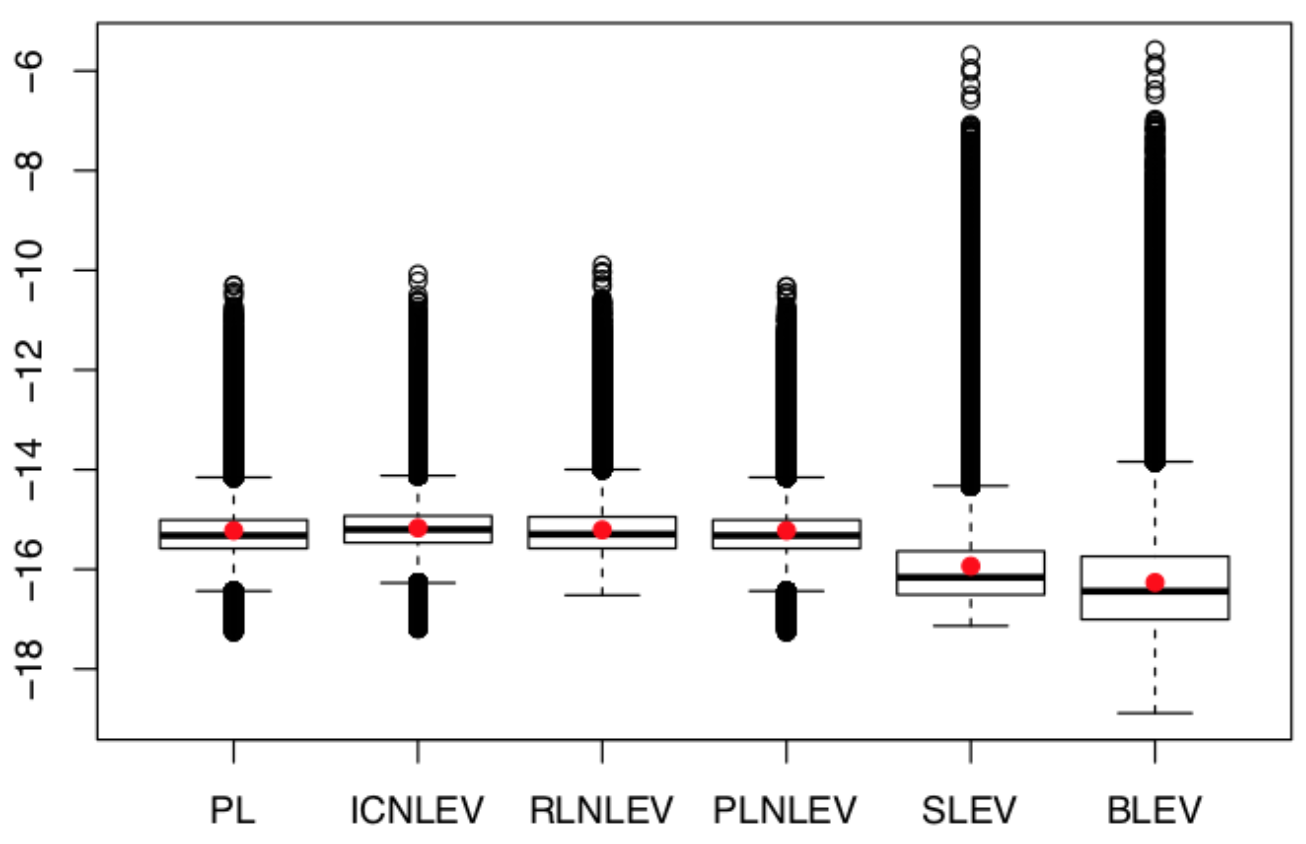}
\includegraphics[scale=0.32]{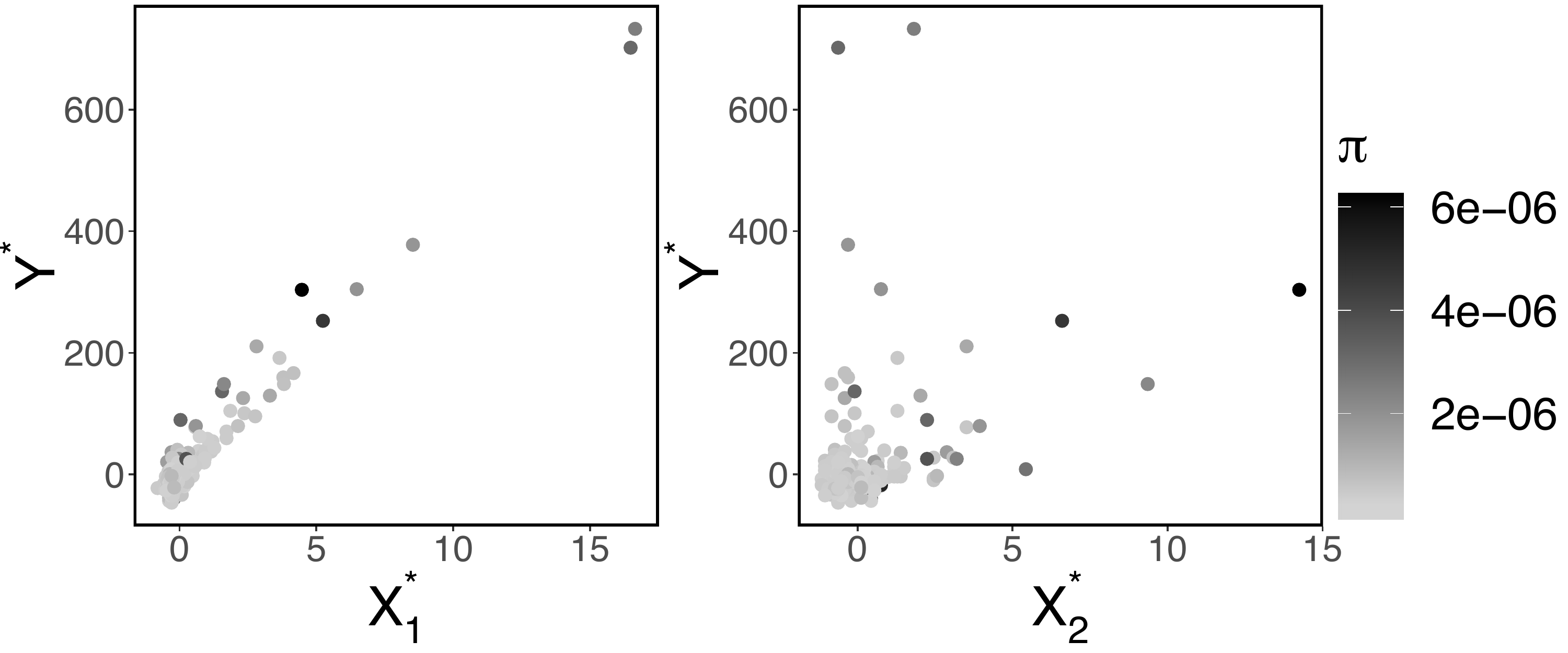}
\caption{Flight delay dataset. Left: the box plots of sampling probabilities (in log scale) of all data points in PL, ICNLEV, RLNLEV, PLNLEV, and BLEV. Middle and Right: the scatter plots of the $200$ sampled response vector (ARRIVAL$\_$DELAY) and two predictors (DEPARTURE$\_$DELAY and TAXI$\_$OUT) using the ICNLEV sampling probability distribution.  }
\label{fig:air:full}
\end{figure}
 
\begin{figure}[t] 
\centering
\includegraphics[scale=0.45]{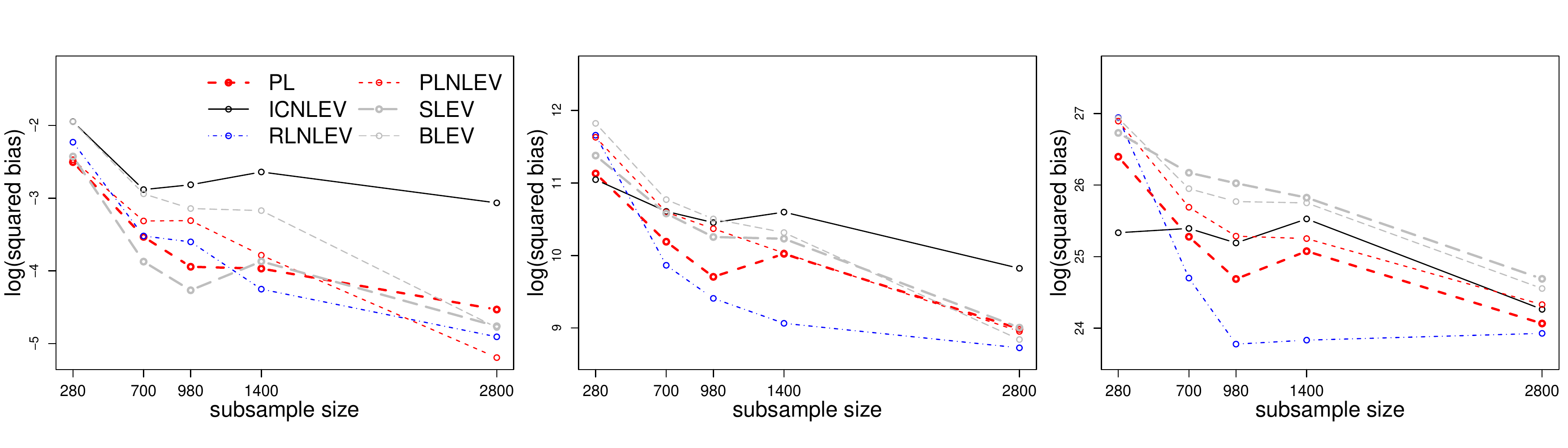}
\includegraphics[scale=0.45]{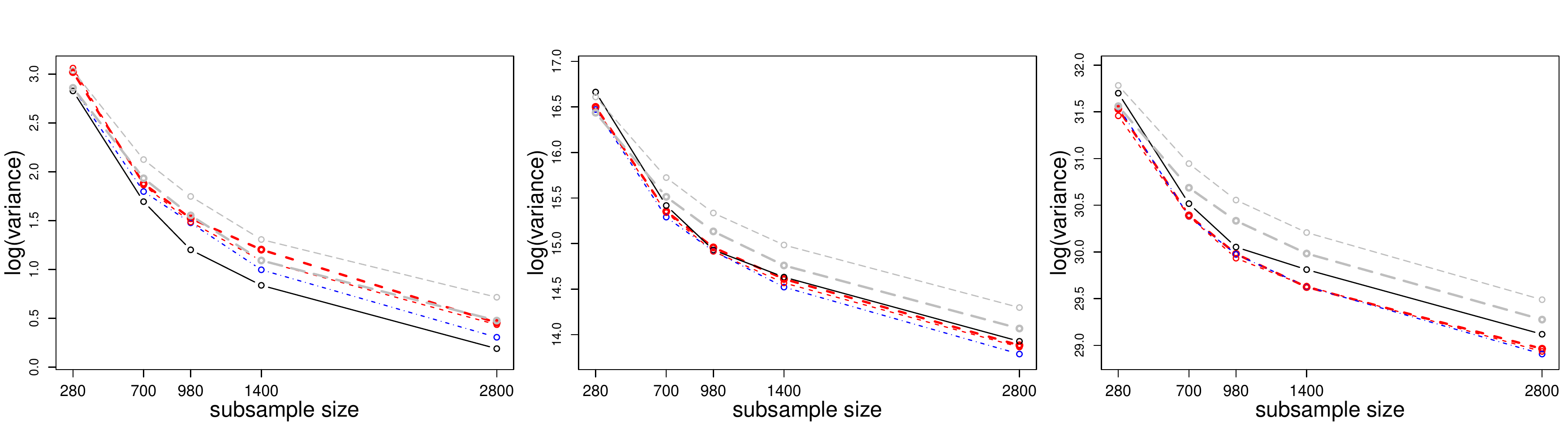}
\caption{Squared biases (first row) and variances (second row) of PL, ICNLEV, RLNLEV, PLNLEV, SLEV, and BLEV estimates for approximating $\hat{\sbf{\beta}}_{OLS}$ (first column), $\hat{\sbf{Y}}_{OLS}$ (second column) and $\mbf{X}^T\mbf{X}\hat{\sbf\beta}_{OLS}$ (third column) (in log scale) at different sample sizes for Airline Delay~data.}
\label{fig:air:bt:yhat}
\end{figure}

In the left panel of Figure~\ref{fig:air:full}, we present the box plots of sampling probabilities (in log scale) of all data points in PL, ICNLEV, RLNLEV, PLNLEV, SLEV, and BLEV. 
Observe that the sampling probability distributions are right-skewed, similar to those in Figure~\ref{fig:prob} in the simulation study. 
Using the sampling probability distribution in ICNLEV, we took a sample of size 200 from the full data. 
The middle and right panels in Figure~\ref{fig:air:full} are the scatter plots of the sampled response and the first two predictors, respectively. 
These scatter plots provide a visual sketch of the full sample data. 

We repeatedly applied the PL, ICNLEV, IC, PLNLEV, SLEV, and BLEV methods to this dataset for 100 times at sample size $r=20p,50p,70p,100p,200p$, where $p=14$. 
We calculated the squared bias and variance of the resulting estimates in approximating $\hat{\sbf\beta}_{OLS}$, $\hat{\sbf Y}_{OLS}$ and $\mbf{X}^T\mbf{X}\hat{\sbf\beta}_{OLS}$, for each method. 
The results are summarized in Figure~\ref{fig:air:bt:yhat}. 
Observe that the squared biases of all methods are all much smaller than the corresponding variances for all methods at all sample sizes. 
For approximating $\hat{\sbf\beta}_{OLS}$, the ICNLEV estimates have the smallest variance consistently at all sample sizes among all estimators. 
For approximating $\hat{\sbf Y}_{OLS}$ and $\mbf{X}^T\mbf{X}\hat{\sbf\beta}_{OLS}$, the estimates using PLNLEV, PL, and RLNLEV are very similar to each other, and they have better performance in terms of variances at all sample sizes than those using BLEV and SLEV.

\subsection{``YearPredictionMSD'' Dataset}

Here, we evaluate the performance of the sampling estimators on the ``YearPredictionMSD'' dataset~\citep{Bertin-Mahieux2011}, which we downloaded from the UCI machine learning repository.%
\footnote{See \url{http://archive.ics.uci.edu/ml/}.} 
The dataset consists of records of 515,345 songs released between the year 1922 and 2011. 
For each song, multiple segments are taken, and each segment is characterized by 12 timbre features. These timbre features capture timbral characteristics, such as brightness and flatness, of each segment. The mean and variance of each timbre feature, as well as the covariances between every two timbre features, are calculated. 
Our primary interest for our analysis is to use all timbre feature information to predict the year of release. 
We fitted  Model~(\ref{linreg-matrix}),
where the response is the year (in log scale) of releasing of the song, and the predictors include all timbre~features. 

\begin{figure}[t] 
\centering
\includegraphics[scale=0.26]{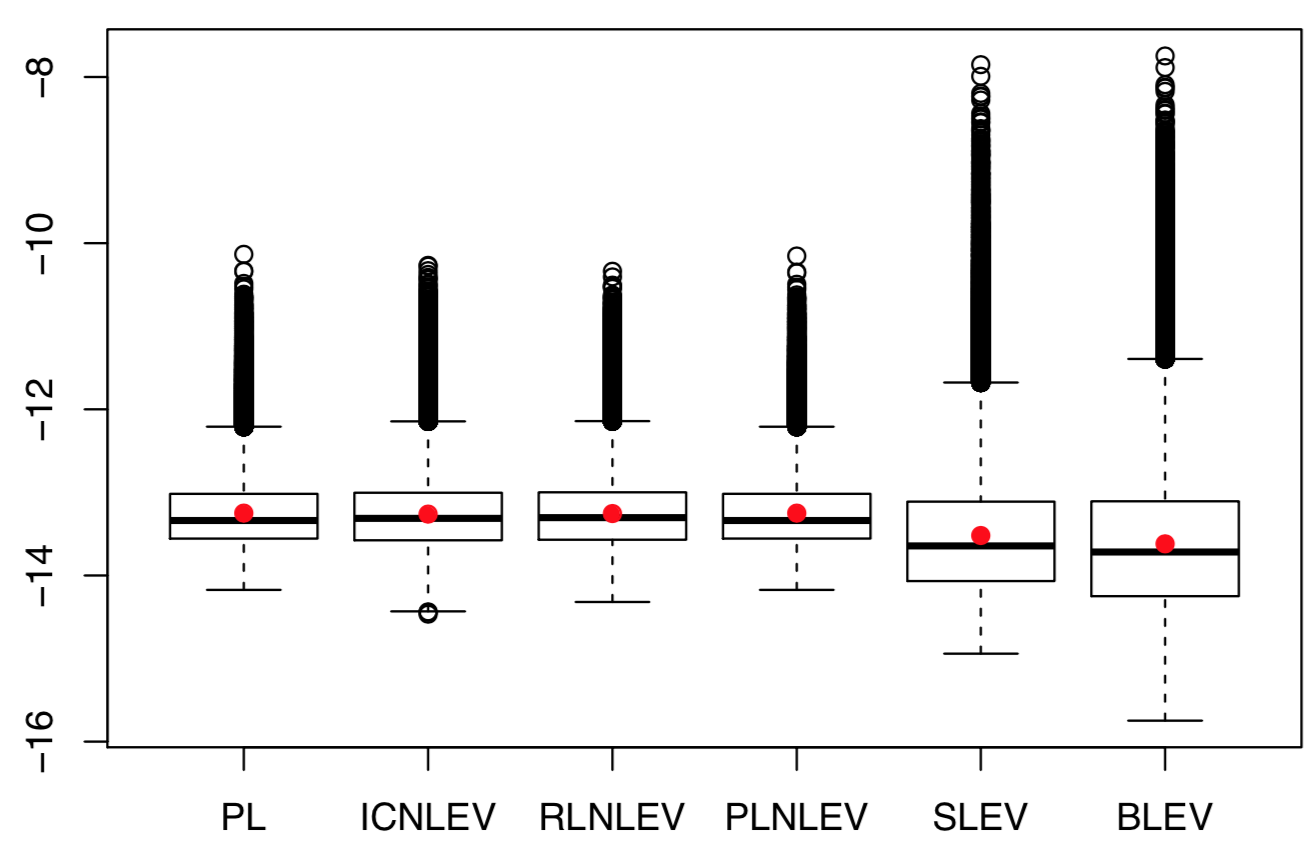}
\includegraphics[scale=0.32]{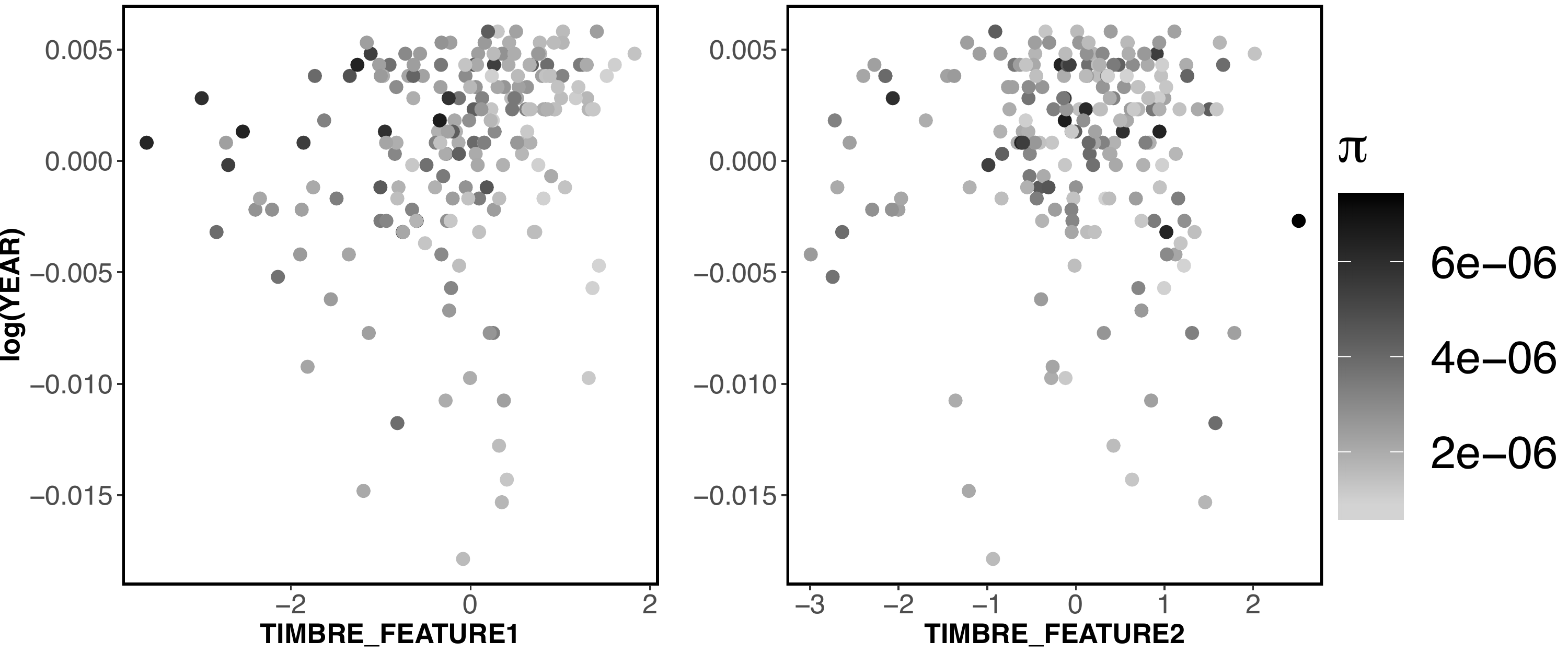}
\caption{ 
``YearPredictionMSD'' data. Left: the box plots of sampling probabilities (in log scale) for all data points for PL, ICNLEV, RLNLEV, PLNLEV, SLEV, and BLEV. A sample of size $200$ is taken from the full data using the sampling probabilities of ICNLEV. Middle and Right:  the scatter plots of sampled response  and two timbre feature predictors.}
\label{fig:music:2full}
\end{figure}

\begin{figure}[t] 
\centering
\includegraphics[scale=0.45]{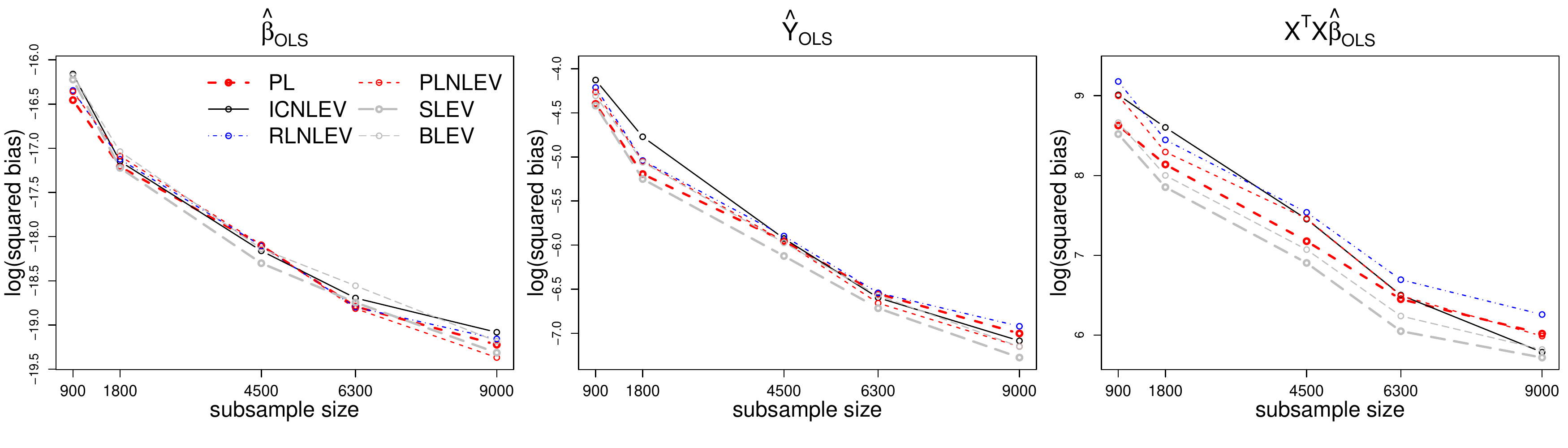}
\includegraphics[scale=0.45]{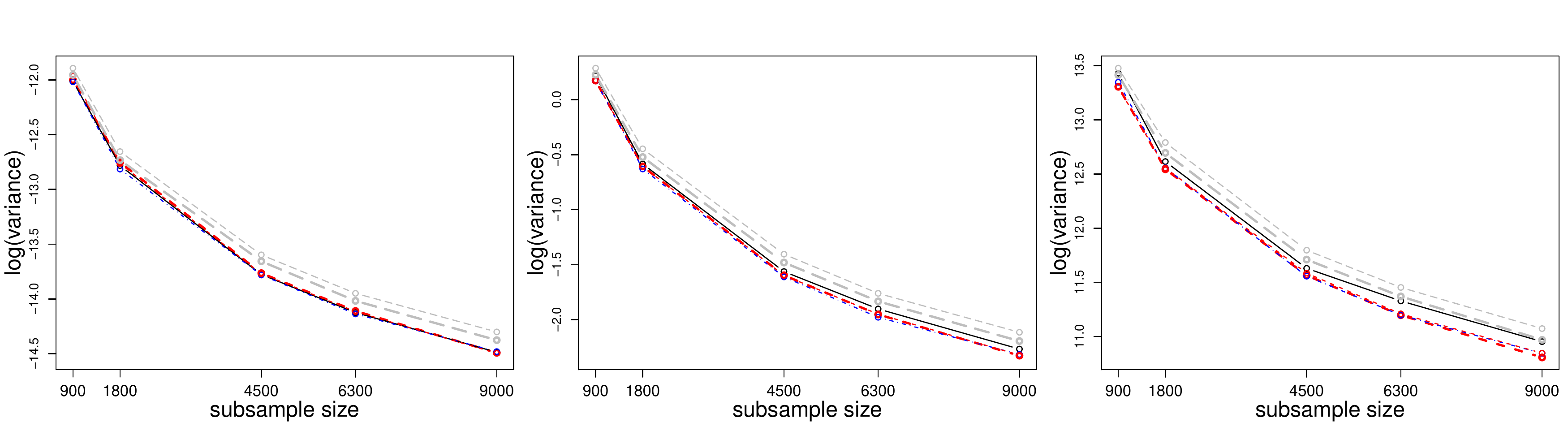}
\caption{Squared biases (first row) and variances (second row) of PL, ICNLEV, RLNLEV, PLNLEV, SLEV, and BLEV estimates for approximating $\hat{\sbf{\beta}}_{OLS}$ (first column), $\hat{\sbf{Y}}_{OLS}$ (second column), and $\mbf{X}^T\mbf{X}\hat{\sbf\beta}_{OLS}$ (third column) (in log scale) at different sample sizes for ``YearPredictionMSD'' data.  }
\label{fig:music:beta:yhat}
\end{figure}

In the left panel of Figure~\ref{fig:music:2full}, we present the box plots of sampling probabilities (in log scale) of all data points in PL, ICNLEV, RLNLEV, PLNLEV, SLEV, and BLEV.
Inspecting the box plots reveals that all sampling distributions are right-skewed and that the sampling distributions of SLEV and BLEV are much more dispersed than those of other estimators. 
Using the sampling probability distribution in ICNLEV, we took a sample of size 200 from the full data. 
The middle and right panels of Figure~\ref{fig:music:2full} are the scatter plots of the sampled response and two timbre features,~respectively. 

We repeatedly applied the ICNLEV, RLNLEV, PLNLEV, SLEV, and BLEV methods to the dataset for 100 times at sample sizes $r=10p,20p,50p,70p,100p$, where $p=90$. 
In Figure~\ref{fig:music:beta:yhat}, we plot the squared biases and the variances (in log scale) of the estimates for all weighted sampling methods for approximating $\hat{\sbf{\beta}}_{OLS}$, $\hat{\sbf{Y}}_{OLS}$, and $\mbf{X}^T\mbf{X}\hat{\sbf{\beta}}_{OLS}$.
For all three scenarios, the squared biases are much smaller than the corresponding variances, for all methods at all sample sizes. 
For approximating $\hat{\sbf{\beta}}_{OLS}$, the variances of ICNLEV, RLNLEV, and PLNLEV estimates are comparable to each other and consistently smaller than those of SLEV and BLEV estimates at all sample sizes. 
For approximating $\hat{\sbf{Y}}_{OLS}$ and $\mbf{X}^T\mbf{X}\hat{\sbf{\beta}}_{OLS}$, the variances of PLNLEV and PL estimates are consistently smaller than those of other estimates.

\section{Conclusion}
\label{sec:conclusion}

We have studied the asymptotic properties of RandNLA sampling estimators in LS linear regression models. 
We showed that under certain regularity conditions on the data distributions and sampling probability distributions, the sampling estimators are asymptotically normally distributed.  
Moreover, the sampling estimators are asymptotically unbiased for approximating the full sample OLS estimate and for estimating true coefficients.
Based on these asymptotic results, we proposed optimality criteria to assess the performance of the sampling estimators, based on AMSE and EAMSE.
In particular, we developed six sampling estimators, i.e.,  IC, RLEV, PL, ICNLEV, RLNLEV, and PLNLEV, for minimizing AMSE and EAMSE, under a variety of settings. 
These empirical results demonstrate that these new sampling estimators outperform the conventional ones in the literature.
For generalization, depending on the application, one may consider criteria other than AMSE and EAMSE. 
For example, when hypothesis testing problems are of primary interest, the power of the test is a more reasonable choice to serve as a criterion.
Developing scalable sampling methods to optimize criteria such as this are of interest. 

\section*{Acknowledgment}
We would like to thank 
Shusen Wang for providing constructive comments on an earlier version of this paper
and 
Bin Yu for helpful discussions.
PM, XZ, and XX 
acknowledge NSF and NIH for providing partial support of this work. 
MWM 
acknowledges 
ARO, DARPA, NSF, and ONR for providing partial support of this work.


\begin{thebibliography}{}

\bibitem[\protect\citeauthoryear{Avron, Maymounkov, and Toledo}{Avron
  et~al.}{2010}]{AMT10}
Avron, H., P.~Maymounkov, and S.~Toledo (2010).
\newblock Blendenpik: Supercharging {LAPACK}'s least-squares solver.
\newblock {\em SIAM Journal on Scientific Computing\/}~{\em 32}, 1217--1236.

\bibitem[\protect\citeauthoryear{Bertin-Mahieux, Ellis, Whitman, and
  Lamere}{Bertin-Mahieux et~al.}{2011}]{Bertin-Mahieux2011}
Bertin-Mahieux, T., D.~P. Ellis, B.~Whitman, and P.~Lamere (2011).
\newblock The million song dataset.
\newblock In {\em {Proceedings of the 12th International Conference on Music
  Information Retrieval ({ISMIR} 2011)}}.

\bibitem[\protect\citeauthoryear{Billingsley}{Billingsley}{1995}]{billingsley2012probability}
Billingsley, P. (1995).
\newblock {\em Probability and Measure}.
\newblock Wiley Series in Probability and Mathematical Statistics. Wiley.

\bibitem[\protect\citeauthoryear{Chen, Varma, Singh, and
  Kova{\v{c}}cevi{\'c}}{Chen et~al.}{2016}]{chen2016statistical}
Chen, S., R.~Varma, A.~Singh, and J.~Kova{\v{c}}cevi{\'c} (2016).
\newblock A statistical perspective of sampling scores for linear regression.
\newblock In {\em Information Theory (ISIT), 2016 IEEE International
  Symposium}, pp.\  1556--1560. IEEE.

\bibitem[\protect\citeauthoryear{Cover and Thomas}{Cover and
  Thomas}{2006}]{Cover:2006:EIT:1146355}
Cover, T.~M. and J.~A. Thomas (2006).
\newblock {\em Elements of Information Theory (Wiley Series in
  Telecommunications and Signal Processing)}.
\newblock Wiley-Interscience.

\bibitem[\protect\citeauthoryear{Derezi{\'n}ski, Clarkson, Mahoney, and
  Warmuth}{Derezi{\'n}ski et~al.}{2019}]{DCMW19_TR}
Derezi{\'n}ski, M., K.~L. Clarkson, M.~W. Mahoney, and M.~K. Warmuth (2019).
\newblock Minimax experimental design: Bridging the gap between statistical and
  worst-case approaches to least squares regression.
\newblock Technical report.
\newblock Preprint: arXiv:1902.00995.

\bibitem[\protect\citeauthoryear{Drineas, Magdon-Ismail, Mahoney, and
  Woodruff}{Drineas et~al.}{2012}]{DMMW12_JMLR}
Drineas, P., M.~Magdon-Ismail, M.~W. Mahoney, and D.~P. Woodruff (2012).
\newblock Fast approximation of matrix coherence and statistical leverage.
\newblock {\em Journal of Machine Learning Research\/}~{\em 13}, 3475--3506.

\bibitem[\protect\citeauthoryear{Drineas and Mahoney}{Drineas and
  Mahoney}{2016}]{Drineas:Mahoney:2016:RRN:ACM}
Drineas, P. and M.~W. Mahoney (2016).
\newblock {RandNLA: Randomized Numerical Linear Algebra}.
\newblock {\em Communications of the ACM\/}~{\em 59\/}(6), 80--90.

\bibitem[\protect\citeauthoryear{Drineas and Mahoney}{Drineas and
  Mahoney}{2018}]{RandNLA_PCMIchapter_chapter}
Drineas, P. and M.~W. Mahoney (2018).
\newblock Lectures on randomized numerical linear algebra.
\newblock In M.~W. Mahoney, J.~C. Duchi, and A.~C. Gilbert (Eds.), {\em The
  Mathematics of Data}, IAS/Park City Mathematics Series, pp.\  1--48.
  AMS/IAS/SIAM.

\bibitem[\protect\citeauthoryear{Drineas, Mahoney, and Muthukrishnan}{Drineas
  et~al.}{2006a}]{DMM06}
Drineas, P., M.~W. Mahoney, and S.~Muthukrishnan (2006a).
\newblock Sampling algorithms for $\ell_2$ regression and applications.
\newblock In {\em Proceedings of the 17th Annual ACM-SIAM Symposium on Discrete
  Algorithms}, pp.\  1127--1136.

\bibitem[\protect\citeauthoryear{Drineas, Mahoney, and Muthukrishnan}{Drineas
  et~al.}{2006b}]{Drineas:06}
Drineas, P., M.~W. Mahoney, and S.~Muthukrishnan (2006b).
\newblock Sampling algorithms for $l_2$ regression and applications.
\newblock In {\em Proceedings of the 17th Annual ACM-SIAM Symposium on Discrete
  Algorithms}, pp.\  1127--1136.

\bibitem[\protect\citeauthoryear{Drineas, Mahoney, and Muthukrishnan}{Drineas
  et~al.}{2008}]{DMM08}
Drineas, P., M.~W. Mahoney, and S.~Muthukrishnan (2008).
\newblock Relative-error {CUR} matrix decomposition.
\newblock {\em SIAM Journal on Matrix Analysis and Applications\/}~{\em 30},
  844--881.

\bibitem[\protect\citeauthoryear{Halko, Martinsson, and Tropp}{Halko
  et~al.}{2011}]{HMT09_SIREV}
Halko, N., P.-G. Martinsson, and J.~A. Tropp (2011).
\newblock Finding structure with randomness: Probabilistic algorithms for
  constructing approximate matrix decompositions.
\newblock {\em SIAM Review\/}~{\em 53\/}(2), 217--288.

\bibitem[\protect\citeauthoryear{Hubbard and Hubbard}{Hubbard and
  Hubbard}{1999}]{hubbard1999vector}
Hubbard, J. and B.~Hubbard (1999).
\newblock {\em Vector Calculus, Linear Algebra, and Differential Forms: A
  Unified Approach}.
\newblock Prentice Hall.

\bibitem[\protect\citeauthoryear{Huber}{Huber}{1973}]{huber1973robust}
Huber, P.~J. (1973).
\newblock Robust regression: asymptotics, conjectures and monte carlo.
\newblock {\em The Annals of Statistics\/}, 799--821.

\bibitem[\protect\citeauthoryear{Lai, Robbins, and Wei}{Lai
  et~al.}{1978}]{lai1978PNAS}
Lai, T.~L., H.~Robbins, and C.~Z. Wei (1978).
\newblock Strong consistency of least squares estimates in multiple regression.
\newblock {\em Proceedings of the National Academy of Sciences\/}~{\em
  75\/}(7), 3034--3036.

\bibitem[\protect\citeauthoryear{Le~Cam}{Le~Cam}{1986}]{lecam1986asymptotic}
Le~Cam, L. (1986).
\newblock {\em Asymptotic Methods in Statistical Decision Theory}.
\newblock Springer-Verlag.

\bibitem[\protect\citeauthoryear{Lehmann and Romano}{Lehmann and
  Romano}{2006}]{lehmann2006testing}
Lehmann, E.~L. and J.~P. Romano (2006).
\newblock {\em Testing Statistical Hypotheses}.
\newblock Springer Science \& Business Media.

\bibitem[\protect\citeauthoryear{Ma, Mahoney, and Yu}{Ma et~al.}{2014}]{Ma:13}
Ma, P., M.~Mahoney, and B.~Yu (2014).
\newblock A statistical perspective on algorithmic leveraging.
\newblock In {\em Proceedings of the 31th ICML Conference}, pp.\  91--99.

\bibitem[\protect\citeauthoryear{Ma, Mahoney, and Yu}{Ma et~al.}{2015}]{Ma:15}
Ma, P., M.~Mahoney, and B.~Yu (2015).
\newblock A statistical perspective on algorithmic leveraging.
\newblock {\em Journal of Machine Learning Research\/}~{\em 16}, 861--911.

\bibitem[\protect\citeauthoryear{Ma, Zhang, Xing, Ma, and Mahoney}{Ma
  et~al.}{2020}]{asymptotic_RandNLA_estimators_CONF}
Ma, P., X.~Zhang, X.~Xing, J.~Ma, and M.~W. Mahoney (2020).
\newblock Asymptotic analysis of sampling estimators for randomized numerical
  linear algebra algorithms.
\newblock In {\em Proceedings of the 23rd International Workshop on Artificial
  Intelligence and Statistics}.

\bibitem[\protect\citeauthoryear{Mahoney}{Mahoney}{2011}]{Mah-mat-revBOOK}
Mahoney, M. (2011).
\newblock {\em Randomized Algorithms for Matrices and Data}.
\newblock Foundations and Trends in Machine Learning. Boston: NOW Publishers.
\newblock Also available at: arXiv:1104.5557.

\bibitem[\protect\citeauthoryear{Mahoney and Drineas}{Mahoney and
  Drineas}{2016}]{MD16_chapter}
Mahoney, M.~W. and P.~Drineas (2016).
\newblock Structural properties underlying high-quality randomized numerical
  linear algebra algorithms.
\newblock In P.~B\"{u}hlmann, P.~Drineas, M.~Kane, and M.~van~de Laan (Eds.),
  {\em Handbook of Big Data}, pp.\  137--154. CRC Press.

\bibitem[\protect\citeauthoryear{Meng, Saunders, and Mahoney}{Meng
  et~al.}{2014}]{meng2014lsrn}
Meng, X., M.~A. Saunders, and M.~W. Mahoney (2014).
\newblock {LSRN}: A parallel iterative solver for strongly over- or
  under-determined systems.
\newblock {\em SIAM Journal on Scientific Computing\/}~{\em 36\/}(2),
  C95--C118.

\bibitem[\protect\citeauthoryear{Morris}{Morris}{1975}]{cltformult1975}
Morris, C. (1975).
\newblock Central limit theorems for multinomial sums.
\newblock {\em Annals of Statistics\/}~{\em 3\/}(1), 165--188.

\bibitem[\protect\citeauthoryear{Pilanci and Wainwright}{Pilanci and
  Wainwright}{2016}]{JMLR:v17:14-460}
Pilanci, M. and M.~J. Wainwright (2016).
\newblock Iterative {H}essian sketch: Fast and accurate solution approximation
  for constrained least-squares.
\newblock {\em Journal of Machine Learning Research\/}~{\em 17\/}(53), 1--38.

\bibitem[\protect\citeauthoryear{Portnoy}{Portnoy}{1984}]{portnoy1984asymptoticI}
Portnoy, S. (1984).
\newblock Asymptotic behavior of {M}-estimators of $p$ regression parameters
  when $p^2/n$ is large. {I.} {C}onsistency.
\newblock {\em The Annals of Statistics\/}, 1298--1309.

\bibitem[\protect\citeauthoryear{Portnoy}{Portnoy}{1985}]{portnoy1985asymptoticII}
Portnoy, S. (1985).
\newblock Asymptotic behavior of {M} estimators of $p$ regression parameters
  when $p^2/n$ is large; {II.} {N}ormal approximation.
\newblock {\em The Annals of Statistics\/}, 1403--1417.

\bibitem[\protect\citeauthoryear{Raskutti and Mahoney}{Raskutti and
  Mahoney}{2015}]{raskutti:2015}
Raskutti, G. and M.~Mahoney (2015).
\newblock A statistical perspective on randomized sketching for ordinary
  least-squares.
\newblock In {\em Proceedings of the 32th ICML Conference}, pp.\  617--625.

\bibitem[\protect\citeauthoryear{Serfling}{Serfling}{2001}]{Serfling:01}
Serfling, R. (2001).
\newblock {\em Approximation Theorems of Mathematical Statistics}.
\newblock Wiley, New York.

\bibitem[\protect\citeauthoryear{Shao}{Shao}{2003}]{shao2003mathematical}
Shao, J. (2003).
\newblock {\em Mathematical Statistics}.
\newblock Springer Texts in Statistics. Springer Verlag.

\bibitem[\protect\citeauthoryear{Sheldon}{Sheldon}{2006}]{sheldon2006first}
Sheldon, R. (2006).
\newblock {\em A First Course in Probability\/} (7th ed.).
\newblock Pearson Education India.

\bibitem[\protect\citeauthoryear{Steck}{Steck}{1957}]{steck1957condlimit}
Steck, G.~P. (1957).
\newblock {\em Limit Theorems for Conditional Distributions}.
\newblock University of California Press, Berkeley.

\bibitem[\protect\citeauthoryear{Wang}{Wang}{2019}]{wang2019more}
Wang, H. (2019).
\newblock More efficient estimation for logistic regression with optimal
  subsamples.
\newblock {\em Journal of Machine Learning Research\/}~{\em 20\/}(132), 1--59.

\bibitem[\protect\citeauthoryear{Wang, Zhu, and Ma}{Wang
  et~al.}{2018}]{wang2018optimal}
Wang, H., R.~Zhu, and P.~Ma (2018).
\newblock Optimal subsampling for large sample logistic regression.
\newblock {\em Journal of the American Statistical Association\/}~{\em
  113\/}(522), 829--844.

\bibitem[\protect\citeauthoryear{Wang, Lee, Mahdavi, Kolar, and Srebro}{Wang
  et~al.}{2017}]{wang2017sketching}
Wang, J., J.~D. Lee, M.~Mahdavi, M.~Kolar, and N.~Srebro (2017).
\newblock Sketching meets random projection in the dual: A provable recovery
  algorithm for big and high-dimensional data.
\newblock {\em Electronic Journal of Statistics\/}~{\em 11\/}(2), 4896--4944.

\bibitem[\protect\citeauthoryear{Wang, Yu, and Singh}{Wang
  et~al.}{2017}]{2016arXiv160102068W}
Wang, Y., A.~W. Yu, and A.~Singh (2017).
\newblock On computationally tractable selection of experiments in
  measurement-constrained regression models.
\newblock {\em Journal of Machine Learning Research\/}~{\em 18\/}(143), 1--41.

\bibitem[\protect\citeauthoryear{Woodruff et~al.}{Woodruff
  et~al.}{2014}]{woodruff2014sketching}
Woodruff, D.~P. et~al. (2014).
\newblock Sketching as a tool for numerical linear algebra.
\newblock {\em Foundations and Trends{\textregistered} in Theoretical Computer
  Science\/}~{\em 10\/}(1--2), 1--157.

\bibitem[\protect\citeauthoryear{Yohai and Maronna}{Yohai and
  Maronna}{1979}]{yohai1979asymptotic}
Yohai, V.~J. and R.~A. Maronna (1979).
\newblock Asymptotic behavior of {M}-estimators for the linear model.
\newblock {\em The Annals of Statistics\/}, 258--268.

\bibitem[\protect\citeauthoryear{Zou}{Zou}{2006}]{zou:adaptive:2006}
Zou, H. (2006).
\newblock The adaptive lasso and its oracle properties.
\newblock {\em Journal of the American Statistical Association\/}~{\em
  101\/}(476), 1418--1429.

\end{thebibliography}

\appendix

\section{Proofs of Our Main Results}

In this Appendix, we collect the proofs of our main results.

\subsection{Notation and Technical  Preliminaries}

Let $K_i$ represent the number of times the $i^{th}$ observation is sampled. 
It is easy to see that $(K_1,\ldots, K_n)$ follows a multinomial distribution, Mult($r, \{\pi_i\}_{i=1}^n$), with sample size $r$, as the total number of trials. 
Define $\mbf{K} = \text{diag}\{K_i\}_{i=1}^n$, $\mbf{\Omega} = \text{diag}\{1/r\pi_i\}_{i=1}^n$, and $\mbf{W}=\mbf{\Omega} \mbf{K}$. 
For the $i^{th}$ diagonal element of matrix $\mbf{W}$, denoted as $W_i$, we have 
\begin{equation}
 \E(W_{i}) = 1,  \quad  \text{Var}(W_{i}) = \frac{(1-\pi_i)}{r\pi_i}, \quad  \text{Cov}(W_{i},W_{j}) = -\frac{1}{r}, \quad   i \neq j, \quad  i, j =1,\ldots, n. 
\end{equation}
Simple algebra yields that the sampling estimator of 
Eqn.~(\ref{wlsq-sample})
can be written as
\begin{equation}\label{eqn:wt:sub:trans}
 \tilde{\sbf\beta}=(\mbf{X}^{*T}\mbf{\Phi}^{*2}\mbf{X}^{*})^{-1}\mbf{X}^{*T}\mbf{\Phi}^{*2}\mbf{Y}^{*}=(\mbf{X}^T\mbf{W}\mbf{X})^{-1}\mbf{X}^T\mbf{W}\mbf{Y}.
\end{equation}

\noindent
\textbf{$O_p$ Notation.}
\begin{itshape}
The $O_p$ notation is the stochastic counterpart of the regular big-$O$ notation, i.e., it describes the limiting behavior of (or the order of) a sequence of random variables, rather than that of sequence of fixed numbers.

For a sequence of random variable $\{A_n\}$ and a sequence of constants $\{a_n\}$, the notation $A_n=O_p(a_n)$, means that $\{A_n/a_n\}$ is stochastically bounded (or bounded in probability). 
That is, for any $\epsilon>0$, 
\begin{equation}
\lim_{n\rightarrow\infty}P(|A_n/a_n|>\epsilon)=0.  
\end{equation}
More details and examples of this can be found in Section 1.2 of \citet{Serfling:01}.
\end{itshape}

\textbf{Remark.}
If 
Var($A_n$)=$O(n^{2\delta})$, where $\delta$ is a real number, then we have that $\{A_n/n^{\frac\delta2}\}$ is bounded in probability by Chebyshev's inequality. 
We write $A_n=O_p(n^{\delta})$

\textbf{Remark.}
Throughout this paper, for a matrix $\mbf A$, we write $\mbf A=O_p(n^\delta)$ to denote that all elements of $\mbf A$ are of the order of $O_p(n^\delta)$.

Other than in the statement and proof of Theorem~\ref{coro:diverg:p}, we assume that $p$ is fixed in all lemmas and theorems.
The Cramer-Wold Device and Lemma~\ref{lem:1} below govern the proofs for Theorem~\ref{thm:beta_tilde-beta} and Theorem~\ref{thm:beta_tilde-beta_hat}.

\noindent
\textbf{Cramer-Wold Device.}
\begin{itshape}
For random vectors $\mbf Z_n=(Z_{n1}, \ldots, Z_{np})^T$ and $\mbf Z=(Z_{1}, \ldots, Z_{p})^T$, a necessary and sufficient condition for $\mbf Z_n\stackrel{d}{\rightarrow}\mbf Z$ is that $\sbf b^T \mbf Z_n\stackrel{d}{\rightarrow}\sbf b^T\mbf Z$ as $n\rightarrow\infty$, for each $\sbf b\in \mathbb{R}^p$.
\end{itshape}

\textbf{Remark.}
To derive the asymptotic distribution for the sampling estimator $\tilde{\sbf\beta}$ in (\ref{eqn:wt:sub:trans}), which a vector of random variables, we use the Cramer-Wold device to reduce the derivation of the asymptotic distribution for \emph{vectors} to the usual \emph{scalar} case. 
For more details about the Cramer-Wold device, see Section 29 of \citet{billingsley2012probability}.

\noindent
\textbf{Convergence of Geometric Series of Matrices.}
\begin{itshape}
Let $\mbf A$ be an $n\times n$ square matrix. We use $\rho(\mbf A)$ to denote the spectral radius of matrix $\mbf A$, i.e., $\rho(\mbf A)=\max \left\{|\lambda _{1}|,\dotsc ,|\lambda _{n}|\right\}$, where $\lambda _{1},\dotsc ,\lambda _{n}$ are the eigenvalues of matrix $\mbf A$. If $\rho(\mbf A)<1$, then $(\mbf I-\mbf A)$ is invertible, and the series 
\begin{eqnarray} 
\mbf S&=&\nonumber \mbf I+ \mbf A+\mbf A^2+\ldots
\end{eqnarray}
converges to $(\mbf I-\mbf A)^{-1}$. 
\end{itshape}

\textbf{Remark.}
The convergence of geometric series of matrices will be used in the proof of Lemma~\ref{lem:1} below. 
For more details and a proof of this result, see Section 1.5 of \citet{hubbard1999vector}.

\begin{lem}\label{lem:1}
Assume that $0<\pi_i<1$, for $i=1,\ldots,n$. 
If 
 \begin{eqnarray}
 (\mbf{X}^T\mbf{X})^{-1}\mbf{X}^T(\mbf{W}-\mbf{I})\mbf{X}&=&\label{eqn:delta:lev}O_p\left(\frac{1}{r^{\frac\delta2}}\right), 
 \end{eqnarray} 
where $\delta$ is a positive constant, then the weighted sample estimator in (\ref{eqn:wt:sub:trans}) can be written as
 \begin{equation} \label{order}
 \tilde{\sbf\beta} =  \hat{\sbf\beta}_{OLS}+(\mbf{X}^T\mbf{X})^{-1}\mbf{X}^T\mbf{W}\sbf e+O_p(1/r^{\delta}), 
\end{equation}
where $\sbf e=\mbf Y- \mbf{X}\hat{\sbf\beta}_{OLS}$. 
\end{lem}
\begin{proof}
By (\ref{eqn:delta:lev}), we have
\begin{eqnarray}
((\mbf{X}^T\mbf{X})^{-1}\mbf{X}^T(\mbf{W}-\mbf{I})\mbf{X})^2&=&\label{eqn:A2} O_p(1/r^{\delta}).
\end{eqnarray}
Therefore, 
\begin{equation}\label{eqn:goto0}
[\mbf{I}+(\mbf{X}^T\mbf{X})^{-1}\mbf{X}^T(\mbf{W}-\mbf{I})\mbf{X}]^{-1}=\mbf{I}-(\mbf{X}^T\mbf{X})^{-1}\mbf{X}^T(\mbf{W}-\mbf{I})\mbf{X}+O_p(1/r^{\delta}).
\end{equation}

\noindent
Note that $(\mbf{X}^T\mbf{X})^{-1}\mbf{X}^T(\mbf{W}-\mbf{I})\mbf{X}$, $(\mbf{X}^T\mbf{X})^{-1}\mbf{X}^T(\mbf{W}-\mbf{I})\mbf{Y}$, and $(\mbf{X}^T\mbf{X})^{-1}\mbf{X}^T(\mbf{W}-\mbf{I})\sbf{e}$ are of the same order since the variances of $\mbf{Y}$ and $\sbf{e}$ are both bounded.
We expand (\ref{eqn:wt:sub:trans}) as follows: 
\begin{eqnarray}
 \tilde{\sbf\beta} &=&\nonumber (\mbf{X}^T\mbf{W}\mbf{X})^{-1}(\mbf{X}^T\mbf{W}\mbf{Y})\\
 &=&\nonumber [\mbf{I}+(\mbf{X}^T\mbf{X})^{-1}\mbf{X}^T(\mbf{W}-\mbf{I})\mbf{X}]^{-1}(\mbf{X}^T\mbf{X})^{-1}(\mbf{X}^T\mbf{W}\mbf{Y})\\
 &=&\label{eqn:talor:inverse} [\mbf{I}-(\mbf{X}^T\mbf{X})^{-1}\mbf{X}^T(\mbf{W}-\mbf{I})\mbf{X}+O_p(1/r^{\delta})](\mbf{X}^T\mbf{X})^{-1}(\mbf{X}^T\mbf{Y}+\mbf{X}^T(\mbf{W}-\mbf{I})\mbf{Y})\\
 &=&\nonumber [\mbf{I}-(\mbf{X}^T\mbf{X})^{-1}\mbf{X}^T(\mbf{W}-\mbf{I})\mbf{X}+O_p(1/r^{\delta})](\hat{\sbf\beta}_{OLS}+(\mbf{X}^T\mbf{X})^{-1}\mbf{X}^T(\mbf{W}-\mbf{I})\mbf{Y})\\
 &=&\nonumber \hat{\sbf\beta}_{OLS}+(\mbf{X}^T\mbf{X})^{-1}\mbf{X}^T(\mbf{W}-\mbf{I})\sbf e+O_p(1/r^{\delta})\\
  &=&\label{eqn:lem1:final} \hat{\sbf\beta}_{OLS}+(\mbf{X}^T\mbf{X})^{-1}\mbf{X}^T\mbf{W}\sbf e+O_p(1/r^{\delta}), 
\end{eqnarray}
where the expansion in (\ref{eqn:talor:inverse}) is by the convergence of geometric series of matrices and the assumption that $\delta>0$, and where the equality in (\ref{eqn:lem1:final}) holds since $\mbf{X}^T\sbf{e}=0$.
\end{proof}

\textbf{Remark.}
Lemma~\ref{lem:1} relates the sampling estimator $\tilde{\sbf\beta}$ to the quantity $\hat{\sbf\beta}_{OLS}$, with an order constraint on the residual term, i.e., $O_p(1/r^{\delta})$. 
In the application of Lemma~\ref{lem:1} to the proof of Theorem~\ref{thm:beta_tilde-beta} (asymptotic normality of $\tilde{\sbf\beta}$ in estimating $\sbf\beta_0$), we subtract $\sbf\beta_0$ from both sides of (\ref{eqn:lem1:final}) to relate $\tilde{\sbf\beta}$ to $\sbf\beta_0$. 
In the proof of Theorem~\ref{thm:beta_tilde-beta_hat}, Lemma~\ref{lem:1} is directly applied (asymptotic normality of $\tilde{\sbf\beta}$ in approximating $\hat{\sbf\beta}_{OLS}$). 

\textbf{Remark.}
The assumption that $\delta>0$ implies that $\rho((\mbf{X}^T\mbf{X})^{-1}\mbf{X}^T(\mbf{W}-\mbf{I})\mbf{X})\rightarrow 0$ as $r\rightarrow\infty$. 
By the convergence of geometric series of matrices, the inverse of $[\mbf{I}+(\mbf{X}^T\mbf{X})^{-1}\mbf{X}^T(\mbf{W}-\mbf{I})\mbf{X}]=\mbf{X}^T\mbf{W}\mbf{X}$ exists and the expansion in ($\ref{eqn:talor:inverse}$) is valid asymptotically. 
In the proof of Theorem~\ref{thm:beta_tilde-beta} and Theorem~\ref{thm:beta_tilde-beta_hat}, we will verify the condition in Lemma~\ref{lem:1}, i.e., that $\delta>0$.
The exact magnitude of $\delta$ depends on $(\mbf W-\mbf I)$, and it is different in Theorem~\ref{thm:beta_tilde-beta} and Theorem~\ref{thm:beta_tilde-beta_hat}.

In  Appendix~\ref{sec:proof:thm:beta_tilde-beta} and Appendix~\ref{sec:proof:thm:beta_tilde-betahat}, we present the proofs of Theorem~\ref{thm:beta_tilde-beta} and Theorem~\ref{thm:beta_tilde-beta_hat}, respectively. 
The proof of Theorem~\ref{thm:beta_tilde-beta} is much more complicated than that of Theorem~\ref{thm:beta_tilde-beta_hat}. 
In conditional inference of Theorem~\ref{thm:beta_tilde-beta_hat}, the data are given and the only randomness comes from sampling. 
However, in unconditional inference of Theorem~\ref{thm:beta_tilde-beta}, we consider both unobserved hypothetical data sampled from the underlying population as well as the sample sampled from observations. 
Thus, one more layer of randomness needs to take into account.

\subsection{Proof of Theorem~\ref{thm:beta_tilde-beta}}
\label{sec:proof:thm:beta_tilde-beta}

We start by establishing several preliminary technical lemmas, and then we will present the main proof of Theorem~\ref{thm:beta_tilde-beta}.

\subsubsection{Preliminary Material for the Proof of Theorem~\ref{thm:beta_tilde-beta}}

To facilitate the proof of Theorem~\ref{thm:beta_tilde-beta}, we first present the Hajek-Sidak CLT, as well as Lemma~\ref{lem:Us} and Lemma~\ref{lem:cond:Us}, as follows. 

\begin{thm}[Hajek-Sidak Central Limit Theorem]
Let $X_1, \ldots, X_n$ be independent and identically distributed (i.i.d.) random variables such that $\E(X_i)=\mu$ and $\vv(X_i)=\sigma^2$ are both finite. 
Define $T_n = d_1X_1+\ldots+d_nX_n$, then 
 \begin{eqnarray}
  \frac{T_n-\mu\sum_{i=1}^nd_i}{\sigma\sqrt{\sum_{i=1}^n d_i^2}}\stackrel{d}{\rightarrow}\text{N}\left( 0,1\right)   ,   && \label{thm:HSCLT}  
 \end{eqnarray}
 whenever the Noether condition,  
 \begin{eqnarray}
  \frac{\max_{1\le i\le n} d_i^2}{\sum_{i=1}^n d_i^2}\to 0\label{eqn:HSCLT:cond},  \quad\text{as }n\to\infty  ,
 \end{eqnarray}
 is satisfied. 
\end{thm}

\textbf{Remark.}
The Hajek-Sidak CLT 
is used in the proof of Lemma~\ref{lem:Us}.

\begin{lem}\label{lem:Us}
Define $ \mbf{U}=\text{diag}(U_1, \ldots, U_n)$ where independent random variables $U_i {\sim} \mbox{Poisson}(r\pi_i)$, for $i=1,\ldots, n$, and $\sbf\varepsilon=(\varepsilon_1,\ldots, \varepsilon_n)^T$, where $\varepsilon_i$s are independently and identically distributed with mean 0 and variance $\sigma^2$. 
If conditions  (A1) and (A2) in Theorem~\ref{thm:beta_tilde-beta} hold, 
then as $n\to\infty$,
\begin{eqnarray}\label{eqn:uconvU}
 \mbf{\Sigma}_{0}^{-\frac 12}(\mbf{X}^T\mbf{X})^{-1}\mbf{X}^T\mbf{\Omega} \mbf{U}\sbf\varepsilon&\stackrel{d}{\rightarrow}&\textbf{N}(\sbf 0, \mbf{I}_p), 
 \end{eqnarray}
 where $\sbf \Sigma_0$  and $\mbf{\Omega}$ are defined in Theorem~\ref{thm:beta_tilde-beta}.
\end{lem}
\begin{proof}

We derive the asymptotic normality of the random vector $(\mbf{X}^T\mbf{X})^{-1}\mbf{X}^T\mbf{\Omega} \mbf{U}\sbf\varepsilon$ using the Cramer-Wold device. 
For any nonzero constant vector $\sbf b\in \mathbb{R}^p$, 
we write
 \begin{equation}\label{eqn:cram:wold}
 \sbf b^T(\mbf{X}^T\mbf{X})^{-1}\mbf{X}^T\mbf{\Omega} \mbf{U}\sbf\varepsilon=\sum_{i=1}^n d_i \zeta_i,
 \end{equation}
  where $d_i = \sbf b^T(\mbf{X}^T\mbf{X})^{-1}\mbf x_i\frac{\sqrt{r\pi_i+r^2\pi_i^2}}{r\pi_i}$ and $\zeta_i=U_i \varepsilon_i/\sqrt{r\pi_i+r^2\pi_i^2}$, E$(\zeta_i)=0$, and Var$(\zeta_i)=\sigma^2$.  
  
Since Eqn. (\ref{eqn:cram:wold}) is a weighted average of independent random variables $\zeta_i$, it suffices to verify the Noether condition (\ref{eqn:HSCLT:cond}) of Hajek-Sidak CLT to show the asymptotic normality of $ \sbf b^T(\mbf{X}^T\mbf{X})^{-1}\mbf{X}^T\mbf{\Omega} \mbf{U}\sbf\varepsilon$.
For $d_i^2$, we have
  \begin{equation}\label{eqn:dmax}
  d_i^2 \le  \left(1+\frac{1}{r\pi_{min}}\right)(\sbf a^T \mbf x_i)^2   \le  \left(1+\frac{1}{r\pi_{min}}\right)\sbf a^T\sbf aM_x, 
  \end{equation}
where $\sbf a = (\mbf{X}^T\mbf{X})^{-1}\sbf b$, $M_x=\max\{\mbf x_i^T\mbf x_i\}_{i=1}^n$, and the last inequality is derived using the Cauchy-Schwarz inequality. 
Thus, $\max_{1\le i\le n}d_i^2\le (1+\frac{1}{r\pi_{min}})\sbf a^T\sbf aM_x$. 
For $\sum_{i=1}^n d_i^2$, we have
  \begin{eqnarray}\label{eqn:sumd} 
   \sum_{i=1}^n d_i^2 = \sum_{i=1}^n (1+\frac{1}{r\pi_i})\sbf a^T \mbf x_i
  \sbf a^T \mbf x_i  \ge  (1+\frac{1}{r\pi_{max}}) \sbf a^T \mbf{X}^T \mbf{X}\sbf a \ge  (n+\frac{n}{r\pi_{max}})\lambda_{min}\sbf a^T \sbf a,
  \end{eqnarray}
 where  $\lambda_{min}$ is the minimum eigenvalue of $\mbf{X}^T\mbf{X}/n$.   Combining (\ref{eqn:dmax}) and ($\ref{eqn:sumd}$), we have
\begin{eqnarray}
   \lim_{n\rightarrow\infty}\frac{\max_{1\le i\le n} d_i^2}{\sum_{i=1}^n d_i^2} \le\label{eqn:useMx} \lim_{n\rightarrow\infty}\frac{(1+\frac{1}{r\pi_{min}})M_x} {(n+\frac{n}{r\pi_{max}})\lambda_{min}}
   \le\frac{M_x}{\lambda_{min}}\lim_{n\rightarrow\infty}\frac{1+r\pi_{min}} {(nr\pi_{min}+\frac{n\pi_{min}}{\pi_{max}})}=0, 
  \end{eqnarray}
  where the last equality is obtained since condition (A2) 
  implies $nr\pi_{min}\rightarrow\infty$ as $n\rightarrow\infty$.
 Since
 \begin{equation*}
 \sum_{i=1}^n \vv(d_i\zeta_i) =\sigma^2 \sum_{i=1}^n (\sbf a^T\mbf x_i )^2(1+\frac{1}{r\pi_i})= \sigma^2\sbf a^T\mbf{X}^T(\mbf{I}_p+\mbf{\Omega})\mbf{X}\sbf a, 
 \end{equation*}
 by the Cramer-Wold device,  the proof is thus complete.

\end{proof}

In the following statement and proof of Lemma~\ref{lem:cond:Us}, as well as in the proof of Theorem~\ref{thm:beta_tilde-beta} below, we use $A|B$ to denote 
   random variable $A$ given random variable $B$. 
   
 \begin{lem}\label{lem:cond:Us}
 Given any nonzero constant vector $\sbf b\in \mathbb{R}^p$, 
 as $n\to\infty$
 we have
 \begin{eqnarray}
  (\sbf b^T\mbf{\Sigma}_{0}\sbf b)^{-\frac 12}\sbf b^T(\mbf{X}^T\mbf{X})^{-1}\mbf{X}^T\mbf{\Omega} \mbf{U}\sbf\varepsilon|\sum_{i=1}^nU_i=r&\stackrel{d}{\rightarrow}&
  N(0,1),
  \end{eqnarray}
  where $\sbf\Omega$, $\mbf U$, and $\sbf\Sigma_0$ are defined  in Lemma~\ref{lem:Us}.
\end{lem}
\begin{proof}
 
 For $i=1,\ldots,n$, we have 
 \begin{equation}
 \text{Cov}(\sbf b^T(\mbf{X}^T\mbf{X})^{-1}\mbf{X}^T\mbf{\Omega} U_i\varepsilon_i, \sum_{i=1}^nU_i)=
 \sum_{i=1}^n\sbf b^T(\mbf{X}^T\mbf{X})^{-1}\mbf{X}^T\mbf{\Omega}\text{Cov}( U_i\varepsilon_i, U_i)=0, 
 \end{equation}
 and thus we have
\begin{equation}
  \text{Cov}(\sbf b^T(\mbf{X}^T\mbf{X})^{-1}\mbf{X}^T\mbf{\Omega} \mbf{U}\sbf\varepsilon, \sum_{i=1}^nU_i) =\nonumber  0. 
\end{equation}
 By Lemma~\ref{lem:Us}, we have 
 \begin{eqnarray}
 \left(
  \begin{matrix}
  (\sbf b^T\mbf{\Sigma}_{0}\sbf b)^{-\frac 12}\sbf b^T(\mbf{X}^T\mbf{X})^{-1}\mbf{X}^T\mbf{\Omega} \mbf{U}\sbf\varepsilon\\
  \frac{1}{\sqrt{r}}(\sum_{i=1}^nU_i-r)
  \end{matrix}
  \right)
  &\stackrel{d}{\rightarrow}&\label{eqn:Ue:U:asym}
  \textbf{N}\left(
   \left(
  \begin{matrix}
  0\\
  0
  \end{matrix}
  \right),
  \left(
  \begin{matrix}
   1
   & 0\\
    0& 1
  \end{matrix}
  \right)
  \right). 
 \end{eqnarray}
Furthermore, we have 
  \begin{eqnarray}
 (\sbf b^T\mbf{\Sigma}_{0}\sbf b)^{-\frac 12} \sbf b^T(\mbf{X}^T\mbf{X})^{-1}\mbf{X}^T\mbf{\Omega} \mbf{U}\sbf\varepsilon|\sum_{i=1}^nU_i=r&\stackrel{d}{\rightarrow}&\label{eqn:dist:joint:u:sumu} N(0,1),
  \end{eqnarray}
  provided we can show the convergence of conditional distributions is the uniform equicontinuity of conditional characteristic functions \citep{steck1957condlimit}, as we do below. 

  Here, for the ease of notation, we  define $Q_n=\sbf b^T(\mbf{X}^T\mbf{X})^{-1}\mbf{X}^T\mbf{\Omega} \mbf{U}\sbf\varepsilon$, $L_n = \frac{1}{\sqrt{r}}(\sum_{i=1}^nU_i-r)$, and $s_n^2=\sbf b^T\mbf{\Sigma}_{0}\sbf b$. 
  Let 
  $$
  \psi_{n}(t_n; t)=\E(\exp(\text{i}tQ_n|\sum_{i=1}^n U_n=t_n))  ,
  $$
  where $\text{i}$ denotes the imaginary unit. 
  Hence, we aim to show the uniform equicontinuity of $\psi_{n}(t_n; t)$. 
  When $L_n=l_n$, $\sum_{i=1}^nU_i=r+\sqrt{r}l_n$; when $L_n=l_n+h$, $\sum_{i=1}^nU_i=r+\sqrt{r}l_n+\sqrt{r}h$. 
  Note~that 
  $$
  (Q_n|L_n=l_n+h)=\sbf b^T(\mbf{X}^T\mbf{X})^{-1}\mbf{X}^T\mbf{\Omega} (\mbf{M}+\mbf{R})\sbf\varepsilon,
  $$
  where $\mbf{M}=\text{diag}\{M_i\}_{i=1}^n$,  $(M_1,\ldots, M_n)\sim \text{Mult}(h\sqrt{r}, (\pi_1,\ldots, \pi_n))$, $\mbf{R}=\text{diag}\{R_i\}_{i=1}^n$, and $(R_1,\ldots, R_n)\sim \text{Mult}(r+\sqrt{r}l_n, (\pi_1,\ldots, \pi_n))$. 
Thus, we have that
  \begin{eqnarray}
   \nonumber
   && \hspace{-15mm} |\psi_n(l_n+h;t)-\psi_n(l_n;t)|\\
   &=&         |\E\left(\exp(\text{i}t/s_n\sbf b^T(\mbf{X}^T\mbf{X})^{-1}\mbf{X}^T\mbf{\Omega}(\mbf{M}+\mbf{R})\sbf\varepsilon)\right)-\E\left(\exp(\text{i}t/s_n\sbf b^T(\mbf{X}^T\mbf{X})^{-1}\mbf{X}^T\mbf{\Omega} \mbf{R}\sbf\varepsilon)\right)|\\
   &\le&\label{eqn:charac:1} \E\left(|\exp(\text{i}t/s_n\sbf b^T(\mbf{X}^T\mbf{X})^{-1}\mbf{X}^T\mbf{\Omega}(\mbf{M}+\mbf{R})\sbf\varepsilon)-\exp(\text{i}t/s_n\sbf b^T(\mbf{X}^T\mbf{X})^{-1}\mbf{X}^T\mbf{\Omega} \mbf{R}\sbf\varepsilon|)\right)\\
   &\le&\label{eqn:charac:2} |t/s_n|\E\left(|\sbf b^T(\mbf{X}^T\mbf{X})^{-1}\mbf{X}^T\mbf{\Omega}(\mbf{M}+\mbf{R})\sbf\varepsilon-\sbf b^T(\mbf{X}^T\mbf{X})^{-1}\mbf{X}^T\mbf{\Omega} \mbf{R}\sbf\varepsilon|\right)\\
   &=&\label{eqn:charac:3}(t/s_n)\E\left(|\sbf b^T(\mbf{X}^T\mbf{X})^{-1}\mbf{X}^T\mbf{\Omega}\mbf{M}\sbf\varepsilon|\right)\\
   &\rightarrow&\label{eqn:char_goto0}0\quad\quad\text{ as } h\to 0,  
  \end{eqnarray}  
  where ($\ref{eqn:charac:1}$) is by Jensen's inequality, and ($\ref{eqn:charac:2}$) is by the fact that $|e^{\text{i}a}-e^{\text{i}b}|=\sqrt{2(1-cos(\frac{a-b}{2}))}=2|\sin(\frac{a-b}{2})|\le |a-b|$, for any $a$, $b$.
  Thus, the uniform equicontinuity of conditional characteristic function is verified, and the proof is complete. 

\end{proof}

\textbf{Remark.}
The proof of Lemma~\ref{lem:cond:Us} is a simplified version of the proof of Theorem 2.1 in \citet{cltformult1975}. 

\textbf{Remark.}
The key difference between Lemma~\ref{lem:Us} and Lemma~\ref{lem:cond:Us} is that we consider a conditional distribution in Lemma~\ref{lem:cond:Us}, whereas we consider an unconditional distribution in Lemma~\ref{lem:Us}. 

We will also need the following lemma, the proof of which can be found in Section 6.4 of \citet{sheldon2006first}.

\begin{lem}
\label{lem:mult:pois}
If independent random variables $U_i\sim \text{Poisson}(\lambda_i)$, $i=1,\ldots, n$, then 
\begin{equation}\nonumber
(U_1, \ldots, U_n)|\sum_{i=1}^nU_i=r \sim \text{Mult}\left(r, \left\{\frac{ \lambda_i}{\sum_{i=1}^n \lambda_i}\right\}_{i=1}^n \right).
\end{equation} 

\end{lem}

\subsubsection{Main Part of the Proof of Theorem~\ref{thm:beta_tilde-beta}}

We first verify the condition that $\delta>0$ in Lemma \ref{lem:1}.
To do this, we derive the magnitude of $\delta$ in Eqn.~(\ref{eqn:delta:lev}), under the conditions of Lemma \ref{lem:1}.
Note~that
\begin{equation*}
(\mbf{X}^T\mbf{X})^{-1}\mbf{X}^T(\mbf{W}-\mbf{I})\mbf{X}= (\mbf{X}^T\mbf{X}/n)^{-1}\mbf{X}^T(\mbf{W}-\mbf{I})\mbf{X}/n,
\end{equation*}
where the order of $(\mbf{X}^T\mbf{X}/n)^{-1}$ is $O(1)$ by Condition (A1). 
Thus, the order of $(\mbf{X}^T\mbf{X})^{-1}\mbf{X}^T(\mbf{W}-\mbf{I})\mbf{X}$ depends on that of $\mbf{X}^T(\mbf{W}-\mbf{I})\mbf{X}/n$. 
We next derive the order of the $(s,t)^{th}$ element of $\mbf{X}^T(\mbf{W}-\mbf{I})\mbf{X}/n$, i.e., of $\frac{1}{n}\sum_{i=1}^n x_{si}x_{it}(W_{i}-1)$.  
To do so, we have 
\begin{eqnarray}
\text{Var}\left(\frac{\sum_{i=1}^n x_{si}x_{it}(W_{i}-1)}{n}\right)
 &=&\nonumber
 \frac{1}{n^2} \text{Var}\left(\sum_{i=1}^n x_{si}x_{it}(W_{i}-1)\right)\\
 &=&\nonumber\frac{1}{n^2} \left(\sum_{i=1}^n (x_{si}x_{it})^2 \frac{1-\pi_i}{r\pi_i}-2\sum_{i<j}x_{si}x_{it}x_{sj}x_{tj}\frac{1}{r} \right)\\
 &=&\nonumber\frac{1}{rn^2} \left[\sum_{i=1}^n (x_{si}x_{it})^2 \frac{1-\pi_i}{\pi_i}-\left(\left(\sum_{i=1}^n x_{si}x_{it}\right)^2-\sum_{i=1}^n (x_{si}x_{it})^2\right) \right]\\
 &=&\nonumber\frac{1}{r} \left[\sum_{i=1}^n \frac{(x_{si}x_{it})^2}{n^2\pi_i}-\left(\sum_{i=1}^n \frac{x_{si}x_{it}}{n}\right)^2 \right]\\
 &=&\label{eqn:delta:beta} O\left(\frac{1}{rn^2}\sum_{i=1}^n \frac{1}{\pi_i}\right). 
\end{eqnarray}
Combining the facts that $n^2\le \sum_{i=1}^n \frac{1}{\pi_i}\le \frac{n}{\pi_{min}}$ 
and $0<(2-\gamma_0-\alpha) \le\delta$ in Eqn.~(\ref{eqn:delta:lev}), we thus verify that the assumption in Lemma~\ref{lem:1} holds. 

Subtracting $\sbf\beta_0$ from both sides of Eqn.~(\ref{order}) in Lemma~\ref{lem:1}, we get 
\begin{equation}\label{order:thm:beta}
 \tilde{\sbf\beta}-\sbf\beta_0 =(\mbf{X}^T\mbf{X})^{-1}\mbf{X}^T\mbf{W}\sbf e+ \hat{\sbf\beta}_{OLS}-\sbf\beta_0+O_p\left(\frac{1}{r^{\delta}}\right), 
\end{equation}
where $\sbf e=\mbf Y- \mbf{X}\hat{\sbf\beta}_{OLS}$. 
Since $\vv(\hat{\sbf\beta}_{OLS}-\sbf\beta_0) = O\left(\frac{1}{n}\right)$, 
we have $\hat{\sbf\beta}_{OLS}-\sbf\beta_0 = O_p\left(\frac{1}{n^{2-\alpha-\delta}}\right)$. 
Thus, both $\hat{\sbf\beta}_{OLS}-\sbf\beta_0$ and the residual term in the right hand side of (\ref{order:thm:beta}) are negligible. 
Hence, the asymptotic distribution of $\tilde{\sbf\beta}-\sbf\beta_0$ is equivalent to that of $(\mbf{X}^T\mbf{X})^{-1}\mbf{X}^T\mbf{W}\sbf{e}$. 

Thus, for the rest of the proof, we derive the asymptotic normality of $(\mbf{X}^T\mbf{X})^{-1}\mbf{X}^T\mbf{W}\sbf{e}$. 
Note~that  
\begin{equation}\label{e-epsilon-transform}
  (\mbf{X}^T\mbf{X})^{-1}\mbf{X}^T\mbf{W}\sbf e  =  (\mbf{X}^T\mbf{X})^{-1}\mbf{X}^T\mbf{W}\sbf\varepsilon + (\mbf{X}^T\mbf{X})^{-1}\mbf{X}^T\mbf{W}(\sbf e-\sbf\varepsilon),
\end{equation}
where $\sbf\varepsilon$ is the random noise in Model (\ref{linreg-matrix}). 
We will show that the order of $(\mbf{X}^T\mbf{X})^{-1}\mbf{X}^T\mbf{W}(\sbf e-\sbf\varepsilon)$ is bounded by calculating the variances of $s^{th}$ element of $\mbf{X}^T\mbf{W}(\sbf e-\sbf\varepsilon)/n$.  
We have 
\begin{multline}
 \text{Var}\left(\frac{\sum_{i=1}^n x_{si}W_{i}(e_i-\varepsilon_i)}{n}\right)\\
 = \frac{1}{n^2} \left(\sum_{i=1}^nx_{si}^2\vv(W_{i}(e_i-\varepsilon_i))+2\sum_{i<j}x_{si}x_{sj}\text{Cov}\left[W_{i}(e_i-\varepsilon_i), W_{j}(e_j-\varepsilon_j)\right] \right)\label{e-epsilon-order}.
\end{multline}

\noindent
Now, we analyze the two terms on the right hand side of Eqn.~(\ref{e-epsilon-order}).
For the first term, we have
\begin{eqnarray}
 \sum_{i=1}^n \vv(W_{i}(e_i-\varepsilon_i))&= &\nonumber\sum_{i=1}^n \E(W_{i}^2(e_i-\varepsilon_i)^2)=\sum_{i=1}^n  \vv(W_{i})\vv(e_i-\varepsilon_i)+(\E W_{i})^2\vv(e_i-\varepsilon_i)\\
 &= &\label{e_epsilon_order_1}\sum_{i=1}^n \frac{1-\pi_i}{r\pi_i}h_{ii}\sigma^2+h_{ii}\sigma^2  = O_p\left( \frac{1}{r\pi_{min}} \right), 
\end{eqnarray}
where the last equality holds since  $\sum_{i=1}^n h_{ii} =p$. 
For the second term, we have
\begin{eqnarray}
\sum_{i<j} \text{Cov}\left(W_{i}(e_i-\varepsilon_i), W_{j}(e_j-\varepsilon_j)\right)&= &\nonumber\sum_{i<j} \E(W_{i}W_{j}(e_i-\varepsilon_i)(e_j-\varepsilon_j))\\
 &= &\label{e_epsilon_order_2}\sum_{i<j}\E(W_{i}W_{j})\E((e_i-\varepsilon_i)(e_j-\varepsilon_j))= O_p\left(\frac{n}{r}\right). 
\end{eqnarray}
Substituting (\ref{e_epsilon_order_1}) and (\ref{e_epsilon_order_2}) into (\ref{e-epsilon-order}), 
we have that
\begin{equation}
\label{e-epsilon-order-final}
\text{Var}\left(\frac{\sum_{i=1}^n x_{si}W_{i}(e_i-\varepsilon_i)}{n}\right)= O_p\left(\frac{1}{n^2r\pi_{min}}\right).
\end{equation}
Combining (\ref{e-epsilon-transform}) and (\ref{e-epsilon-order-final}), 
we aim to show that $(\mbf{X}^T\mbf{X})^{-1}\mbf{X}^T\mbf{W}(\sbf e-\sbf\epsilon)$ is of higher order than $(\mbf{X}^T\mbf{X})^{-1}\mbf{X}^T\mbf{W}\sbf \varepsilon $. 
Thus, if we establish the asymptotic normality of $(\mbf{X}^T\mbf{X})^{-1}\mbf{X}^T\mbf{W}\sbf\varepsilon$, then the asymptotic normality of $ \tilde{\sbf\beta}-\sbf\beta_0$ in Eqn.~(\ref{order:thm:beta}) will follow directly. 

Note that $\mbf{W}$ can be written as $\mbf{W}=\mbf{\Omega} \mbf{K}$. 
By Lemma~\ref{lem:mult:pois}, it follows that $(K_1,\ldots, K_n)$ and $[(U_1, \ldots, U_n)|\sum_{i=1}^nU_i=r]$ are identically distributed.
Hence, 
$$ 
(\mbf{X}^T\mbf{X})^{-1}\mbf{X}^T\mbf{W}\sbf\varepsilon
\quad\mbox{and}\quad
(\mbf{X}^T\mbf{X})^{-1}\mbf{X}^T\mbf{\Omega} \mbf{U}\sbf\varepsilon|\sum_{i=1}^n U_i 
$$ 
are identically distributed. 
Thus, Lemma~\ref{lem:cond:Us} can be applied, and the asymptotic normality is obtained using the Cramer-Wold device.

Finally, combining Eqn.~(\ref{order:thm:beta}), 
Lemma~\ref{lem:Us}, and Lemma~\ref{lem:cond:Us}, we have that
\begin{eqnarray}
  \mbf{\Sigma}_0^{-\frac 12}(\sbf{\tilde\beta}-\sbf\beta_0)&\stackrel{d}{\rightarrow}&\label{eqn:thm:beta_tilde-beta_app}\textbf{N}(\sbf 0,\mbf{I}_p), \quad\text{as } n\to\infty, 
\end{eqnarray}
where $\mbf{\Sigma}_0=\sigma^2(\mbf{X}^T\mbf{X})^{-1}\mbf{X}^T(\mbf{I}_p+\mbf{\Omega})\mbf{X}(\mbf{X}^T\mbf{X})^{-1}$.

\subsection{Proof of Proposition \ref{thm:beta:pi1}}
\label{sxn:proof_of_prop1}

By Theorem~\ref{thm:beta_tilde-beta}, 
the asymptotic squared bias for $\tilde{\sbf\beta}$ is $0$.  
By the definition of AMSE in Eqn.~(\ref{eqn:mse20}), 
$AMSE(\tilde{\sbf\beta}; {\sbf\beta}_{0})=tr(Avar(\tilde{\sbf\beta}))$, i.e., the expression given in Eqn.~(\ref{eqn:mse2}). 
We consider minimizing $AMSE(\tilde{\sbf\beta}; {\sbf\beta}_{0})$ as a function of $\{\pi_i\}_{i=1}^n$.
It is straightforward to employ the method of Lagrange multipliers to find the minimizer of the right-hand side of Eqn.~(\ref{eqn:mse2}), subject to the constraint $\sum_{i=1}^n\pi_i=1$.
If we do this, then we let 
$$
L(\pi_1, \ldots, \pi_n)=tr(Avar(\tilde{\sbf\beta}))+\lambda(\sum_{i=1}^n\pi_i-1). 
$$
Then, we can solve $\partial{L/\partial \pi_i}=0$, $i=1,\ldots,n$, for the optimal sampling probabilities.

The proofs of Propositions \ref{thm:beta:pi3}--\ref{thm:beta_hat:pi2} all follow in a manner similar to that of Proposition \ref{thm:beta:pi1}, and thus they will be omitted.

\subsection{Proof of Theorem~\ref{coro:diverg:p}}
\label{sec:proof-of-corollaryThm}

The proof of Theorem~\ref{coro:diverg:p}, in which we allow the number of predictors $p$ to diverge, is readily derived from that of Theorem~\ref{thm:beta_tilde-beta}. 
Without loss of generality, we let $\sbf a\in\mathbb{R}^p$ such that $\|\sbf a\|^2=\sbf a^T\sbf a=1$.
By combining Eqns.~(\ref{order:thm:beta}) and~(\ref{e-epsilon-transform}), it follows that
\begin{eqnarray}
 \sbf a^T(\tilde{\sbf\beta}-\sbf\beta_0)  = \nonumber\sbf a^T(\mbf{X}^T\mbf{X})^{-1}\mbf{X}^T\mbf{W}\sbf \varepsilon+\sbf a^T(\mbf{X}^T\mbf{X})^{-1}\mbf{X}^T\mbf{W}\sbf(\sbf e- \sbf\varepsilon)+ \sbf a^T(\hat{\sbf\beta}_{OLS}-\sbf\beta_0)+O_p\left(\frac{1}{r^{\delta}}\right)
\end{eqnarray}

By results in \citep{huber1973robust,yohai1979asymptotic,portnoy1984asymptoticI,portnoy1985asymptoticII}, we note that $\|\sbf a\|^2=1$, and $\sbf a^T(\hat{\sbf\beta}_{OLS}-\sbf\beta_0)=O_p(1/\sqrt{n})$, which is of the highest order. 
Further, by a similar argument in Theorem 1 (from (\ref{e-epsilon-transform}) to (\ref{e-epsilon-order-final})) we have that $\sbf a^T(\mbf{X}^T\mbf{X})^{-1}\mbf{X}^T\mbf{W}\sbf(\sbf e- \epsilon)$ is of higher order than $\sbf a^T(\mbf{X}^T\mbf{X})^{-1}\mbf{X}^T\mbf{W}\sbf \varepsilon$. 
To prove Theorem~\ref{coro:diverg:p}, it suffices to establish the asymptotic normality of $\sbf a^T(\mbf{X}^T\mbf{X})^{-1}\mbf{X}^T\mbf{W}\sbf \varepsilon$. 
This follows from applying Condition (B2) to Lemma~\ref{lem:Us} and  by noting that $M_x=O(p)$ in (\ref{eqn:useMx}).

\subsection{Proof of Theorem~\ref{thm:beta_tilde-beta_hat}}
\label{sec:proof:thm:beta_tilde-betahat}

Given data $\{\mbf{X},\mbf{Y}\}$, we first determine the value of $\delta$ in Eqn.~(\ref{eqn:delta:lev}) in order to use Lemma \ref{lem:1}. 
Since $\|\mbf x_i\|<\infty$, where $\mbf x_i$ is the $i^{th}$ row of $\mbf{X}$, each element of $\mbf{X}^T\mbf{X}$ is a fixed matrix and is finite in norm. 
Since the $(s,t)^{th}$ element of $\mbf{X}^T(\mbf{W}-\mbf{I})\mbf{X}$ is $\sum_{i=1}^n x_{is}x_{it}(W_{i}-1)$, it follows that 
\begin{eqnarray}
 \text{Var}(\sum_{i=1}^n x_{is}x_{it}(W_{i}-1))
  = \label{eqn:delta:betahat} \frac{1}{r}\left(\sum_{i=1}^n (x_{is}x_{it})^2 \frac{1-\pi_i}{\pi_i}-2\sum_{i<j}x_{is}x_{it}x_{js}x_{tj} \right)
  = O_p\left(\frac{1}{r}\right), 
\end{eqnarray}
i.e., $\delta=1$ in Eqn.~(\ref{eqn:delta:lev}). 

Next, note that $\mbf{K}$ can be written as $\mbf{K}=\sum_{j=1}^r \mbf{K}^{(j)}$, where $\mbf{K}^{(j)}=\text{Diag}\{K_i^{(j)}\}_{i=1}^n$,  
and where
$(K_1^{(j)}, \ldots, K_n^{(j)})\stackrel{iid}{\sim}$ Mult$(1, \{\pi_i\})_{i=1}^n$, for $j=1,\ldots, n$.
Combining Eqn.~(\ref{order}) in Lemma~\ref{lem:1} and Eqn.~(\ref{eqn:delta:betahat}), we can show that 
\begin{eqnarray}
  \tilde{\sbf{\beta}}-\hat{\sbf{\beta}}_{OLS} &=&\nonumber (\mbf{X}^T\mbf{X})^{-1}\mbf{X}^T\mbf{W}\sbf e+O_p(1/r)\\
  &=&\nonumber (\mbf{X}^T\mbf{X})^{-1}\sum_{j=1}^r \mbf{X}^T\mbf{\Omega} \mbf{K}^{(j)}\sbf e+O_p(1/r)  .  
\end{eqnarray}
Given this, 
we can use the Cramer-Wold device to establish the asymptotic normality of 
$$ 
(\mbf{X}^T\mbf{X})^{-1}\sum_{j=1}^r \mbf{X}^T\mbf{\Omega} \mbf{K}^{(j)}\sbf e.
$$ 
To do so, for any constant vector $\sbf b\in \mathbb{R}^p$ such that $\sbf b\ne \sbf 0,$ we consider $\sum_{j=1}^r \sbf b^T(\mbf{X}^T\mbf{X})^{-1}\mbf{X}^T\mbf{\Omega} \mbf{K}^{(j)}\sbf e$.
This is a summation of $r$ independent random variables. 
Since the elements in $\sbf X$ and $\sbf e$ are fixed numbers, finite in norm, and $\pi_i>0$, the Noether condition in Hajek-Sidek CLT is satisfied. 
 
Without loss of generality, we have
\begin{eqnarray}
   \vv(\sbf b^T(\mbf{X}^T\mbf{X})^{-1}\mbf{X}^T\mbf{\Omega} \mbf{K}^{(1)}\sbf e)&=&\nonumber 
   \vv(\sum_{i=1}^n \sbf b^T(\mbf{X}^T\mbf{X})^{-1}\mbf x_i\frac{1}{r\pi_i} K_i^{(1)}e_i)\\
   &=&\nonumber 
   \sum_{i=1}^n( \sbf a^T\mbf x_ie_i\frac{1-\pi_i}{r\pi_i}e_i\mbf x_i^T\sbf a)-2\sum_{i\le j}\sbf a^T\mbf x_i e_i\frac{1}{r}e_j \mbf x_j^T \sbf a\\
   &=&\nonumber \frac{1}{r}\sbf a^T\left(\sum_{i=1}^n\frac{e_i^2}{\pi_i}\mbf x_i\mbf x_i^T\right)\sbf a-
   \frac{1}{r}\sbf a^T\left(\sum_{i=1}^n\mbf x_ie_i^2\mbf x_i^T+2\sum_{i<j}\mbf x_ie_ie_j\mbf x_j^T\right)\sbf a\\
   &=&\nonumber \frac{1}{r}\sbf a^T\left(\sum_{i=1}^n\frac{e_i^2}{\pi_i}\mbf x_i\mbf x_i^T\right)\sbf a-
   \frac{1}{r}\sbf a^T \mbf{X}^T\sbf e\sbf e^T \mbf{X}\sbf a\\
    &=&\label{eqn:xe_0} \frac{1}{r}\sbf a^T\left(\sum_{i=1}^n\frac{e_i^2}{\pi_i}\mbf x_i\mbf x_i^T\right)\sbf a,
\end{eqnarray}
where $\sbf a =(\mbf{X}^T\mbf{X})^{-1}\sbf b$, and where Eqn.~(\ref{eqn:xe_0}) follows since $\mbf{X}^T\sbf e=\sbf 0$.  
By the Lindeberg-L$\Acute{e}$vy CLT, we have that
$$
\sbf b^T(\mbf{X}^T\mbf{X})^{-1}\sum_{j=1}^r \mbf{X}^T\mbf{\Omega} \mbf{K}^{(j)}\sbf e\stackrel{d}{\rightarrow}N(\sbf 0,\sbf b^T \mbf{\Sigma}_{c}\sbf b)  ,  
$$
where 
$\mbf{\Sigma}_{c}=(\mbf{X}^T\mbf{X})^{-1}\mbf{\Sigma}_e(\mbf{X}^T\mbf{X})^{-1}$ 
and 
$\mbf{\Sigma}_e=\frac{1}{r}\sum_{i=1}^n \frac{e_i^2}{\pi_i}\mbf{x}_i\mbf{x}_i^T$.
Thus, by the Cramer-Wold device, Theorem~\ref{thm:beta_tilde-beta_hat} follows.

\end{document}